\documentclass[11pt,reqno]{article}
\usepackage{a4wide}
\usepackage{mathrsfs}
\usepackage{amsfonts}
\usepackage{amsmath}
\usepackage{stmaryrd}
\usepackage{amssymb}
\usepackage{amsthm}
\usepackage{mathrsfs}
\usepackage{url}
\usepackage{amsfonts}
\usepackage{amscd}
\usepackage{indentfirst}
\usepackage{enumerate}
\usepackage{amsmath,amsfonts,amssymb,amsthm}
\usepackage{amsmath,amssymb,amsthm,amscd}
\usepackage{graphicx,mathrsfs}
\usepackage{appendix}
\usepackage{authblk}
\usepackage[numbers,sort&compress]{natbib}
\usepackage{tikz}
\usepackage{xcolor}
\usepackage[colorlinks, linkcolor=blue, citecolor=blue]{hyperref}
 \usepackage{verbatim}

\newtheorem{teo}{Theorem}[section]
\newtheorem{lemma}[teo]{Lemma}
\newtheorem{prop}[teo]{Proposition}
\newtheorem{cor}[teo]{Corollary}
\newtheorem{rem}[teo]{Remark}
\newtheorem{defi}[teo]{Definition}

\numberwithin{equation}{section}
\newcommand\N{\mathbb N}

\newcommand\R{\mathbb R}

\renewcommand\S{\mathbb S}

\newcommand\mbb\mathbb
\newcommand\mbf\mathbf
\newcommand\mcal\mathcal
\newcommand\mfrak\mathfrak
\newcommand\mrm\mathrm
\newcommand\msf\mathsf
\renewcommand\a\alpha
\renewcommand\b\beta
\newcommand\g\gamma
\newcommand\G\Gamma
\renewcommand\d\delta
\newcommand\D\Delta
\newcommand\e\varepsilon
\newcommand\z\zeta
\renewcommand\t\theta
\newcommand\Th\Theta
\newcommand\la\lambda
\newcommand\La\Lambda
\newcommand\s\sigma
\newcommand\si\varsigma
\newcommand\Si\Sigma
\newcommand\ups\upsilon
\newcommand\U\Upsilon
\newcommand\ph\varphi
\renewcommand\o\omega
\renewcommand\O\Omega
\newcommand\wt\widetilde
\newcommand\wh\widehat
\newcommand\ol\overline
\newcommand\ul\underline
\newcommand\mr\mathring
\newcommand\ub\underbrace
\newcommand\pa\partial
\newcommand\n\nabla
\newcommand\fa\forall
\newcommand\ex\exists
\newcommand\es\emptyset
\newcommand\wk\rightharpoonup
\newcommand\inc\hookrightarrow
\newcommand\linf\varliminf
\newcommand\lsup\varlimsup
\newcommand\os\overset
\newcommand\us\underset
\newcommand\sr\stackrel
\newcommand\Ot\Leftarrow
\newcommand\To\Rightarrow
\newcommand\map\mapsto
\newcommand\ot\leftarrow
\newcommand\lot\longleftarrow
\newcommand\lto\longrightarrow
\newcommand\tot\leftrightarrow
\newcommand\ltot\longleftrightarrow
\newcommand\sm\backslash
\renewcommand\Cup\bigcup
\renewcommand\Cap\bigcap
\newcommand\sub\subset
\newcommand\Sub\Subset
\newcommand\sne\subsetneq
\newcommand\bus\supset
\newcommand\Bus\Supset

\newcommand\eq\equiv
\newcommand\ox\otimes
\newcommand\Ox\bigotimes
\newcommand\pl\oplus
\newcommand\Pl\bigoplus
\newcommand\x\times
\renewcommand\c\circ
\newcommand\q\quad
\renewcommand\l\left
\renewcommand\r\right
\newcommand\fr\frac

\allowdisplaybreaks
\newcommand{\Ma}{\color{blue}}
\newcommand{\Lau}{\color{magenta}}
\newcommand{\Ying}{\color{red}}

\def\sideremark#1{\ifvmode\leavevmode\fi\vadjust{\vbox to0pt{\vss
			\hbox to 0pt{\hskip\hsize\hskip1em
				\vbox{\hsize2.1cm\tiny\raggedright\pretolerance10000
					\noindent #1\hfill}\hss}\vbox to15pt{\vfil}\vss}}}%

\title{Qualitative properties of eigenfunctions in domains with small holes}
\author[1]{Laura Abatangelo}
\author[2]{Massimo Grossi}
\author[3]{Ying Li}
\affil[1]{Politecnico di Milano, Dipartimento di Matematica}
\affil[2]{Sapienza Università di Roma, Dipartimento di Scienze di base ed applicate per l'Ingegneria}
\affil[3]{Central China
Normal University, School of Mathematics and Statistics}
\date{July 2026}

\begin{document}
	
	\maketitle
\begin{abstract}
  In this paper we study qualitative properties of the eigenvalues and eigenfunctions of $-\Delta$ with Dirichlet boundary condition in a smooth bounded domain $\Omega$ with a small circular hole. In the literature, this is known as a "singular perturbation", in contrast with the "regular perturbation" case,
considered for example by Micheletti and Uhlenbeck in \cite{Micheletti} and \cite{Uhlenbeck}.

Denoting by $\O_\e:=\O\setminus B(P,\e)$ where $B(P,\e)$ is the ball centered at $P$ and radius $\e$, for $P\in\O$ and $\e$ small enough we investigate 
\begin{itemize}
\item quantitative
estimates for the eigenfunctions of $-\Delta$ in $\O_\e$;
\item the simplicity of the eigenvalues of $-\Delta$ in $\O_\e$;
\item the behavior of nodal sets of the eigenfunctions of $-\Delta$ in $\O_\e$.
\end{itemize}

A key ingredient in our analysis consists of pointwise estimates on the so-called $u$-capacitary potential firstly introduced in \cite{afhl}.
\end{abstract}	
	\section{Introduction and main results}
	    	The aim of this paper is to study the asymptotic behavior of eigenvalues and eigenfunctions of the Dirichlet Laplacian in a bounded domain when a small hole is removed.

    Let $\Omega\subset\R^N$ be a smooth bounded domain, $N\geq2$. 
	We denote by $\lambda$ any  eigenvalue of the Dirichlet Laplacian on $\Omega$ and by $u$ any of the corresponding $L^2$-normalized eigenfunction, namely,
	\begin{equation}\label{eq:eigenvalueproblem}
		\begin{cases}
			-\Delta u=\lambda u~&\mbox{in}\ \Omega,\\
			u=0~&\mbox{on}\ \partial\Omega.
		\end{cases}
	\end{equation}

    We focus on the case when we remove a small ball $B(P,\e)$ centered at $P\in\Omega$ of radius $\e$ from the domain $\O$. We then set the perturbed problem 
		\begin{equation}\label{eq:perturbedeigenvalueproblem}
			\begin{cases}
				-\Delta u_\e=\lambda_\e u_\e~&\mbox{in}\ \O_\e,\\
				u_\e=0~&\mbox{on}\ \partial\O_\e,
			\end{cases}
		\end{equation}
    where
    \begin{equation}\label{ome}
        \O_\e:=\O\setminus B(P,\e).
    \end{equation}
Analogously to the previous case, we let $(\lambda_\varepsilon, u_\varepsilon)$ be the pair playing the same role as $(\lambda, u)$.

    It is well known that the eigenvalues of $-\Delta$ in $D\subset\R^N$ are given by a sequence satisfying
    $$0<\la_1(D)<\la_2(D)\le..\le \la_k(D)\le..$$
They can be characterized via the Rayleigh quotient
\[
\lambda_k(D) = \min_{\substack{V \subset H_0^1(D) \\ \dim V = k}} \max_{u \in V \setminus \{0\}} 
\frac{\int_D |\nabla u|^2\,dx}{\int_D u^2\,dx}.
\]

 A natural question is how these eigenvalues (and the corresponding eigenfunctions) change when $D=\O_\e$ and $\e$ is small. We address to this case as the {\em singular perturbation}, differently from the most studied "regular perturbation", where only small deformation of $\partial\Omega$ are allowed.

Let us make a brief history of this second case.
Early systematic studies go back to the work of Courant and Hilbert, where qualitative stability of eigenvalues under domain inclusion was already observed.

The study of the dependence of eigenvalues on domain perturbations has a long history. A major milestone is the work of Hadamard, who introduced a formal way to compute derivatives of eigenvalues under smooth deformations of the boundary. Consider a family of domains $\Omega_t=\O+tV$ for a smooth vector field $V$ and $t$ small. Then one obtains the celebrated Hadamard variation formula: for a simple eigenvalue $\la_k$ we have
\[
\frac{d}{dt}\bigg|_{t=0} \lambda_k(\Omega_t)
=
- \int_{\partial \Omega} \left(\frac{\partial u_k}{\partial \nu}\right)^2 (V \cdot \nu)\, d\sigma,
\]
where $u_k$ is the normalized eigenfunction associated to $\lambda_k(\Omega)$.


 Another fundamental contribution was given by Micheletti \cite{Micheletti} and Uhlenbeck\cite{Uhlenbeck} which proved that, for generic smooth domains, the eigenvalues of the Dirichlet Laplacian are simple, showing that multiple eigenvalues are unstable under small perturbations.

A different regime arises  when the perturbation is not smooth but {\em singular}, when for instance $\O_\e$ verifies \eqref{ome}.
In this setting, more general perturbations may also be considered. Instead of $B(P,\e)$, one may take $\e K$, where $K$ is an arbitrary domain in $\R^N$, or families of sets $K_\e \to K$, where $K$ can also be a set with positive $(N-1)$-dimensional Hausdorff measure (see, for instance, \cite{afhl}). 
In this context a crucial role is played by the {\em capacity} of a set $K$ in $\O$ whose definition is recalled below.

For a compact set $K\subset\O$ the capacity of $K$ in $\O$ is defined as 
    \begin{equation}\label{defCap2}
        \hbox{Cap}_\O(K)=\inf \left\{\int_\O |\nabla f|^2\ dx:\ f\in H^1_0 (\O) \ \hbox{and}\ f-\eta_K\in H_0^1(\O\setminus K) \right\},
    \end{equation}
    where $\eta_K$ is a fixed smooth function such that supp$\eta_K\subset\O$ and $\eta_K\equiv1$ in a neighborhood of $K$. It is easy to prove that the infimum is achieved by a function $V_K\in H_0^1(\O)$ such that $V_K-\eta_K\in H_0^1(\O\setminus K)$, so that $\hbox{Cap}_\O(K)=\int_\O |\nabla V_K|^2\ dx$.

 If the eigenvalue $\lambda(\O)$ is simple, for a family of compact sets $(K_\e)_{\e>0}$  concentrating to a compact set $K$ with $\hbox{Cap}_\O(K)=0$, then  \cite[Theorem 1.2]{Courtois1995}  provides that, for some explicit constant $\mu_K$,
    \begin{equation}\label{eq-afh-1.4}
        \lambda(\O\setminus K_\e)=\lambda(\O)+\mu_K\hbox{Cap}_\O(K_\e)+o(\hbox{Cap}_\O(K_\e)).
    \end{equation}
This type of result was also developed in various forms by Ozawa, Rauch and Taylor, and later refined by many authors in the context of potential theory and homogenization. 

Observe that, even in the simple case where $K_\e=B(P,\e)\to K=\{P\}$, the previous estimates are meaningful when the eigenvalue $\lambda$ is simple and the corresponding eigenfunction satisfies $u(P)\neq 0$.

The reason for these restrictions is twofold: indeed, if $\la$ is a simple eigenvalue of $-\Delta$ in $\O$, then there exists a family of eigenvalues $\la_\e$ of $-\Delta$ in $\O_\e$ such that $\la_\e \to \la$. This uniquely determines $\lambda_\varepsilon$.  Secondly, if $u(P)=0$, then the leading term in the expansions in \eqref{eq-afh-1.4} vanishes.

The case where $\lambda$ is simple and $u$ vanishes at $P$ was studied by Abatangelo, Felli, Hillairet, and L\'ena \cite{afhl} where introduced the $u$-capacity  $\hbox{Cap}_\O(K_\e,u)$ to obtain a sharp asymptotic expansion in case of simple eigenvalues. The definition, which will play an important role in this paper, is the following,
    \begin{defi}\label{intr-def-Cap}
    Given a function $u\in H_0^1(\O)$, consider the u-capacity of a compact set $K_\e\subset\O$, defined as in \eqref{defCap2} with  $\eta_K$ replaced by $u$.
    The infimum in \eqref{defCap2} is achieved by a unique function $V_{K_\e,u}\in H_0^1(\O)$, so that $\hbox{Cap}_\O(K_\e,u)=\int_\O|\nabla V_{K_\e,u}|^2\ dx$. We call $V_{K_\e,u}$ the capacitary potential associated with $u$ and $K_\e$.

    Furthermore, $V_{K_\e,u}$ is the unique weak solution of the Dirichlet problem 
    \begin{equation*}
		\begin{cases}
			\Delta V_{K_\e,u} =0 &\text{in }\Omega\setminus K_\e,\\
			V_{K_\e,u} =0 &\text{on }\partial \Omega,\\
			V_{K_\e,u} =u &\text{on }K_\e,
		\end{cases}
	\end{equation*}
    where, by weak solution, we mean that $V_{K_\e,u}\in H_0^1(\O)$, $u-V_{K_\e,u}\in H_0^1(\O\setminus K_\e)$ and 
    $$\int_{\O}\nabla V_{K_\e,u}\cdot\nabla\phi\ dx=0,\quad \forall\phi\in H^1_0(\O\setminus K_\e).
    $$
    \end{defi}
    As the paper \cite{afhl} shows, sharp estimates on $ \lambda(\O\setminus K_\e)$ derive from asymptotic behavior of $\hbox{Cap}_\O(K_\e,u)$.
Although this definition is applicable in more general contexts, we will only use it in the case where $K_\e=B(P,\e)$. In this setting, it becomes natural to introduce the notion of the order of vanishing of $u$ at $P$.
\begin{defi}\label{defvanishing}
        Let $k\in\N$, we say that  the vanishing order of $u$  at $P$ is equal to $k$, if $D^\beta u(P)=0$ for any $|\beta|\leq k-1$ and  there exists  $ \alpha=(\alpha_1,\cdots,\alpha_N)$ such that $|\alpha|=k$ and $D^\alpha u(P)\neq 0$.
    \end{defi}
In the case of eigenfunctions, the number $k$ in Definition \ref{defvanishing} is finite by the analyticity of the operator $-\Delta - \lambda$. In this framework we call $k$ the \emph{vanishing order} of $u$ at $P$.

  Thus, by \cite[Theorem 1.14]{afhl} and \cite[Theorem 1.7]{alm2} the expansion of a {\em simple} eigenvalue $\lambda_\e$ corresponding to   our problem \eqref{eq:perturbedeigenvalueproblem} becomes
  
	\begin{equation}\label{eq:eigenvaluesbehavior}
	\ 	\lambda_\e - \lambda = C \e^{N-2+2k}(1+o(1)) \qquad\qquad\qquad{\hbox{for}\  N\ge 3},
	\end{equation}
    \begin{equation}\label{eq:N=2eigenvaluesbehavior}
      \lambda_\e - \lambda =
        \begin{cases}
            \frac{C}{|\log\e|}(1+o(1)), \quad& if\ k=0,\\
            C\e^{2k}(1+o(1)), & if\  k\ge 1,
        \end{cases}
       \quad {\hbox{for}\  N=2},
    \end{equation}
as $\e\to 0$, for some positive constant $C>0$ where $\la$ is the simple eigenvalue associated to $\O$.

Afterward, Abatangelo, L\'ena and Musolino \cite{alm1,alm2,alm2026} studied the case that $\lambda$ is multiple and $u$ can vanish. As remarked before, this situation presents substantial difficulties. 

Observe that when $\la$ is an eigenvalue of multiplicity $m$, although one can still prove the existence of $m$ eigenvalues $\la_{j,\e},\to\la$, $j=1,\ldots,m$, one of the main problem is to derive the corresponding asymptotics according to $j=1,\ldots,m$.

Let us start when the multiplicity of the eigenvalue $\la$ is {\em two}. One of the results in \cite{alm1,alm2} is that under the condition
\begin{equation}\label{ii1}
\big(u_1(P),u_2(P)\big)\ne(0,0)
\end{equation}
(here $u_1,u_2$ are the two eigenfunctions
related to $\la$) then we have $two$ simple eigenvalues $\la_{1,\e},\la_{2,\e}\to\la$ and their asymptotic expansion is given by (here $N\ge3$)
\begin{equation}
        \begin{cases}
            \la_{1,\e}=\displaystyle\la +(N-2)N\o_N\left(u_1^2(P)+u_2^2(P)+o(1)\right)\e^{N-2}, \vspace{3mm}\\
            \la_{2,\e}=\displaystyle\la +o\big(\e^{N-2}\big).
        \end{cases}
			\end{equation}
            An important remark is that the two eigenvalues $\la_{1,\e}$ and $\la_{2,\e}$ converge to $\lambda$ at different rates. This immediately implies that they are distinct for sufficiently small $\e>0$.

The case of multiple eigenvalues, namely $m(\lambda)\ge 2$, is much more delicate. Indeed, 
the eigenvalues $\lambda_\varepsilon$ of $\Omega_\varepsilon$ bifurcate from $\lambda$ and may converge to it with the {\em same} asymptotic rate. In this situation, the coefficients of the leading-order terms in their asymptotic expansions play a crucial role.

This makes the problem considerably more involved. Moreover, in general, it is no longer true that, if $m(\lambda)>2$, all the approximating eigenvalues are simple.

When dealing with a multiple eigenvalue, it is clearly necessary to fix a basis of eigenfunctions. Our choice falls on a decomposition of the eigenspace defined by the order of vanishing of the eigenfunctions at the point $P$.
 
\begin{prop}\label{dec} \cite[Proposition 1.10]{alm1}
 Suppose that $\la$ is an eigenvalue with multiplicity $m(\la)\ge2$. Then there exists a decomposition of $E(\lambda)$ into an $L^2$-orthogonal sum of non-trivial subspaces
\[E(\lambda)=E_0\oplus E_1\oplus \dots \oplus E_p\]
(with $p\ge 0$), associated with a increasing sequence of integers 
\[0\leq k_0<k_1<\dots<k_p,\]
such that,
\begin{equation}\label{intr-eqdefE}
			\begin{split}
				&E_0=span\{u_1,\cdots,u_{d_0}\}
                \hbox{ with the vanishing order of } u_i \hbox{ at } P \hbox{ equal to } k_0 \hbox{ for }i=1,\cdots,d_0,\\
                &E_1=span\{u_{1+d_0},\cdots,u_{d_0+d_1}\}
                \hbox{ with the vanishing order of } u_i \hbox{ at } P \hbox{ equal to } k_1 \hbox{ for }i=1+d_0,\cdots,d_0+d_1,\\
				&\cdots\\
                &E_p=span\{u_{1+d_0+\cdots+d_{p-1}},\cdots,u_{m(\la)}\} \hbox{ with the vanishing order of } u_i \hbox{ at } P \hbox{ equal to } k_p \\&\hbox{ for }i=1+d_0+\cdots+d_{p-1},\cdots,m(\lambda).
			\end{split}
		\end{equation}
 which means that, for any $J\in \{0,1,\dots,p\} $, all non-zero functions in $E_J$ have an order of vanishing at $P$ equal to $k_J$. For future reference, we denote by $d_J\ge1$ the dimension of $E_J$, so that 
\[   \sum_{J=0}^p d_J=m(\lambda). 
\]
Moreover, if $k_0=0$, we always have $d_0=1$.
\end{prop}
Some examples of the previous decomposition will be shown in the preliminaries (see Section \ref{pre}).
\begin{rem}\label{ir}
We observe that in the decomposition of Proposition \ref{dec}, it is assumed that the dimension of the subspaces $E_i$ is greater than or equal to $1$. We stress that this is not for free!

Indeed, in order to get it, we need to impose some restrictions on the eigenfunctions of the eigenspace $E_i$.

For example, if $E_0=span\{u_1,\cdots,u_m\}$  then $E_0\ne\emptyset$ with $k_0=0$ if and only if
\begin{equation}\label{i19}
(u_1(P),..,u_m(P))\neq(0,..,0)
\end{equation}
and in this case $l=1$.
However, we note that this assumption is widely satisfied by the eigenfunctions $u$ in any domain $\O$.

Analogously, for $d_J\ge 1$, for example, $m=2$,
we have that the corresponding condition is the existence of $\alpha,\beta\in\mathbb{N}^N$ with $|\alpha|=|\beta|=k_J$ such that 
\begin{equation}\label{ii2}
(D^\alpha u_1(P),D^\beta u_2(P))\neq (0,0) .
\end{equation}
\end{rem}

  In what follows we sum up the organization of the paper. Each one of the following subsections describes the content of the forecoming paper sections.

  \subsection{Refined estimate on $V_{\e,u}$}\label{is0}
Although the results of this section may appear rather technical, they actually play a crucial role throughout the paper. In contrast to the works \cite{afhl,alm1,alm2}, where estimates are established in $L^p$ spaces, here we shall derive pointwise estimates for sufficiently small $\e$. It is worth emphasizing that the Green function will play a central role in this analysis. We believe that this approach may also prove useful in the study of other related problems. The main result of this section is the following,

\begin{lemma}\label{into-l:improvedL2Ve}
Set,
\begin{equation}\label{An2}
A(2,k)=
    \begin{cases}
        2\pi&\hbox{if }k=0,1\\
      \frac{2\pi}{2\times \cdots\times2(k-1)} &\hbox{if }k\ge2 
    \end{cases}
    \qquad\hbox{for }N=2 
\end{equation}
and
\begin{equation}\label{An3}
A(N,k)=
    \begin{cases}
       N(N-2)\omega_N&\hbox{if }k=0\\
       \frac{N(N-2)\omega_N}{(N-2)N\cdots(N-4+2k)} &\hbox{if }k\ge1 
    \end{cases}
    \qquad\hbox{for }N\ge3 .
\end{equation}
Then, as $\e\to 0$, uniformly for $x\in\O_\e$, the following pointwise estimate of $V_{\e,u}$  holds.
\vskip0.2cm
                For $N=2$,
                \begin{equation}\label{intoV1}
                \begin{aligned}
                 V_{\e,u}(x)=\begin{cases}
                    \displaystyle \frac{A(N,0)}{|\log \e|}u(P)G_\Omega(x,P)+O\left(\frac{1}{|\log \e|^2}G_\Omega(x,P)\right),\hbox{ if } k=0,\vspace{2mm}\\
                   \displaystyle A(N,k)\e^{2k}
					\sum_{|\alpha|=k}\frac1{\alpha!}
				\frac{\partial^ku(P)}{\partial x_1^{\alpha_1}\partial x_2^{\alpha_2}}\frac{\partial^kG_{\Omega}(x,P)}{\partial y_1^{\alpha_1}\partial y_2^{\alpha_2}}+O\left(\e\frac{\partial^kG_{\Omega}(x,P)}{\partial y_1^{\alpha_1}\partial y_2^{\alpha_2}}+\frac{\e^{2+k}}{|\log \e|}G_{\Omega}(x,P)\right),\hbox{ if }  k\ge 1.
				\end{cases}
                \end{aligned}
				\end{equation}                   
                For $N\ge 3$,
                \begin{equation}
                \begin{aligned}
                 V_{\e,u}(x)=\begin{cases}
                    \displaystyle A(N,0)\e^{N-2}u(P)G_\Omega(x,P)+O\left(\e^{N-1}G_\Omega(x,P)\right),\hbox{ if } k=0,\vspace{2mm}\\
                   \displaystyle A(N,k)\e^{N-2+2k}
					\sum_{|\alpha|=k}\frac1{\alpha!}
					\frac{\partial^ku(P)}{\partial x_1^{\alpha_1}\cdots\partial x_N^{\alpha_N}}+O\left(\e\frac{\partial^kG_{\Omega}(x,P)}{\partial y_1^{\alpha_1}\cdots\partial y_N^{\alpha_N}}+\e^{N+k}G_{\Omega}(x,P)\right),\hbox{ if }  k\ge 1,
				\end{cases}
                \end{aligned}
				\end{equation}
		 where $\alpha=(\alpha_1,\cdots,\alpha_N) $, $\alpha_1,\cdots,\alpha_N\in \N$, $|\alpha|=\alpha_1+,\cdots,+\alpha_N$, and $\alpha!=\alpha_1!\cdots\alpha_N!$.
            \end{lemma}
 So $V_{\varepsilon,u}$ admits an explicit expansion in terms of derivatives of the Green function of $\Omega$, where the leading order is determined by the vanishing order of $u$ at $P$.

A consequence of the previous formula is that for any bounded smooth domain $\O$, we have
$$
    \hbox{Cap}_\O(B(P,\e),u)=\int_\O|\nabla V_{\e,u}|^2\ dx=\e^{N-2+2{ k}}\sum_{\substack{|\alpha|={ k}\\
           |\beta|={ k}}}\frac{1}{\alpha!\beta!}\frac{\partial^{{ k}} u (P)}{\partial x_{1}^{\alpha_1}\cdots\partial x_{N}^{\alpha_N}}\frac{\partial^{{ k_L}} u (P)}{\partial x_{1}^{\beta_1}\cdots\partial x_{N}^{\beta_N}}C_{\alpha\beta}+O(\e^{N-1+2{ k}}).
    $$
The estimates on $V_{\e,u}$ established in the previous lemmas will be repeatedly used throughout the paper.
            
  \subsection{The multiple eigenvalue case: asymptotics of approximating eigenfunctions}\label{is1}

Let $\lambda$ be a multiple eigenvalue and $\lambda_{j,\varepsilon}\to\lambda$ the corresponding approximating eigenvalues. In this section, we study the asymptotic behavior of the eigenfunctions $u_{j,\varepsilon}\to u_j$, where $u_j$ belongs to the eigenspace $E(\lambda)$. One of the goals of this section is to characterize $u_j$. Clearly, closely related to this problem is the asymptotics of $\lambda_{j,\varepsilon}$.

Clearly, a simple eigenvalue $\la$ is approximated by an eigenvalue $\la_\e$ that converges to $\la$ and is, in turn, simple. Of course, proving that $u_\varepsilon \to u$ presents no difficulty.
Conversely, if an eigenvalue $\la$ has multiplicity $m(\la)\ge2$, it is well known that there exist $m(\la)$ eigenvalues $\la_{1,\e}, \dots, \la_{m(\la),\e}$ converging to $\la$. However, the simplicity of these approximating eigenvalues is far from trivial just as identifying the limit of the eigenfunctions $u_{1,\e}, \dots, u_{m(\la),\e}$ is far from obvious

 Regarding the convergence of the eigenvalues, significant progress in this direction has been made in \cite{alm1,alm2}, where it was shown that the approximating eigenvalues $\la_{1,\e}, \dots, \la_{m(\la),\e}$ converge to $\la$ at rates that depend on the vanishing order of the corresponding eigenfunction $u$ at $P$. More precisely, using the decomposition of $E(\la)$ in Proposition \ref{dec}, we have that
\begin{teo}[see \cite{alm1,alm2}]
 There exist a finite sequence of positive numbers $\{\mu_j\}_{j=1}^{m(\lambda)}$ such that, 
$$
\lambda_{j,\e}=\lambda+\mu_{j}\e^{N-2+2k_j}+o\left(\e^{N-2+2k_j}\right),\q \hbox{as }\e\to 0.
$$
\end{teo}
This result shows the existence of different groups of eigenbranches sharing the same rate of convergence to the limit multiple eigenvalue. Nevertheless, eigenbranches in the same group may not split if  the coefficient in the leading term is the same. 
To the possible splitting of two eigenbranches sharing the same rate of convergence, a crucial role is played by the coefficients $\mu_j$. In \cite{alm2026}, it is proved that these real numbers arise from a suitable eigenvalue problem in finite dimensional spaces, but they can hardly be explicitly described. 

One of the main goals of Section \ref{simpl} is precisely to express the eigenvalues $\mu_j$. By doing so, we will achieve three main objectives:
\begin{enumerate}
    \item We will prove a result which ensure the simplicity of the eigenvalues in term of a matrix whose eigenvalues are more easily computable than \cite{alm2026}
    \item We will provide explicit geometric conditions (in terms of the vanishing order of the eigenfunction $u$ at $P$) that guarantee the simplicity of the approximating eigenvalues (see Remark \ref{ii3}).
    \item Most importantly, this new approach allows us to identify the pointwise limit of the eigenfunctions $u_{j,\e}$ for each $j = 1, \dots, m(\lambda)$.
\end{enumerate}
Finally, we observe that despite the similarity of our results to those in \cite{alm1,alm2}, the proofs and the techniques employed are rather different. We now provide an outline of the proofs and state the main results.

For the consistency of notation with Proposition \ref{dec}, when $J=0$, we interpret $j=1+d_0+\cdots+d_{J-1},\cdots,d_0+\cdots+d_J$ as $j=1,\cdots,d_0$.

The next result, which is a consequence of Lemma \ref{into-l:improvedL2Ve} establishes the announced criterion by providing an explicit condition for the simplicity of the eigenvalues. We also emphasize that for any $J=0,1,\cdots,p$, the matrix $G_{k_J}(P)$ plays a fundamental role in determining the limiting profile of the approximating eigenfunction $u_{j,\varepsilon}$ for $j=1+d_0+..+d_{J-1},\cdots,d_0+..+d_J$. Indeed, the eigenvectors of $G_{k_J}(P)$ determine the linear combination of the generators of $E(\lambda)$ to which $u_{j,\varepsilon}$ converges (see \eqref{L2approx}).

\begin{teo}\label{sim}
		Assume that 
        \begin{equation}\label{i15}
          E(\lambda)=E_0\oplus E_1\oplus\cdots\oplus E_p  
        \end{equation}
        and let $\{u_1,\dots,u_{m(\lambda)}\}$ be an orthonormal basis of $E(\lambda)$ associated to the order decomposition. Then we have the following results.
        
        $(1)$ If $(u_1(P),..,u_{m(\lambda)}(P))\neq(0,..,0)$ (and so $E_0\ne\emptyset$ with $k_0=0$ ) then $\la_{1,\e}$ is simple and it verifies 
		\begin{equation}\label{i16}
        \la_{1,\e}=
        \begin{cases}
            \displaystyle\la +(N-2)N\o_N\left(\sum_{i=1}^{m(\lambda)}u_i^2(P)+o(1)\right)\e^{N-2}, \ \ \hbox{if } N\ge 3,\vspace{3mm}\\
           \displaystyle\la +\frac{2\pi\left(\sum_{i=1}^{m(\lambda)}u_i^2(P)+o(1)\right)}{|\log\e|}, \ \ \hbox{if } N=2.\\
        \end{cases}
			\end{equation}
    Moreover, $u_{1,\e}\to u_1$ in $H^1_0(\O)$.
    \vskip0.2cm
    $(2)$ Otherwise, for any $J=0,1,\cdots,p$, 
			\begin{equation}\label{d12}
				\la_{j,\e}=\la +\big(\La_j+o(1)\big)\e^{N-2+2k_J}\quad\hbox{for  }j=1+d_0+..+d_{J-1},\cdots,d_0+..+d_J,
			\end{equation}
 where $\La_j,\ j=1+d_0+..+d_{J-1},\cdots,d_0+..+d_J,$ are eigenvalues of the matrix $G_{k_J}(P)$
$$
	\!\!\!\!\!\!\!\!\!\!\!\!\!\!\!\!\!\!\!\!\!\!\!G_{k_J}(P)=
\begin{bmatrix}
\displaystyle\sum_{\substack{|\alpha|={ k_J}\\
           |\beta|={ k_J}}}\frac{C_{\alpha\beta}}{\alpha!\beta!}\frac{\partial^{{ k_J}} u_{1+..+d_{J-1}} (P)}{\partial x_{1}^{\alpha_1}\cdots\partial x_{N}^{\alpha_N}}\frac{\partial^{{ k_J}} u_{1+..+d_{J-1}} (P)}{\partial x_{1}^{\beta_1}\cdots\partial x_{N}^{\beta_N}} &\dots & \displaystyle\sum_{\substack{|\alpha|={ k_J}\\
           |\beta|={ k_J}}}\frac{C_{\alpha\beta}}{\alpha!\beta!}\frac{\partial^{{ k_J}} u_{1+..+d_{J-1}} (P)}{\partial x_{1}^{\alpha_1}\cdots\partial x_{N}^{\alpha_N}}\frac{\partial^{{ k_J}} u_{d_0+..+d_J} (P)}{\partial x_{1}^{\beta_1}\cdots\partial x_{N}^{\beta_N}} \\
    \vdots&  &\vdots\\
  \displaystyle \sum_{\substack{|\alpha|={ k_J}\\
           |\beta|={ k_J}}}\frac{C_{\alpha\beta}}{\alpha!\beta!}\frac{\partial^{{ k_J}} u_{1+..+d_{J-1}} (P)}{\partial x_{1}^{\alpha_1}\cdots\partial x_{N}^{\alpha_N}}\frac{\partial^{{ k_J}} u_{d_0+..+d_J} (P)}{\partial x_{1}^{\beta_1}\cdots\partial x_{N}^{\beta_N}} &\dots & \displaystyle\sum_{\substack{|\alpha|={ k_J}\\
           |\beta|={ k_J}}}\frac{C_{\alpha\beta}}{\alpha!\beta!}\frac{\partial^{{ k_J}} u_{d_0+..+d_J} (P)}{\partial x_{1}^{\alpha_1}\cdots\partial x_{N}^{\alpha_N}}\frac{\partial^{{ k_J}} u_{d_0+..+d_J} (P)}{\partial x_{1}^{\beta_1}\cdots\partial x_{N}^{\beta_N}} \\
\end{bmatrix},$$
      where 
        \begin{equation}\label{defCalbe}
            C_{\alpha\beta}:=
            \begin{cases}
              \displaystyle 2(N-2+|\alpha|+|\beta|)\frac{\prod_{i=1}^N\Gamma(\frac{\alpha_i+\beta_i+1}{2})}{\Gamma(k+\frac{N}{2})},  &if\ \alpha_i+\beta_i\ is  \ even, \ i=1,\cdots,N,\vspace{2mm}\\
                0,& otherwise.
            \end{cases}
        \end{equation}
 In particular, if all the eigenvalues of $G_{k_J}(P)$ are simple then $\la_{j,\e}$ are simple.
 
 Moreover, there exist coefficients $\alpha_{ij}$ such that
        \begin{equation}\label{L2approx}
            u_{j,\e}\to \sum_{i=1+d_0+..+d_{J-1}}^{d_0+..+d_J}\alpha_{ij} u_i \hbox{ in } H^1_0(\O),
        \end{equation}
        and, for each $j=1+d_0+..+d_{J-1},\cdots,d_0+..+d_J$, the vectors $(\alpha_{1+d_0+..+d_{J-1}j},\cdots,\alpha_{d_0+..+d_Jj})$,  is an eigenvector of the matrix $G_{k_J}(P)$  corresponding to $\Lambda_j$.

	\end{teo}
\begin{rem}
   The numbers $\Lambda_j$ appearing in \eqref{d12} are independent of the choice of the orthonormal basis of $E(\la)$. The proof of this fact  is postponed to Lemma \ref{ind} in Section \ref{simpl}. 
\end{rem}
\begin{rem}\label{ii3}
The condition $(u_1(P),..,u_l(P))\neq(0,..,0)$ not only guarantees that $E_0\ne\emptyset$ with $k_0=0$ (see Remark \ref{ir}) but also turns out to be necessary for the validity of the results in Theorem \ref{sim}. 
In fact, even in the simple case that $E(\la)=\{u_1,u_2\}$
if $u_1(P)=u_2(P)=0$  then it is no longer true that the eingenvalue $\la_{1,\e}$ is simple.

An example is given by $N=2$, $\Omega=B_1(0)$, and $\la$ being the second eigenvalue. In this case, if $P=0$, it is immediate that the any two eigenfunctions satisfy $u_1(P)=u_2(P)=0$. On the other hand, if we remove a small ball $B_\e(0)$, the second eigenvalue in the annular domain will continue to be double for every small $\e>0$ and so the simplicity of $\la_{1,\e}$ does not hold. 

Moreover, even the estimate \eqref{i16} is no true. Actually in this case it is possible to show that
\begin{equation}\label{i17}
        \la_{1,\e},\la_{2,\e}=\left(C_0+o(1)\right)\e^N.
			\end{equation}
Clearly, the same problem arises in any (possibly non-symmetric) domain such that the nodal set of two eigenfunctions $u_1$ and $u_2$ intersect somewhere at a point $P$. In this case, we expect that, ``generically'',  the coefficient $C_0$ appearing in \eqref{i17} is different for the two eigenvalues $\la_{1,\e},\la_{2,\e}$. Consequently, the simplicity of the two eigenvalues would be recovered. However, despite its undoubted interest, we will not pursue this matter further.
\end{rem}

For the sake of clarity, we state the previous result in the case of $E(\lambda)=E_0\oplus E_1$ with $k_0=0$ and $k_1=1$.
	\begin{cor}\label{cor4.2}
		Let $N\ge 3$,
		$E(\la)=E_0\oplus E_1,$  with $k_0=0$ and $k_1=1$. Moreover assume that \eqref{i19} holds.
        
		Then we have that $\la_{1,\e}$ is simple and it verifies 
			\begin{equation}\label{eq0412-1}
				\la_{1,\e}=\la +(N-2)N\o_N\left(\sum_{i=1}^{m(\la)}u_i^2(P)+o(1)\right)\e^{N-2}\hbox{ and }u_{1,\e}\to u_1 \hbox{ in } H^1_0(\O).
			\end{equation}
		Next,
		\begin{equation}\label{eq0412-2}
			\la_{j,\e}=\la +\big(\La_j+o(1)\big)\e^N\quad\hbox{ for }j=2,..,m(\la), 
		\end{equation}
and
\begin{equation}\label{eq0412-2bis}
			u_{j,\e}\to \sum_{i=2}^{m(\lambda)}\alpha_{ij}u_i\hbox{ in } H^1_0(\O)\quad\hbox{ for }j=2,..,m(\la),
		\end{equation}       
		where $\La_j$, $j=2,..,m(\la),$ are eigenvalues of the matrix 
			\begin{equation*}\label{mat*}
				G(P)=N\o_N\big(\nabla u_i(P)\cdot\nabla u_l(P)
				\big)=N\o_N\begin{pmatrix}
					|\nabla u_2(P)|^2&\cdots&\nabla u_2(P)\cdot\nabla u_{m(\la)}(P)\\
				\vdots&  &\vdots\\
					\nabla u_{m(\la)}(P)\cdot\nabla u_2(P)&\cdots&
					|\nabla u_{m(\la)}(P)|^2
				\end{pmatrix},
			\end{equation*}
       and the vector $(\alpha_{2j},\cdots,\alpha_{m(\lambda)j})$, $j=2,..,m(\la),$ are eigenvectors of the matrix $G(P)$.
            
		In particular, if all the eigenvalues of $G(P)$ are simple, then $\la_{j,\e}$ are simple for any $j=2,..,m(\la).$ \\
	\end{cor}
    An analogous result holds if $N=2$ (see Corollary \ref{cor1510-1} in Section \ref{simpl}).
    \begin{rem}\label{ir1}
It follows from the previous corollary that if the multiplicity of $\la$ is $2$, then the approximating eigenvalues $\la_{1,\e}$ and $\la_{2,\e}$ are both simple and their asymptotic expansions are given by \eqref{eq0412-1} and \eqref{eq0412-2}. This result had nonetheless been proven in \cite{alm1,alm2}.
\end{rem}

\begin{rem}\label{ii4}
We prove here that the simplicity condition for the eigenvalues of the matrix $G(P)$ in Corollary \ref{cor4.2} is sharp. Again, let $\Omega=B_1(0)$, $\la$ be the second eigenvalue, and $P=0$. Using the same notation as in Example 2 in Subsection \ref{ss2}, we have 
that the  Bessel functions $J_n(s)$ satisfy that (see Section 2.12 of \cite{milne1945})
\begin{equation}\label{PropBess}
    J_n'(s)=\frac{1}{2}\left(J_{n-1}(s)-J_{n+1}(s)\right),\ \ J_0(0)=1 \ \ \hbox{and}\ \  \lim_{s\to0}\frac{J_n(s)}{\frac{1}{n!}(\frac{s}{2})^n}=1.
\end{equation}
For $P=0$, we have  $u_1(0)=u_2(0)=0$ and
        \begin{equation}
            \nabla u_1(0)=(\frac{j_{1,1}}{2},0)\ \ and \  \  \nabla u_2(0)=(0,\frac{j_{1,1}}{2}).
        \end{equation}
Then matrix $G(0)$ in \eqref{mat*} becomes
$$G(0)=
\begin{pmatrix}
    \frac{j_{1,1}^2}{4}&0\\
                   0&\frac{j_{1,1}^2}{4}\\
\end{pmatrix}
$$
which has the eigenvalues $\Lambda_2=\Lambda_3$. So the condition on the simplicity of the eigenvalues of $G(0)$ fails.
On the other hand, from the standard separation of variables, we get that the eigenvalue $\lambda_{2,\e}$ is multiple. It shows the sharpness of the condition: in general if the eigenvalues $\Lambda_j$ of $G_{k_J}$ are not simple then the eigenvalues $\lambda_{j,\e}$ may fail to be simple.

The same example can be provided if $\O=B_1(0)\subset\R^3$
The corresponding matrix $G(0)$ here becomes
\begin{equation}
  G(0)
                =  3\omega_3\begin{pmatrix}
					\frac{j_{1,1}^2}{9}& 0& 0\\
                   0& \frac{j_{1,1}^2}{9}& 0\\
                   0& 0& \frac{j_{1,1}^2}{9}\\
				\end{pmatrix}.
\end{equation}
Hence $G(0)$ has one multiple eigenvalue and as before from the standard separation of variables we get that $\lambda_{2,\e}$ is multiple. 
\end{rem}
The validity of condition \eqref{ii1} (as well as \eqref{ii2}) as the point $P$ varies in $\Omega$ is an extremely interesting open problem. This problem is connected to the potential "generic" simplicity of the eigenvalues under small displacements of the point. The subsequent question is natural: 
\vskip0.1cm
\noindent {\bf Question.} {\em Is it true that small movements of the ball $B(P,\e)$ make all the eigenvalues simple?} 
\vskip0.1cm
This question represents the natural analogue of the result by Micheletti and Uhlenbeck, where small deformations of $\partial\Omega$ are replaced by small displacements of the point $P$.
In the next example, we will show that, as the question is posed, the answer is negative!

{\bf Example}  Again, let $\Omega=B_1(0)$, $\la$ be the second eigenvalue, and $P=0$ as in Remark \ref{ii4}. We will prove that
\begin{equation}\label{ir4}
\la_{2,\e}\hbox{ has }
\begin{cases}
\hbox{multiplicity }3&\hbox{if }P=0\\
\hbox{multiplicity }2&\hbox{if }P\ne0
\end{cases}
\end{equation}
This proves that it is not possible to ``move'' the point $P$ and obtain simple eigenvalues.

The claim for $P=0$ in \eqref{ir4} is immediate because if $P=0$ then $\O_\e=B_1\setminus B(0,\e)$ and it is known that in this case the second eigenvalue $\la_{2,\e}$ is still multiple with explicit eigenfunctions in term of the Bessel functions.

If $P\ne0$, up to rotation we can assume that $P=(0,0,|P|)$. A direct calculation by separation of variables shows that, in this case, a double eigenvalue occurs for every $\e>0$. Since this property holds for every point $P$ belonging to the $z$-axis, we have that \eqref{ir4} follows.

In a forthcoming paper, we intend to further investigate this phenomenon, addressing the question of for which domains the generic simplicity of eigenvalues holds when a small ball is removed from a domain $\O$.

Further properties and examples illustrating the sharpness of our results are established in Section~\ref{simpl}.

  \subsection{Quantitative estimates on the eigenfunction $u_\e$}\label{is2}
 A second main result concerns the asymptotic behavior of eigenfunctions. As mentioned above, the standard regularity theory guarantees only that $u_\e \to u$ uniformly on compact subsets of $\O$ that do not contain $P$. Using the estimates for the function $V_{\e,u}$, combined with suitable spectral-theoretic arguments, we substantially strengthen this convergence result by deriving quantitative estimates for the difference $u_\e - u$. Analogously to the case of the eigenvalues, one may wonder what the influence of the vanishing order of $u$ at $P$ might be on the asymptotic estimate of $u_\varepsilon$.
 
 In the case of {\bf simple} eigenvalues, we show that the perturbed eigenfunction $u_\varepsilon$ can be approximated by (obviously) the eigenfunction $u$ and its associated capacitary correction $V_{\e,u}$. More precisely, one has

 \begin{teo}\label{intr-T1}
Let $(\lambda,u)$ a pair satisfying \eqref{eq:eigenvalueproblem} and $u_\e$ the $L^2(\Omega)$-normalized solution to \eqref{eq:perturbedeigenvalueproblem}. Then, denoting by $k$ the vanishing order of $u$ at $P$ (see definition \ref{defvanishing}),
     \begin{equation}\label{d0}
			u_\e(x)-u(x)+V_{\e,u}(x)= \begin{cases}
                O\left(\frac{1}{|\log\e|}\right)\quad& if  \ N=2  \ and\ k=0\\
                O\big(\e^2|\log\e|\big) & if  \ N=2  \ and\ k=1\\
                O\big(\e^{k+\frac{N}{2}}\big) & if  \ N=2  \ and\ k\ge 2\\
                O\big(\e\big) & if  \ N=3  \ and\ k=0\\
                O\big(\e^{k+\frac{N}{2}}\big) & if  \ N=3  \ and\ k\ge 1\\
                 O\big(\e^{k+ 2(1-\delta)}\big) & if  \ N\ge 4  \ and\ k\ge 0\\
			\end{cases}
		\end{equation}
		as $\e\to0$, uniformly for $x\in\O_\e$ and for any $0<\delta<\frac{1}{2}$.
	\end{teo}
From the previous result, two consequences are obtained: the first is a quantitative estimate of the norm $\|u_\varepsilon - u\|$ on every compact subset $K \subset\subset \Omega \setminus \{P\}$. 
\begin{cor}\label{T2'}{\bf (Improved local estimate)}
		We have that
		\begin{equation}\label{d1}
			||u_\e-u||_{L^\infty(K)}=\begin{cases}
                O\left(\frac{1}{|\log\e|}\right)\quad& if  \ N=2  \ and\ k=0\\
                O\big(\e^2|\log\e|\big) & if  \ N=2  \ and\ k=1\\
                O\big(\e^{k+\frac{N}{2}}\big) & if  \ N=2  \ and\ k\ge 2\\
                O\big(\e\big) & if  \ N=3  \ and\ k=0\\
                O\big(\e^{k+\frac{N}{2}}\big) & if  \ N=3  \ and\ k\ge 1\\
                 O\big(\e^{k+ 2(1-\delta)}\big) & if  \ N\ge 4  \ and\ k\ge 0\\
            \end{cases}
		\end{equation}
as $\e\to 0$,		for any compact set $K\subset\O_\e$.
	\end{cor}
The second consequence is the pointwise behavior of $u_\varepsilon$ near $\partial B(P, \varepsilon)$. This will be important in describing the nodal domains of $u$ close to $\partial B(P, \varepsilon)$. Observe that this result cannot be obtained by means of the standard regularity theory.
\begin{cor}\label{T2}{\bf (Sharp boundary  estimate)}
		We have that, as $\e\to 0$,
        \begin{equation}\label{d2}
          u_\e(x)= 
          \begin{cases}
              \displaystyle  \left(1+\frac{\log |x-P|}{|\log\e|}\right)u(x)
              +O\left(\frac{1}{|\log\e|}\right), &if \ N=2  \ and\ k=0,\vspace{2mm}\\
			\displaystyle	  \left(1-\frac{\e^{2k}}{|x-P|^{2k}}\right)u(x)+O(\e^{2}|\log\e|),& if  \ N=2  \ and\ k=1,\vspace{2mm}\\
            \displaystyle        \left(1-\frac{\e^{N-2+2k}}{|x-P|^{N-2+2k}}
				\right)u(x)+O(\e^{k+1}), &otherwise,
                \end{cases}
        \end{equation}
		uniformly for $|x-P|\le C\e$.
	\end{cor}
The previous results hold for simple eigenvalues. If the eigenvalue is multiple, the estimates we obtain are not too different. However a more involved decomposition is required. Since we believe this topic is somewhat technical, we prefer to avoid discussing it here in the introduction, and we refer the interested reader to the results in Section \ref{sec-function}.

  \subsection{The nodal set of $u_\e$}\label{is3}

Finally, we apply the previous results to the study of nodal sets of $u_\e$.

The geometry of the nodal domains of Laplacian eigenfunctions with zero Dirichlet boundary conditions is deeply intertwined with the shape of the domain. Geometric properties such as convexity and  curvature play a crucial role in determining the convexity of the nodal lines.

 A famous conjecture of Payne \cite{payne1967} states that  any second eigenfunction of problem \eqref{eq:eigenvalueproblem} for a bounded domain $\O\subset\R^2$ cannot have a closed nodal line. This conjecture has been extensively investigated in the literature and it was proved by Melas \cite{melas1992} 
(see \cite{PU1990,MHTONN1997,BCM2021} for more reference).

     On the other hand, the convexity of the domain turns out to be a necessary condition for the nodal line of the second eigenfunction to touch the boundary. Several notable counterexamples for domains with holes have been shown in \cite{MHTONN1997} and \cite{fl}. Moving on to the case of our domain $\O_\e=\O\setminus B(P,\e)$, Mukherjee and Saha \cite{MMSS2022} interested in seeing whether the nodal set of the second Dirichlet eigenfunction for $\O_\e$ intersects $\partial \O_\e$ for $\e$ small enough. In the literature this is called "Payne property". 
     
     They proved a perturbed version of the main result in \cite{melas1992}. (See  \cite[Proposition 4.2]{MMSS2022}.) More precisely, let $\Omega\subset\mathbb{R}^N$ be a convex domain with smooth boundary, and let $x_1,x_2,\ldots,x_l\in\Omega$ be points lying outside the nodal set of the second Dirichlet eigenfunction of $\Omega$. Define
\[
\Omega(\varepsilon):=\Omega\setminus\bigcup_{i=1}^l B(x_i,\varepsilon).
\]
If $\Omega$ satisfies the Payne property, then $\Omega(\varepsilon)$ also satisfies the Payne property for all sufficiently small $\varepsilon$.
     
In this paper, we move in the same direction as \cite{MMSS2022}, aiming to understand how the nodal domain of $u_\varepsilon$ behaves compared to that of $u$. Unlike \cite{MMSS2022}, we allow the point $P$ to belong to the nodal domain of $u$ and we consider eigenfunctions of any order in any dimension.

Of course the boundary estimates proved in Corollary \ref{T2} will play the crucial role.

Let us introduce the following notations.
Set
	\begin{equation}
		\mathcal{N_\e}=\overline{\{x\in\O_\e\hbox{ such that }u_\e=0\}}
	\end{equation}
	and 
	\begin{equation}
		\mathcal{N}=\overline{\{x\in\O\hbox{ such that }u=0\}}.
	\end{equation}
    We now consider the two different cases in which $P$ either belongs to the nodal set of $u$ or does not.
    \vskip0.2cm
{\bf The case $P\notin\mathcal{N}$}

Here we basically show that no other connected component of $\mathcal{N_\e}$ appears "near" $\partial B(P,\e)$. Our first result is the following,


	\begin{prop}\label{intr-p0}
Let $\O\subset\R^N$ with $N\ge2$. If $P\notin \mathcal{N}$, then		there exists $\delta>0$ such that  $\mathcal{N_\e}\cap B(P,\delta)=\emptyset$.
	\end{prop}
Furthermore, the number of nodal regions of $u_\e$ will be not greater than the number of nodal regions of $u$.
    \begin{teo}\label{intr-th-0804}
   Let $\O\subset\R^N$ with $N\ge2$. If $P\notin \mathcal{N}$, then 
          $$
    \sharp\{\hbox{nodal regions of } u_{\e}\}\leq \sharp \{\hbox{nodal regions of } u\}.
    $$
    \end{teo}
    An immediate corollary is the following,
   \begin{cor}
 Let $\O\subset\R^N$ with $N\ge2$ and $P\notin \mathcal{N}$. If $\sharp \{\hbox{nodal regions of } u\}=2$, then $\sharp \{\hbox{nodal regions of } u_\e\}= 2$.
   \end{cor}

     For a second eigenfunction $u_2$, it has 2 nodal regions. If $u_2(P)\neq 0$, then by Theorem \ref{intr-th-0804}, it follows that $\sharp \{\hbox{nodal regions of } u_\e\}=2$. We obtain  the same result as \cite[Proposition 2.6]{MukherjeeSaha2025} in the case  $N=2$, for the set $\O_\e=\O\setminus B(P,\e)$ as follows.  It   generalizes the "Payne property".
\begin{cor}
     Let $N=2$ and assume that $P\notin \mathcal{N}$ and $\sharp \{\hbox{nodal regions of } u\}=2$.  If the nodal set $\mathcal{N}$ of $u$ intersects $\partial \O$ at exactly 2 points, then the nodal set $\mathcal{N}_\e$ of $u_\e$ intersects $\partial \O$ at exactly 2 points. If the nodal set $\mathcal{N}$ of $u$ does not intersect $\partial \O$, then the nodal set $\mathcal{N}_\e$ of $u_\e$ does not intersect $\partial \O$. 
   \end{cor}
    \vskip0.2cm
   {\bf The case $P\in\mathcal{N}$}
   
Here the situation becomes considerably more involved. For this reason we restrict our attention to the case $N = 2$ only.
\cite[Theorem 2.5 (ii)]{cheng1976} shows that 
        there exists $\delta_0>0$ such that the nodal line $\mathcal{N}\cap B(P,\delta_0)$ is a union of $k$ curves that only intersect at point $P$. 

        We will prove that this property naturally extends to the boundary of the ball $\partial B(P,\varepsilon)$, i.e the nodal line 
$\mathcal{N}_{\e}$ intersects $\partial B(P,\e)$ with $2k$ points (see, for example, the  figure below  $k=2$).
   \begin{teo}\label{intr-apmain}
   If $N=2$, and $k\ge 1$ is the vanishing order of $u$, then 
        $\mathcal{N}_{\e}$ intersects $\partial B(P,\e)$ at $2k$ points.
    \end{teo}
\begin{figure}[h]
    \centering
    \includegraphics[width=0.8\linewidth]{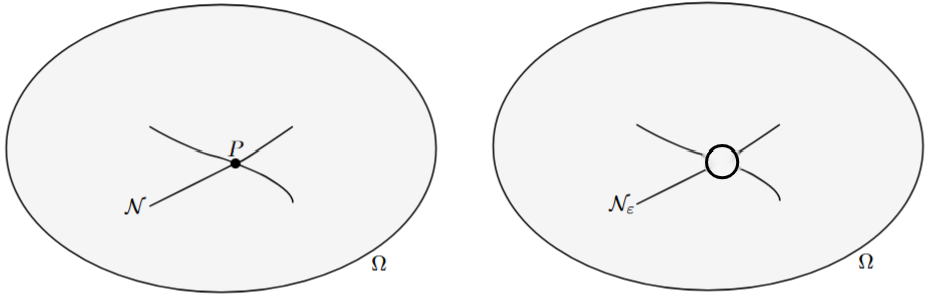}
    \label{fig:placeholder}
\end{figure}

The figure above shows how the vanishing order $k$ of an eigenfunction $u$ determines the local behavior of the nodal set in a neighborhood of $B(P,\varepsilon)$.

The paper is organized as follows. In Section \ref{pre}, we present the preliminaries, where we review several known results that are essential for our analysis and explain the motivation  of our approach.  In Section \ref{sec-Vu}, we  improve the estimates of the $u$-capacitary potential. In Section \ref{simpl}, we establish the asymptotic expansions of the perturbed eigenvalues $\lambda_\varepsilon$ and investigate their simplicity. The asymptotic behavior of the corresponding eigenfunctions $u_\varepsilon$ is analyzed in Section \ref{sec-function}. Applications of these estimates to the nodal sets and nodal domains of the perturbed eigenfunctions are presented in Section \ref{sec-application}. Finally, several auxiliary theorems and technical lemmas used throughout the paper are collected in Section \ref{sec-appendix}.


	\section{Preliminaries and known facts}\label{pre}
     Throughout the paper,  let $\alpha = (\alpha_1,\ldots,\alpha_N) \in \mathbb{N}^N$ be a multi-index and set 
    $$
    |\alpha| = \alpha_1 + \cdots + \alpha_N,\q \alpha!=\alpha_1!\cdots\alpha_N!,\q x^\alpha = x_1^{\alpha_1} \cdots x_N^{\alpha_N} .
    $$
    We begin by recalling some well-known properties of the zeros of eigenfunctions.  
    \subsection{Vanishing order and local expansion}
\begin{lemma}\label{lem1}
		If $u$ is an eigenfunction of the Dirichlet Laplacian on $\Omega$, that is an  open bounded connected domain in $\R^N$ containing $P$.
        Then $u\in C^{\infty}(\Omega)$ and there exists $k\in\N$  and there exist  $k\in\N$ and a harmonic homogeneous polynomial  $g$ of degree $k$ such that  
    \[
    r^{-k}u(P+rx) \to g(x) \quad \text{in }C^{1,\tau}(\overline B_1) \text{ as }r\to0 
    \] 
    for any $\tau\in(0,1)$. 

    In this case, 
    \[
    g(x)= |x|^k \Psi\left(\frac{x}{|x|}\right),
    \]
    where $\Psi$ denotes a spherical harmonic, that is a function such that 
		\[
		-\Delta_{\S^{N-1}} \Psi = k(N-2+k) \Psi. 
		\]  
	\end{lemma}

    \begin{rem}\label{rmi1}
    
        Actually, since $u$ is real-analytic in $\O$ (see \cite[Theorem 7.5.1]{hormander1963linear}), we use the Taylor expansion in a neighborhood of the point $P$, 
\begin{equation*}
    u(P+x)=\sum_{l=0}^\infty\sum_{|\beta|=l}\frac{1}{\beta!}D^\beta u(P)x^\beta=\sum_{l=0}^\infty g_l(x)=\sum_{l=\kappa(u)}^\infty g_l(x), 
\end{equation*}
where $\beta=(\beta_1,\cdots,\beta_N)$, $\beta_1,\cdots,\beta_N\in\N$, $|\beta|=\beta_1+\cdots+\beta_N$,
$$
 D^\beta u(x):=\frac{\partial^{|\beta|}u(x)}{\partial x_1^{\beta_1}\cdots\partial x_N^{\beta_N}},\ \ 
g_l(x):=\sum_{|\beta|=l}\frac{1}{\beta!}D^\beta u(P)x^\beta.
$$
Here we denote by $\kappa(u)$ the smallest $l$ such that 
$g_l(x)\neq 0.$
Since $u$ satisfies \eqref{eq:eigenvalueproblem}, it follows that 
\begin{equation*}
    -\sum_{l=\kappa(u)}^\infty\Delta g_l(x)= \lambda\sum_{l=\kappa(u)}^\infty  g_l(x) ,
\end{equation*}
where $\Delta g_l(x)$ is also a homogeneous polynomial of degree $l-2$. Then we see that $g_{\kappa(u)}$ must be harmonic. In view of the above analysis, it is obvious that the \emph{vanishing order} $k$ of $u$ at $P$ in Lemma \ref{lem1} is actually $\kappa(u)$ and 
\begin{equation}\label{eq-1304-2}
    g(x)=\sum_{|\beta|=k}\frac{1}{\beta!}D^\beta u(P)x^\beta.
\end{equation}
    \end{rem}
Combining Remark \ref{rmi1} and Definition \ref{defvanishing}, we have the following Lemma.
\begin{lemma}\label{lem2212-1}
    If the the vanishing order of $u$  at $P$ is equal to $k$, then $g(x)$ defined as \eqref{eq-1304-2} is harmonic.
\end{lemma}

\subsection{Eigenspace decomposition}\label{ss2}
	Let us consider an eigenvalue $\la$ in $\O$ with multiplicity $m(\la)$
     and let $E(\la)$ denote the corresponding eigenspace.  The main properties of the decomposition of $E(\lambda)$ as the direct sum of the subspaces $E_i$ were established in \cite{alm1,alm2}. For the reader's convenience, we recall some of them here. 

     One of the most important properties we would like to emphasize is that the decomposition of the eigenspace $E(\la)$ in general depends on the position of the hole $B(P,\e)$. This means that in certain cases our results will depend on the position of $P$!
     \vskip0.2cm
{\bf Example 1} Suppose that the $\lambda=\la_1$ is the first eigenvalue and $u_1$ the corresponding eigenfunction. Then
\[
E = E_0 = \operatorname{span}\{u_1\}, \quad(\hbox{with $k_0=0$, since $u_1$ does not vanish}).
\]
More generally, if $\lambda$ is any simple eigenvalue, then again
\[
E = E_0= \operatorname{span}\{u_1\},
\]
where $k_0$ is the smallest index for which there exists a non-zero derivative of $u$.
   \vskip0.2cm
 As can be seen, in this case the position of the point $P$ plays no role. In the next example, things will go differently.
   \vskip0.2cm
{\bf Example 2}   
 Assume that  $N=2$, $\lambda=\lambda_2$ and $\O=B(0,1)\subset\R^2$.
The eigenfunctions corresponding to $\lambda_2$  are  
$$
u_1(x)=J_1(j_{1,1}r)\frac{x_1}{r},\ \ u_2(x)=J_1(j_{1,1}r)\frac{x_2}{r},
$$
where $r=\sqrt{x_1^2+x_2^2}$,  $J_1$ is the Bessel function of the first kind and $j_{1,1}$ is the first positive zero of $J_1$.

Now if $P\ne0$ then $\big(u_1(P),u_2(P)\big)\ne(0,0)$ and then (assuming that $u_1(P)\ne0$),
$$E_0=\operatorname{span}\{u_1\}\quad\big(\hbox{since $u_1(P)\ne0$}\big) $$
and
$$E_1= \operatorname{span}\{u_2(P)u_1-u_1(P)u_2\}\quad\big(\hbox{since $\nabla\big(u_2(P)u_1-u_1(P)u_2\big)\ne\vec{0}$}\big)$$
where the last assertion follows from the fact that $\big(u_1(P),u_2(P)\big)\ne(0,0)$ and $\nabla u_1(P)$ is not parallel to  $\nabla u_2(P)$.
\vskip0.2cm
However, if $P=0$, the previous decomposition does not occur since both eigenfunctions vanish with the same  order at $P=0$.
 Therefore, in the decomposition of $E(\lambda)$ it is necessary to consider the next order of vanishing (which, in this case, is $1$). Indeed, observing that, by the properties of the Bessel functions $j_{1,1}$ is the first positive zero point of $j_1(r)$ we have,
 \begin{equation}\label{bes}
            \nabla u_1(0)=\left(\frac{j_{1,1}}{2},0\right)\ \ and \  \  \nabla u_2(0)=\left(0,\frac{j_{1,1}}{2}\right).
        \end{equation}
we get that $E(\la)=E_0$ with $k_0=1$ and
$$E_1=\operatorname{span}\{u_1,u_2\}\quad\big(\hbox{since $u_1(0)=u_2(0)=0$ and  $\nabla\big(C_1u_1+C_2u_2\big)\ne\vec{0}$ for any $(C_1,C_2)\ne(0,0)$.}\big)$$

\vskip0.2cm
{\bf Example 3: }
Assume that $N=3$, $\lambda=\lambda_2$ and $\O=B(0,1)\subset\R^3$. Here dim$E(\la)=3$ and the eigenfunctions corresponding  $\lambda$  are 
\begin{equation}
    u_1(x)=j_1(j_{1,1} r)\frac{x_1}{r},\ \  u_2(x)=j_1(j_{1,1} r)\frac{x_2}{r},\ \  u_3(x)=j_1(j_{1,1} r)\frac{x_3}{r},\ \ 
\end{equation}
where $r=\sqrt{x_1^2+x_2^2+x_3^2}$, $j_1(r)=\frac{\sin r}{r^2}-\frac{\cos r}{r}$.
\vskip0.1cm
Case 1. $P$ does not belong to any of the coordinate axes $x_1$, $x_2$, or $x_3$.

This case is very similar to the corresponding case in Example $2$. Here we have that $E(\la)=E_0\oplus E_1$, with dim$E_1=2$ with
$$E_0=\operatorname{span}\{u_1\}$$
and
$$E_1= \operatorname{span}\{\big(u_2(P)u_1-u_1(P)u_2,u_3(P)u_1-u_1(P)u_3\big)\}.$$
\vskip0.2cm
Case 2. Let $P$ lie on one of the Cartesian axes $x_1, x_2$ or $x_3$, with $P \neq 0$. We assume that $P=(0,0,\frac{1}{2})\in B(0,1)$. Since $j_{1,1}$ is the first zero point of $j_1(r)$, we have that $j_1(\frac{j_{1,1}}{2})\neq 0$.
Then 
\begin{equation}
    u_1(P)=0,\ \ u_2(P)=0,\ \ u_3(P)\neq 0,
\end{equation}
and 
\begin{equation}
    \nabla u_2(P)=(2j_1(\frac{j_{1,1}}{2}),0,0) ,\ \ \nabla u_3(P)=(0,2j_1(\frac{j_{1,1}}{2}),0).\ \ 
\end{equation}
It is immediately seen that we obtain the same decomposition as in the previous case, i.e. $E(\lambda)=E_0\oplus E_1$ with $k_0=0$, $k_1=1$. This shows that it is sufficient for at least one of the eigenfunctions not to vanish at $P$ in order to recover the same decomposition as in the case where all three eigenfunctions are nonzero at $P$.
\vskip0.2cm
Case 3. $P=0$. Here $u_1(0)=u_2(0)=u_3(0)=0$ and straightforward computation shows that, 
\begin{equation}
    \nabla u_1(0)=(\frac{j_{1,1}}{3},0,0),\ \    \nabla u_2(0)=(0,\frac{j_{1,1}}{3},0),\ \ \nabla u_3(0)=(0,0,\frac{j_{1,1}}{3}).
\end{equation}
Then $dim E(\lambda)=3$ and $E(\lambda)=E_0$ with $k_0=1$ and dim$E_0=3$. More precisely we have that
$$E_0= \operatorname{span}\{u_1,u_2,u_3\}
$$

 \subsection{Known results}
We start recalling a suitable version of the known spectral theorem
\begin{teo}\label{th_spectral}\cite[Proposition 8.20]{helffer2013} Let $T$ be a self-adjoint continuous operator on a Hilbert space $(\mathcal{H},\|\cdot\|)$. Then $$d(\la,\sigma(T))\|x\|\leq \|(T-\la)x\|,$$ for all $x\in D(T)$, where $D(T)$ is the domain of $T$ and $\sigma(T)$ is the spectrum of $T$.  
\end{teo}
We shall apply the previous theorem to the Hilbert space $\mathcal{H}=H_0^1(\O)$ and the operator $-\Delta$.

Next we recall some known estimates on solution of linear elliptic problems. We mainly emphasizes the dependence of the domain in the statement.
	\begin{teo}\label{t1}
		Suppose that $\phi_\e\in H^1_0(\O_\e)$ verifies
		\begin{equation}
			-\Delta\phi_\e=\la_0\phi_\e+f_\e\ in\ \O_\e
		\end{equation}
		with $f_\e\in L^q(\O_\e)$ with $q>\frac N2$.
		Then we have that
		\begin{equation}\label{t2}
			||\phi_\e||_{L^\infty(\O_\e)}\le C \big(||\phi_\e||_{L^2(\O_\e)}+
			||f_\e||_{L^q(\O_\e)}\big),
		\end{equation}
		where we have that $C\le c_0(N)|\O|^p$  for some $p>0$ and any $\e>0$.
	\end{teo}
	\begin{proof}
		It is a classical application of Moser's iteration, see for example Theorem 8.15 in \cite{GT}.
		
        \end{proof}
\begin{lemma}\label{GTlem6.5}
    Let $\O$ be a $C^{2,\alpha}$ domain in $\R^N$, and let $u\in C^{2,\alpha}(\Bar{\O})$ be a solution of 
    \begin{equation*}
        \begin{cases}
          L u=f&\hbox{in}\  \O,\\
          u=0&\hbox{on}\    \partial\O,
        \end{cases}
    \end{equation*}
   where $Lu=\sum_{i,j=1}^N a^{ij}(x)D_{ij}u+\sum_{i=1}^N b^i(x)D_i u+c(x)u$ and  $f\in C^\alpha(\Bar{\O})$. It is assumed that  the coefficients of $L$ satisfy 
    $$
    a^{ij}(x)\xi_i\xi_j\ge \Lambda_1|\xi|^2,\ \forall x\in \O,\ \xi\in\R^N\ \ \hbox{and}\ \ 
    |a^{ij}|_{0,\alpha;\O},\ |b^i|_{0,\alpha;\O},\ |c(x)|_{0,\alpha;\O}\leq \Lambda_2,
    $$
    for some positive constant $\Lambda_1,\ \Lambda_2$.
    Then for some $\delta$ there is a ball $B=B(x_0,\delta)$ at each point $x_0\in \partial\O$ such that 
    \begin{equation}\label{eq0622-1}
        |u|_{2,\alpha;B\cap\O}\leq C(|u|_{0;\O}+|f|_{0,\alpha;\O}),
    \end{equation}
    where $C=C(N,\alpha,\Lambda_1,\Lambda_2,\O)$ and $|u|_{k,\alpha;\O}:=\sum_{j=0}^k|D^j u|_{0,\O}+[D^ku]_{\alpha;\O}$ with 
    $$|D^ju|_{0;\O}=\sup_{|\beta|=j}\sup_{x\in\O}|D^\beta u(x)|,\ \ [D^ku]_{\alpha;\O}=\sup_{|\beta|=k}\sup_{\substack{x,y\in\O\\x\neq y}}\frac{|D^\beta u(x)-D^\beta u(y)|}{|x-y|^\alpha}.$$
\end{lemma}
\begin{proof}
 See   Lemma 6.5 in \cite{GT} for the proof. 
\end{proof}

Now we introduce Lemma \ref{GTlem6.5} to $\O_\e=\O\setminus B(P,\e)$, then the estimates of $u$ around $\partial \O_\e$ is as follows.
\begin{lemma}\label{FurGT}
    Let $\O_\e=\O\setminus B(P,\e)$ and $u_\e\in C^{2,\alpha}(\Bar{\O}_\e) $ be a solution of 
    $$\begin{cases}
        -\Delta u_\e-\lambda_\e u=f_\e \ &in\  \O_\e,\\
        u_\e=0 \  &on\  \partial\O_\e,
    \end{cases}$$
    with $|\lambda_\e|\leq \Lambda$  and  $f_\e\in C^\alpha(\Bar{\O}_\e)$. Then the following estimates hold. \begin{enumerate}
    \item[(i)]
    There exist $\delta>0$ and $A=A(\delta,\Omega)>0$ such that for each $x_0\in \partial\O$, the following estimate holds,
      \begin{equation}\label{eq0304-1}
      |u_\e|_{2,\alpha;B(x_0,\delta)\cap\O}\leq C(|u_\e|_{0;\mathcal{B}_{A\delta}}+|f_\e|_{0,\alpha;\mathcal{B}_{A\delta}}),
 \end{equation}
 where $C=C(N,\alpha,\Lambda,\O)$ and $\mathcal{B}_{A\delta}:=\{x\in\O\ |\ dist(x,\partial\O)\leq A\delta\}$.
  \item[(ii)]  
    There exist  $\delta>0$ and $c=c(N,\alpha,\Lambda)>1$ such that for each $x_0\in \partial B(P,\e)$, the following estimate holds,
    \begin{equation*}
         \e|Du_\e|_{0;B(x_0,\e\delta)\cap \O_\e }+\e^2|D^2u_\e|_{0;B(x_0,\e\delta)\cap \O_\e }\leq C(|u_\e|_{0;\mathcal{O}_\e}+\e^2|f_\e|_{0;\mathcal{O}_\e}+\e^{2+\alpha}[f_\e]_{\alpha;\mathcal{O}_\e}),
    \end{equation*}
    where $C=C(N,\alpha,\Lambda)$ and $\mathcal{O}_\e:=B(P,c\e)\setminus B(P,\e)$.
    \end{enumerate}
\end{lemma}
\begin{proof}
Assertion $(i)$ follows directly from the proof of Lemma \ref{GTlem6.5}.

We now prove $(ii)$.
    Let $\tilde{u}_\e(y):=u_\e(P+\e y)$ and $\tilde{\O}_\e:=\{y\in \R^N \ such \ that \ P+\e y\in \O_\e\}$. Then  $\tilde {u}_\e$ satisfies 
    $$
    \begin{cases}
    -\Delta \tilde{u}_\e-\e^2\lambda_\e\tilde{u}_\e=\e^2\tilde{f}_\e\ &in\ \tilde{\O}_\e,\\
     \tilde{u}_\e=0&on \ \partial \tilde{\O}_\e,\\
     \end{cases}
    $$
    where $\tilde{f}_\e(y)=f_\e(P+\e y)$ and $\partial B(0,1)\subset \partial \tilde{\O}_\e $. Applying $(i)$ to  $\tilde{u}_\e$, it follows that for some $\delta>0$  there exists a ball $B=B(y_0,\delta)$  at each $y_0\in \partial B(0,1)$ such that 
    \begin{equation*}
       |\tilde{u}_\e|_{2,\alpha;B\cap \tilde{\O}_\e} \leq C(|\tilde{u}_\e|_{0;B(0,1+A\delta)\setminus B(0,1)}+|\e^2\tilde{f}_\e|_{0,\alpha;B(0,1+A\delta)\setminus B(0,1)}),
    \end{equation*}
    where $C=C(N,\alpha,\Lambda)$.
  Let $c=1+A\delta$.  Then we obtain that at each $x_0\in \partial B(P,\e)$,
    \begin{equation*}
       \e|Du_\e|_{0;B(x_0,\e\delta)\cap \O_\e }+\e^2|D^2u_\e|_{0;B(x_0,\e\delta)\cap \O_\e }\leq  |\tilde{u}_\e|_{2,\alpha;B\cap \tilde{\O}_\e} \leq C(|u_\e|_{0;\mathcal{O}_\e}+\e^2|f_\e|_{0;\mathcal{O}_\e}+\e^{2+\alpha}[f_\e]_{\alpha;\mathcal{O}_\e}),
    \end{equation*}
    where $\mathcal{O}_\e:=B(P,c\e)\setminus B(P,\e)$.
\end{proof}

 \begin{lemma}\label{LCS}
     Suppose that $\O$ is a bounded smooth domain $\R^2$ and $\varphi$ satisfies that
     \begin{equation*}
         \begin{cases}
             -\Delta \varphi =\lambda\varphi\ & \hbox{in } \O,\\
             \varphi=0 \ &\hbox{on } \partial \O,\\
         \end{cases}
     \end{equation*}
     and $x_0\in \partial\O$. Then $\frac{\partial\varphi}{\partial\nu}(x_0)=0$ iff $x_0\in \mathcal{N}$, where $\mathcal{N}=\overline{\{x\in\O\hbox{ such that }\varphi=0\}}$ and $\frac{\partial\varphi}{\partial\nu}$ is the outnormal derivative of $\varphi$ on the boundary.
 \end{lemma}
\begin{proof}
  See  Lemma 1.2 in \cite{CSLin1987} for the proof.
  \end{proof}

	\begin{lemma}\label{sml}
		Suppose that  $v$ is a solution to
		\begin{equation*}
			\begin{cases}
				-\Delta v=a(x)v&in\ \mathcal{D},\\
				v=0&on\ \partial \mathcal{D},
			\end{cases}  
		\end{equation*}
		with $a\in L^\frac N2(\mathcal{D})$ if $N\ge3$ and $a\in L^\infty(\mathcal{D})$ for $N=2$. Denoting by $S$ the best Sobolev constant we have that if holds
		\begin{equation*}
			\begin{cases}
				||a||_{L^\frac N2(\mathcal{D})}<S,&if\ N\ge3,\\
				||a||_{L^\infty(\mathcal{D})}<\frac{4\pi}{|\mathcal{D}|},&if\ N=2,
			\end{cases}  
		\end{equation*}
		then $v\equiv0$.
	\end{lemma}
	\begin{proof}
		See  \cite[Lemma 2.5]{GrossiMolle} for the case $N\ge3$ and   \cite[Theorem 4.1]{DeCarliHudson} for the case $N=2$. 
	\end{proof}

	\section{Improved estimates on the $u$-capacitary potential}\label{sec-Vu}
In this section we prove Theorem \ref{intr-T1} and Corollaries \ref{T2'}, \ref{T2}.

From now on, we will assume that $f(\e,x)=O(\e^lg(x) )$ means
$$|f(\e,x)|\le C\e^l|g(x)|$$
where the constant $C$ does not depend on $x$.
    
	Given a $L^2$- normalized eigenfunction $u$ to the problem \eqref{eq:eigenvalueproblem},  assume that the vanishing order of $u$ at $P$ is $k$
     (see Definition \ref{defvanishing}).

    Recalling the definition of $u$-capacity potential $V_{\e,u}$ of  $B(P,\e)$ in Definition \ref{intr-def-Cap}, $V_{\e,u}$ is the unique solution to the problem
	\begin{equation}\label{eq:Veproblem}
		\begin{cases}
			\Delta V_{\e,u} =0 &\text{in }\Omega\setminus B(P,\e)\\
			V_{\e,u} =0 &\text{on }\partial \Omega\\
			V_{\e,u} =u &\text{on }\overline{ B(P,\e)},
		\end{cases}
	\end{equation}
 and its energy $\int_{\Omega} |\nabla V_{\e,u}|^2\,dx$  is called $u$-capacity of $B(P,\e)$ with respect to the domain $\Omega$. 
	 Let us recall (e.g., see \cite[Theorem 1.13]{afhl} for $N=2$ and \cite[Theorem 1.6]{alm2} for $N\ge3$)  that there exists a positive constant $ C$  such that 
  \begin{equation}\label{capu}
         \int_{\Omega} |\nabla V_{\e,u}|^2\,dx=
         \begin{cases}
           \displaystyle \frac{ C\,(1+o(1))}{|\log\e|} , &if \ N=2\ and\ k=0,\\
           \displaystyle C\,\e^{N-2+2k}(1+o(1)),& otherwise,
         \end{cases}
    \end{equation}
    and (see \cite[Lemma A.1]{afhl})
	\begin{equation*}
	||V_{\e,u}||^2_{L^2(\Omega)} = o\left( \int_{\Omega} |\nabla V_{\e,u}|^2\,dx \right),
	\end{equation*}
    obtaining
     \begin{equation}\label{eq:normaL2Ve}
         ||V_{\e,u}||^2_{L^2(\Omega)}=
         \begin{cases}
           \displaystyle o\left(\frac{1}{|\log\e|}\right) , &if \ N=2\ and\ k=0,\\
           \displaystyle o\left(\e^{N-2+2k}\right),& otherwise.
         \end{cases}
    \end{equation}

	Let us denote by $G_\O(x,y)$ the Green function of $-\Delta$ with zero Dirichlet boundary conditions. We have the classical decomposition,
	\begin{equation}\label{green}
		G_\O(x,y)=H_\O(x,y)+ 
		\begin{cases}
			\frac1{N(N-2)\omega_N}\frac1{|x-y|^{N-2}},&if\ N\ge3,\\
			-\frac1{2\pi}\log|x-y|,& if\  N=2.
		\end{cases}
	\end{equation}
	Here $H_\O(x,y)$ is the regular part of the Green function which is known to be bounded and harmonic in $\O$ in $x$ and $y$. Next lemma links the expansion of 
$V_{\e,u}$ with the Green function $G_\O(x,y)$.

\begin{lemma}\label{l:improvedL2Ve}
 For $N\ge 3$, let 	$A(N,0)=N(N-2)\omega_N$ and 
			$A(N,k)=\frac{N(N-2)\omega_N}{(N-2)N\cdots(N-4+2k)}$, for any $k\ge1$. 
			We have that
				\begin{equation}\label{V1}
				\!\!\!\!\!\!\!\!\!\!\!\!V_{\e,u}(x)=A(N,k)\e^{N-2+2k}
					\sum_{|\alpha|=k}\frac1{\alpha!}
					\left(\frac{\partial^ku(P)}{\partial x_1^{\alpha_1}\cdots\partial x_N^{\alpha_N}}+O(\e)\right)
					\left(\frac{\partial^kG_{\Omega}(x,P)}{\partial y_1^{\alpha_1}\cdots\partial y_N^{\alpha_N}}+O(\e)\right)+O(\e^{N+k})G_{\Omega}(x,P),
				\end{equation}
			as $\e\to0$,	uniformly for $x\in\O_\e$.
			Furthermore, we obtain
			\begin{equation}\label{V3}
				\begin{aligned}
				&\hbox{For $k=0$,}\quad
					||V_{\e,u}||_{L^p(\O_\e)}=
					\begin{cases}
					\displaystyle	O(\e^{N-2}),  &if\ 1\le p<\frac N{N-2},\\
						O(\e^{N-2}|\log\e|), &if\  p=\frac{N}{N-2},\\
						O(\e^{\frac{N}{p}}),   &if\  p>\frac{N}{N-2}.\\
					\end{cases}\\
				&\hbox{For $k=1$,}\quad
					||V_{\e,u}||_{L^p(\O_\e)}=
					\begin{cases}
						O(\e^{N}),  &if\ 1\le p<\frac N{N-1},\\
						O(\e^{N}|\log\e|), &if\  p=\frac{N}{N-1},\\
						O(\e^{\frac{N}{p}+1}),   &if\  p>\frac{N}{N-1}.\\
					\end{cases}\\
			&	\hbox{For $k= 2$,}\quad
					||V_{\e,u}||_{L^p(\O_\e)}=
					\begin{cases}
						O(\e^{N+2}|\log\e|), &if\  p=1,\\
						O(\e^{\frac{N}{p}+2}),   &if\  p>1.\\
					\end{cases}\\
		&	\hbox{For $k\ge 3$,}\quad
					||V_{\e,u}||_{L^p(\O_\e)}=	O(\e^{\frac{N}{p}+k}),\quad if \ p\ge 1.
				\end{aligned}
			\end{equation}
		\end{lemma}
		\begin{proof}[\bf Proof of \eqref{V1}.]
			Let us show that the function 
            \begin{equation}\label{eqwe}
				\begin{aligned}
					w_\e=&A(N,k)\e^{N-2+2k}	\sum_{|\alpha|=k}\frac1{\alpha!}
					\frac{\partial^ku(P)}{\partial x_1^{\alpha_1}\cdots\partial x_N^{\alpha_N}}
					\frac{\partial^kG_{\Omega}(x,P)}{\partial y_1^{\alpha_1}\cdots\partial y_N^{\alpha_N}}\\
					&+A(N,k+1)\e^{N+2k}	\sum_{|\alpha|=k+1}\frac1{\alpha!}
					\frac{\partial^{k+1}u(P)}{\partial x_1^{\alpha_1}\cdots\partial x_N^{\alpha_N}}
					\frac{\partial^{k+1}G_{\Omega}(x,P)}{\partial y_1^{\alpha_1}\cdots\partial y_N^{\alpha_N}}
				\end{aligned}      
            \end{equation}
			verifies
			\begin{equation}
				\begin{cases}
					\Delta w_\e=0&in\ \O_\e,\\
					w_\e=0&on\ \partial\O.
				\end{cases}
			\end{equation}
			  Let $Q_\alpha(x,P)$ be a polynomial of degree $k$ depending on $\alpha=(\alpha_1,\cdots,\alpha_N)$ such that \eqref{eq0712-1} holds,
				\begin{equation}\label{eq0712-1}
					\frac{\partial^k\left(|x-y|^{2-N}\right)}{\partial y_1^{\alpha_1}\cdots\partial y_N^{\alpha_N}}\Bigg|_{y=P}=C(N,k)\frac{(x_1-P_1)^{\alpha_1}\cdots(x_N-P_N)^{\alpha_N}+Q_\alpha(x,P)}
					{|x-P|^{N-2+2k}},
				\end{equation}
			where  $C(N,0)=1$, $C(N,k)=(N-2)N\cdots(N-4+2k)$ for any $k\ge 1$. 
            Since the vanishing order of $u$ at $P$ is $k$, it follows from Corollary \ref{Cor-Q} in the Appendix that for any $k\ge 0$
             \begin{equation}\label{eq1704-1}
					\sum_{|\alpha|=k}\frac1{\alpha!}
					\frac{\partial^ku(P)}{\partial x_1^{\alpha_1}\cdots\partial x_N^{\alpha_N}}Q_\alpha(x,P)=0,
				\end{equation}   
                and 
                \begin{equation}\label{eq1704-2}
					\sum_{|\alpha|=k+1}\frac1{\alpha!}
					\frac{\partial^{k+1}u(P)}{\partial x_1^{\alpha_1}\cdots\partial x_N^{\alpha_N}}Q_\alpha(x,P)=0.
				\end{equation}   
                     Therefore, for any $k\ge 0$,
                     by \eqref{eqwe} we get
                \begin{align*}
					\!\!\!\!\!\!\!\!\!\!\!\!\!\!\!\!\!\!\!\!\!
                    w_\e(x)=&A(N,k)\e^{N-2+2k}
					\sum_{|\alpha|=k}\frac1{\alpha!}
					\frac{\partial^ku(P)}{\partial x_1^{\alpha_1}\cdots\partial x_N^{\alpha_N}}
					\frac{\partial^k\left(\frac1{N(N-2)\omega_N}|x-y|^{2-N}+H_{\Omega}(x,y)\right)}{\partial y_1^{\alpha_1}\cdots\partial y_N^{\alpha_N}}\Bigg|_{y=P}\\
                    &+A(N,k+1)\e^{N+2k}
					\sum_{|\alpha|=k+1}\frac1{\alpha!}
					\frac{\partial^{k+1}u(P)}{\partial x_1^{\alpha_1}\cdots\partial x_N^{\alpha_N}}
					\frac{\partial^{k+1}\left(\frac1{N(N-2)\omega_N}|x-y|^{2-N}+H_{\Omega}(x,y)\right)}{\partial y_1^{\alpha_1}\cdots\partial y_N^{\alpha_N}}\Bigg|_{y=P}\\
                    =&\frac{\e^{N-2+2k}}{|x-P|^{N-2+2k}}\sum_{|\alpha|=k}\frac1{\alpha!}
					\frac{\partial^ku(P)}{\partial x_1^{\alpha_1}\cdots\partial x_N^{\alpha_N}}(x_1-P_1)^{\alpha_1}\cdots(x_N-P_N)^{\alpha_N}+O(\e^{N-2+2k})\\
                    &+\frac{\e^{N+2k}}{|x-P|^{N+2k}}\sum_{|\alpha|=k+1}\frac1{\alpha!}
					\frac{\partial^{k+1}u(P)}{\partial x_1^{\alpha_1}\cdots\partial x_N^{\alpha_N}}(x_1-P_1)^{\alpha_1}\cdots(x_N-P_N)^{\alpha_N}+O(\e^{N+2k}).\\
				\end{align*}
                For $|x-P|=\e$, we have
                \begin{equation}\label{eq11102025-1}
                    \begin{aligned}
                    w_\e(x)=&\sum_{|\alpha|=k}\frac1{\alpha!}
					\frac{\partial^ku(P)}{\partial x_1^{\alpha_1}\cdots\partial x_N^{\alpha_N}}(x_1-P_1)^{\alpha_1}\cdots(x_N-P_N)^{\alpha_N}\\
                    &+\sum_{|\alpha|=k+1}\frac1{\alpha!}
					\frac{\partial^{k+1}u(P)}{\partial x_1^{\alpha_1}\cdots\partial x_N^{\alpha_N}}(x_1-P_1)^{\alpha_1}\cdots(x_N-P_N)^{\alpha_N}+O(\e^{N-2+2k}).
                    \end{aligned}
                \end{equation}
Since $V_{\e,u}(x)=u(x)$ on $\partial B(P,\e)$, by using Taylor expansion for $x\in\partial B(P,\e)$ and combining with \eqref{eq11102025-1}, we obtain
\begin{equation}\label{eq1901-1}
			\begin{aligned}
				&V_{\e,u}(x)-w_\e(x)=u(x)-w_\e(x)\\
				=&\sum_{|\alpha|=k}\frac1{\alpha!}
					\frac{\partial^ku(P)}{\partial x_1^{\alpha_1}\cdots\partial x_N^{\alpha_N}}(x_1-P_1)^{\alpha_1}\cdots(x_N-P_N)^{\alpha_N}\\
                    &+\sum_{|\alpha|=k+1}\frac1{\alpha!}
					\frac{\partial^{k+1}u(P)}{\partial x_1^{\alpha_1}\cdots\partial x_N^{\alpha_N}}(x_1-P_1)^{\alpha_1}\cdots(x_N-P_N)^{\alpha_N}+O(\e^{k+2})\\
				&-\sum_{|\alpha|=k}\frac1{\alpha!}
					\frac{\partial^ku(P)}{\partial x_1^{\alpha_1}\cdots\partial x_N^{\alpha_N}}(x_1-P_1)^{\alpha_1}\cdots(x_N-P_N)^{\alpha_N}\\
                    &-\sum_{|\alpha|=k+1}\frac1{\alpha!}
					\frac{\partial^{k+1}u(P)}{\partial x_1^{\alpha_1}\cdots\partial x_N^{\alpha_N}}(x_1-P_1)^{\alpha_1}\cdots(x_N-P_N)^{\alpha_N}+O(\e^{N-2+2k})\\
                    =&O(\e^{k+2})+O(\e^{N-2+2k}).
			\end{aligned}
            \end{equation}
		If we directly estimate $V_{\e,u}-w_\e$ by the maximum principle for the harmonic function, the estimation is a little rough. So we use the Green function.\\ On $|x-P|=\e$, $(\e^{N+k}+\e^{2(N-2)+2k})G_{\Omega}(x,P)=\frac1{N(N-2)\omega_N}(\e^{k+2}+\e^{N-2+2k})+O(\e^{N+k})+O(\e^{2(N-2)+2k})$. Hence we can select a large $C>0$ such that   
				\begin{equation}
					\begin{cases}
						\Delta\Big[V_{\e,u}-w_\e-C(\e^{N+k}+\e^{2(N-2)+2k})G_{\Omega}(x,P)\Big]=0&in\ \O_\e,\\
						V_{\e,u}-w_\e-C(\e^{N+k}+\e^{2(N-2)+2k})G_{\Omega}(x,P)=0&on\ \partial\O,\\
						V_{\e,u}-w_\e-C(\e^{N+k}+\e^{2(N-2)+2k})G_{\Omega}(x,P)\le 0&on\ |x-P|=\e,
					\end{cases}
				\end{equation}
				and so by the maximum principle for harmonic function we get
				\begin{equation}
					V_{\e,u}-w_\e \le C(\e^{N+k}+\e^{2(N-2)+2k})G_{\Omega}(x,P)\quad\hbox{in }\O_\e,
				\end{equation}
				and since the same holds for the reverse inequality, we deduce that \begin{equation}\label{eq-0922-1}
					|	V_{\e,u}-w_\e |\leq C(\e^{N+k}+\e^{2(N-2)+2k})G_{\Omega}(x,P)\quad\hbox{in }\O_\e.
				\end{equation}
                Actually,
                 \begin{align*}
					&A(N,k+1) \e^{N+2k}	\sum_{|\alpha|=k+1}\frac1{\alpha!}
					\frac{\partial^{k+1}u(P)}{\partial x_1^{\alpha_1}\cdots\partial x_N^{\alpha_N}}
					\frac{\partial^{k+1}G_{\Omega}(x,P)}{\partial y_1^{\alpha_1}\cdots\partial y_N^{\alpha_N}}\\
                    =&\frac{\e^{N+2k}}{|x-P|^{N+2k}}\sum_{|\alpha|=k+1}\frac1{\alpha!}
					\frac{\partial^{k+1}u(P)}{\partial x_1^{\alpha_1}\cdots\partial x_N^{\alpha_N}}(x_1-P_1)^{\alpha_1}\cdots(x_N-P_N)^{\alpha_N}+O(\e^{N+2k})\\
                    =& O(\e)\frac{\e^{N-2+2k}}{|x-P|^{N-2+2k}}\sum_{|\alpha|=k}(x_1-P_1)^{\alpha_1}\cdots(x_N-P_N)^{\alpha_N}+O(\e^{N+2k})\\
                    =&O(\e)\e^{N-2+2k}	\sum_{|\alpha|=k}
					\frac{\partial^k\Big(G_{\Omega}(x,P)-CH_{\Omega}(x,P)\Big)}{\partial y_1^{\alpha_1}\cdots\partial y_N^{\alpha_N}}+O(\e^{N+2k})\\
                    =&O(\e)\e^{N-2+2k}	\sum_{|\alpha|=k}
					\frac{\partial^kG_{\Omega}(x,P)}{\partial y_1^{\alpha_1}\cdots\partial y_N^{\alpha_N}}+O(\e^{N-1+2k}).\\
                \end{align*}
                This implies that 
                 \begin{equation}\label{eq11102025-2}
                 \begin{aligned}
                   V_{\e,u}(x)=&A(N,k)\e^{N-2+2k}
					\sum_{|\alpha|=k}\frac1{\alpha!}
					\left(\frac{\partial^ku(P)}{\partial x_1^{\alpha_1}\cdots\partial x_N^{\alpha_N}}
                    +O(\e)\right)
					\left(\frac{\partial^kG_{\Omega}(x,P)}{\partial y_1^{\alpha_1}\cdots\partial y_N^{\alpha_N}}
                    +O(\e)\right)\\
                    &+O(\e^{N+k}+\e^{2(N-2)+2k})G_{\Omega}(x,P).  
                 \end{aligned}
				\end{equation}
                As final observation, we can see that $2(N-2)+2k<N+k$ if and only if $N=3$ and $k=0$. 
               In this case, the term $O(\e^2)G_\O(x,P)$ can be absorbed in the main term $A(3,0)\e (u(P)+O(\e))(G_\O(x,P)+O(\e))$.
                Except in this case, the term $O(\e^{2(N-2)+2k})$
                 is of lower order than $O(\e^{N+k})$, and therefore we have
   \begin{equation}\label{eq11102025-3}
					V_{\e,u}(x)=A(N,k)\e^{N-2+2k}
					\sum_{|\alpha|=k}\frac1{\alpha!}
					\left(\frac{\partial^ku(P)}{\partial x_1^{\alpha_1}\cdots\partial x_N^{\alpha_N}}+O(\e)\right)
					\left(\frac{\partial^kG_{\Omega}(x,P)}{\partial y_1^{\alpha_1}\cdots\partial y_N^{\alpha_N}}+O(\e)\right)+O(\e^{N+k})G_{\Omega}(x,P).
				\end{equation}
				{\bf Proof of \eqref{V3} .}\\
                By \eqref{eq11102025-3}, we have
				\begin{equation}
                \begin{aligned}
					||V_{\e,u}||_{L^p(\O_\e)}^p\le& C\e^{(N-2+2k)p}\int_{\O_\e}\sum_{|\alpha|=k}\Big|\frac{\partial^kG_{\Omega}(x,P)}{\partial y_1^{\alpha_1}\cdots\partial y_N^{\alpha_N}}\Big|^p +C\e^{(N+k)p}
					\int_{\O_\e}|G_{\Omega}(x,P)|^p\\
                    \le & C\e^{(N-2+2k)p}\int_{\O_\e} \frac{1}{|x-P|^{(N-2+k)p}}+C\e^{(N+k)p}
					\int_{\O_\e}\frac1{|x-P|^{(N-2)p}}=:I+J.\\
                     \end{aligned}
				\end{equation}
				 We estimate the terms $I$ and $J$ separately, the estimate \eqref{V3} it will follow from the combination of the two. For $k=0$,
				\begin{equation}
					I=
					\begin{cases}
						O(\e^{(N-2)p}),  &if\ 1\le p<\frac N{N-2},\\
						O(\e^{(N-2)p}|\log\e|), &if\  p=\frac{N}{N-2},\\
						O(\e^{N}),   &if\  p>\frac{N}{N-2}.\\
					\end{cases}
				\end{equation}
				For $k=1$,
				\begin{equation}
					I=
					\begin{cases}
						O(\e^{Np}),  &if\ 1\le p<\frac N{N-1},\\
						O(\e^{Np}|\log\e|), &if\  p=\frac{N}{N-1},\\
						O(\e^{N+p}),   &if\  p>\frac{N}{N-1}.\\
					\end{cases}
				\end{equation}
				For $k= 2$,
				\begin{equation}
					I=
					\begin{cases}
						O(\e^{(N+2)p}|\log\e|), &if\  p=1,\\
						O(\e^{N+2p}),   &if\  p>1.\\
					\end{cases}
				\end{equation}
				For $k\ge 3$,
				\begin{equation}
					I=	O(\e^{N+kp}),\quad if \ p\ge 1.
				\end{equation}
                For $k\ge 0$, 
				\begin{equation}
					J=
					\begin{cases}
						O(\e^{(N+k)p}),&if\ 1\le p<\frac N{N-2},\\
						O(\e^{(N+k)p}|\log\e|),&if\  p=\frac N{N-2},\\
						O(\e^{N+(k+2)p}),&if\ p>\frac N{N-2}.
					\end{cases}
				\end{equation}
                Then combining the above equations, we complete the proof of  \eqref{V3}.
				\end{proof} 	   	 
			In next lemma we consider the case $N=2$
		\begin{lemma}\label{2l:improvedL2Ve}
			 For $N=2$, let $A(N,0)=A(N,1)=2\pi$ and $A(N,k)=\frac{2\pi}{2\times \cdots\times2(k-1)}$ for any $k\ge 2$.
               We have that
				\begin{equation}\label{eq10102025-1}
				\hbox{if } k=0,\quad	V_{\e,u}(x)=A(N,0)\frac{1}{|\log \e|}u(p)G_\Omega(x,P)+O\left(\frac{1}{|\log \e|^2}\right)G_\Omega(x,P),
				\end{equation}
                \begin{equation}\label{eq10102025-2}
				\hbox{if } k\ge 1,\quad 	V_{\e,u}(x)=A(N,k)\e^{2k}
					\sum_{|\alpha|=k}\frac1{\alpha!}
					\left(\frac{\partial^ku(P)}{\partial x_1^{\alpha_1}\partial x_2^{\alpha_2}}+O(\e)\right)
					\left(\frac{\partial^kG_{\Omega}(x,P)}{\partial y_1^{\alpha_1}\partial y_2^{\alpha_2}}+O(\e)\right)+O\left(\frac{\e^{2+k}}{|\log \e|}\right)G_{\Omega}(x,P),
				\end{equation}
				as $\e\to 0$, uniformly for $x\in\O_\e$. Furthermore we have that
\begin{equation}\label{eq10102025-3}
				\begin{aligned}
				&\hbox{For $k=0$,}\quad
					||V_{\e,u}||_{L^p(\O_\e)}=O\left(\frac{1}{|\log \e|}\right), \quad if\ p\ge 1.
					\\
				&\hbox{For $k=1$,}\quad
					||V_{\e,u}||_{L^p(\O_\e)}=
					\begin{cases}
						O(\e^{2}),  &if\ 1\le p<2,\\
						O(\e^{2}|\log\e|), &if\  p=2,\\
						O(\e^{\frac{2}{p}+1}),   &if\  p>2.\\
					\end{cases}\\
			&	\hbox{For $k= 2$,}\quad
					||V_{\e,u}||_{L^p(\O_\e)}=
					\begin{cases}
						O(\e^{4}|\log\e|), &if\  p=1,\\
						O(\e^{\frac{2}{p}+2}),   &if\  p>1.\\
					\end{cases}\\
		&	\hbox{For $k\ge 3$,}\quad
					||V_{\e,u}||_{L^p(\O_\e)}=	O(\e^{\frac{2}{p}+k}),\quad if \ p\ge 1.
				\end{aligned}
			\end{equation}

            \end{lemma}
	\begin{proof}[\bf Proof of \eqref{eq10102025-1}.]
    First we consider the case of $k=0$.
	    Let 
        \begin{equation}
             \hat{w}_\e(x)=A(N,0)\frac{1}{|\log \e|}u(p)G_\Omega(x,P).
        \end{equation}
        We observe that 
       $$
       \begin{cases}
           \Delta \hat{w}_\e=0&in\ \O_\e,\\
					\hat{w}_\e=0&on\ \partial\O,\\
             \hat{w}_\e(x)=u(P)+O\left(\frac{1}{|\log \e|}\right)&on\ |x-P|=\e,
       \end{cases}
       $$
        and 
        \begin{equation}
            V_{\e,u}(x)=u(x)=u(P)+O(\e).
        \end{equation}
        By choosing $C>0$ large enough, it follows that
        \begin{equation}
            \begin{cases}
                 \Delta\left(V_{\e,u}- \hat{w}_\e-C\frac{1}{|\log \e|^2}G_\Omega(x,P)\right)=0&in\ \O_\e,\\
					V_{\e,u}- \hat{w}_\e-C\frac{1}{|\log \e|^2}G_\Omega(x,P)=0&on\ \partial\O,\\
                    V_{\e,u}- \hat{w}_\e-C\frac{1}{|\log \e|^2}G_\Omega(x,P)\leq 0 & on \ |x-P|=\e.
            \end{cases}
        \end{equation}
        So by the maximum principle for the harmonic function we get
        \begin{equation}\label{eq11102025-4}
             V_{\e,u}- \hat{w}_\e-C\frac{1}{|\log \e|^2}G_\Omega(x,P)\leq 0.
        \end{equation}
        Similarly, we can obtain  
         \begin{equation}\label{eq11102025-5}
             V_{\e,u}- \hat{w}_\e+C\frac{1}{|\log \e|^2}G_\Omega(x,P)\ge 0.
        \end{equation}
        Combining \eqref{eq11102025-4} and \eqref{eq11102025-5}, we prove the equation \eqref{eq10102025-1}.
        Then, by direct computations, for $k=0$, $||V_{\e,u}||_{L^p(\O_\e)}=O\left(\frac{1}{|\log \e|}\right) $, if $ \ p\ge 1.$
        
{\bf Proof of \eqref{eq10102025-2} and \eqref{eq10102025-3}.}\\
For $k\ge1$, we define 
  \begin{equation}
				\begin{aligned}
					\hat{w}_\e(x)=&A(N,k)\e^{2k}	\sum_{|\alpha|=k}\frac1{\alpha!}
					\frac{\partial^ku(P)}{\partial x_1^{\alpha_1}\partial x_2^{\alpha_2}}
					\frac{\partial^kG_{\Omega}(x,P)}{\partial y_1^{\alpha_1}\partial y_2^{\alpha_2}}\\
					&+A(N,k+1)\e^{2(k+1)}	\sum_{|\alpha|=k+1}\frac1{\alpha!}
					\frac{\partial^{k+1}u(P)}{\partial x_1^{\alpha_1}\partial x_2^{\alpha_2}}
					\frac{\partial^{k+1}G_{\Omega}(x,P)}{\partial y_1^{\alpha_1}\partial y_2^{\alpha_2}}.
				\end{aligned}      
            \end{equation}
                   For $N=2$, define $Q_\alpha(x,y)$ as a polynomial of degree $k$ depending  on $\alpha_1$ and $\alpha_2$ such that \eqref{eq10102025-5} holds,
                    \begin{equation}\label{eq10102025-5}
                        \frac{\partial^k(-\log|x-y|)}{\partial y_1^{\alpha_1}\partial y_2^{\alpha_2}} = C(N,k)\frac{(x_1 - y_1)^{\alpha_1}(x_2 - y_2)^{\alpha_2}+Q_\alpha(x,y)}{|x - y|^{2k}},
                    \end{equation}
                    where $C(N,0)=C(N,1)=1$ and $C(N,k)=2\times\cdots\times2(k-1)$ for any $k\ge 2$.
                   Following  the same computation in Theorem \ref{l:improvedL2Ve}, we obtain \eqref{eq10102025-2} and \eqref{eq10102025-3}. 
	\end{proof}
	{\bf Proof of Lemma \ref{into-l:improvedL2Ve}} It is a direct consequence of Lemma \ref{2l:improvedL2Ve} and Lemma \ref{l:improvedL2Ve}.
\vskip0.2cm\noindent
Next we prove some estimates which will be important in the rest of the paper.

       \begin{lemma}\label{lem2810}
      Assume that  the vanishing order of an eigenfunction $u_l$ at $P$ is equal to $k$. Then
\begin{equation*}
\e\frac{\partial V_{\e,u_l}}{\partial \nu}=O(\e^{k}) \ \ \text{on }\ \partial B(P,\e).
\end{equation*}
\end{lemma}
\begin{proof}
      Since $V_{\e,u_l}$ satisfies \eqref{eq:Veproblem} with $u=u_l$ it is easy to see that
 \begin{equation*}
     	\begin{cases}
			\Delta (V_{\e,u_l}-u_l) =\lambda u_l &\text{in }\Omega_\e,\\
			V_{\e,u_l}-u_l=0 &\text{on }\partial \Omega,\\
			V_{\e,u_l}-u_l =0 &\text{on }\overline{ B(P,\e)}.
		\end{cases}
 \end{equation*}
 Using Lemma \ref{FurGT},  there exists $c=c(N,\alpha)>1$ such that, setting $\mathcal{O}_\e:=B(P,c\e)\setminus B(P,\e)$, we have
 \begin{equation}\label{eq270326}
    \e |D(V_{\e,u_l}-u_l)|_{0;\partial B(P,\e)}\leq C(|V_{\e,u_l}-u_l|_{0;\mathcal{O}_\e}+\e^2 |u_l|_{0;\mathcal{O}_\e}+\e^{2+\alpha}[u_l]_{\alpha;\mathcal{O}_\e}),
 \end{equation}
 where, for any set $\mathcal{D}$, $|D^ju|_{0;\mathcal{D}}$ and $[D^ku]_{\alpha;\mathcal{D}}$ are defined as in Lemma \ref{GTlem6.5}.
 Since the  vanishing order of $u_l$ at $P$ is  $k$,  it follows from \eqref{eq270326}, together with Lemma \ref{l:improvedL2Ve} and Lemma \ref{2l:improvedL2Ve}, that 
 \begin{equation*}
     \e|DV_{\e,u_l}|_{0;\partial B(P,\e)}\leq \e|Du_l|_{0;\partial B(P,\e)}+\e |D(V_{\e,u_l}-u_l)|_{0;\partial B(P,\e)}=O(\e^k).
 \end{equation*}
 The proof is complete.
    \end{proof}
    \begin{lemma}\label{lem1211-1}
        Assume that $u_i$ and $u_l$ vanish with order $k_I$ and $k_L$, respectively.  Then the following results hold.
         \begin{itemize}
             \item [$(1)$] Assume that either $N\ge 3$ or $N=2$ with $ k_I+k_L\ge 1$. Then
              \begin{equation}\label{eq1505-1}
    \begin{aligned}
        &\int_{\O_\e}u_i(x) V_{\e,u_l}(x)\ dx\\
           =&\frac{1}{\lambda }\e^{N-2+{ k_I}+{ k_L}}	\delta_{{ k_I}}^{{ k_L}}\sum_{\substack{|\alpha|={ k_I}\\
           |\beta|={ k_L}}}\frac{1}{\alpha!\beta!}\frac{\partial^{{ k_I}} u_i (P)}{\partial x_{1}^{\alpha_1}\cdots\partial x_{N}^{\alpha_N}}\frac{\partial^{{ k_L}} u_l (P)}{\partial x_{1}^{\beta_1}\cdots\partial x_{N}^{\beta_N}} C_{\alpha\beta} +O(\e^{N-1+{ k_I}+{ k_L}}),\\
              \end{aligned}
    \end{equation}
      where $\delta_{{ k_I}}^{{ k_L}}=\begin{cases}
            1 &if\ k_I=k_L\\
            0 &if\ k_I\neq k_L
        \end{cases}$ and $C_{\alpha\beta}$ is defined in \eqref{defCalbe}.
    \item[$(2)$] If $N=2$ and ${ k_I}={ k_L}=0$,
    \begin{equation}\label{eq0301-1}
        \int_{\O_\e}u_i(x) V_{\e,u_l}(x)\ dx=\frac{2\pi}{\lambda|\log \e|}u_i(P)u_l(P)+O\left(\frac{1}{|\log \e|^2}\right).
    \end{equation}
    
         \end{itemize}
    \end{lemma}
        \begin{proof}
       If  $N\ge3$ and  ${ k_I}={ k_L}=0$, using  \eqref{V1}, we have 
        \begin{equation}
            \begin{aligned}
             & \int_{\O_\e}u_i(x) V_{\e,u_l}(x)\ dx\\
              =  &\int_{\O_\e}u_i(x)\left[N(N-2)\o_N\e^{N-2}\big(u_l(P)+O(\e)\big)\big(G_\O(x,P)+O(\e)\big)+O(\e^NG_\O(x,P))\right]\ dx\\
              =&\frac{N(N-2)\o_N}{\lambda}u_i(P)u_l(P)\e^{N-2}+O(\e^{N-1}),\\
            \end{aligned}
        \end{equation}
        where we use the Green representation formula that $u_i(P)=\lambda\int_{\O}u_i(x)G_\O(x,P)\ dx$. 
        
        If ${ k_I}=0$ and ${ k_L}\ge 1$, by  \eqref{V3}, we see that 
        \begin{equation}
            \int_{\O_\e}u_i(x) V_{\e,u_l}(x)\ dx=O\left(||V_{\e,u_l}||_{L^1(\O_\e)}\right)=O(\e^{N-1+{ k_L}}).
        \end{equation}
        The previous computations show the claim when ${ k_I}=0$. 
        
        In the following, we assume that ${ k_I}\ge1$.
        \begin{equation}\label{eq27102025-0}
            \begin{aligned}
                \int_{\O_\e}u_i(x) V_{\e,u_l}(x)\ dx&=\frac{1}{\la}\int_{\O_\e}-\Delta u_i(x) V_{\e,u_l}(x)\ dx\\
                 &=\frac{1}{\la}\int_{\partial B(P,\e)} \frac{\partial u_i (x)}{\partial \nu} u_l(x)\ d S_x-\frac{1}{\la}\int_{\partial B(P,\e)} u_i (x)\frac{\partial V_{\e,u_l}(x) }{\partial \nu}\ d S_x\\
                 &=:I_1+I_2.\\
            \end{aligned}
        \end{equation}

        Since $u_i$ and $u_l$ vanish with order $k_I$ and $k_L$, respectively,
         using Taylor expansion on $\partial B(P,\e)$, we have
        \begin{equation}\label{eq27102025-1}
        \begin{aligned}
u_l(x)&=\sum_{|\alpha|={ k_L}}\frac1{\alpha!}\frac{\partial^{{ k_L}}u_l(P)}{\partial x_1^{\alpha_1}\cdots\partial x_N^{\alpha_N}}(x_1-P_1)^{\alpha_1}\cdots(x_N-P_N)^{\alpha_N}+O(\e^{{ k_L}+1})\\
       \end{aligned}
        \end{equation}
        and 
        \begin{equation}\label{eq27102025-2}
        \begin{aligned}
        \frac{\partial u_i (x)}{\partial \nu}=&\sum_{s=1}^N \frac{\partial u_i (x)}{\partial x_s}\frac{x_s-P_s}{\e}\\
        =&\sum_{s=1}^N\sum_{|\beta|={ k_I}-1}\frac1{\beta!}\frac{\partial^{{ k_I}-1}\frac{\partial u_i }{\partial x_s}(P)}{\partial x_1^{\beta_1}\cdots\partial x_N^{\beta_N}}(x_1-P_1)^{\beta_1}\cdots(x_N-P_N)^{\beta_N}\frac{x_s-P_s}{\e}+O(\e^{{ k_I}})\\
        =&\frac{ k_I}{\e}\sum_{|\beta|=k_I}\frac{1}{\beta!}\frac{\partial^{{ k_I}} u_i (P)}{\partial x_1^{\beta_1}\cdots\partial x_N^{\beta_N}}(x_1-P_1)^{\beta_1}\cdots(x_N-P_N)^{\beta_N}+O(\e^{{ k_I}}).
         \end{aligned}
        \end{equation}
        A  directly computation then shows that  
        \begin{equation}\label{eq05112025-1}
        \begin{aligned}
           I_1= \frac{ k_I}{\lambda(N-2+k_I+k_L)}\e^{N-2+k_I+k_L}\sum_{\substack{|\alpha|={ k_I}\\
           |\beta|={ k_L}}}\frac{1}{\alpha!\beta!}\frac{\partial^{{ k_I}} u_i (P)}{\partial x_{1}^{\alpha_1}\cdots\partial x_{N}^{\alpha_N}}\frac{\partial^{{ k_L}} u_l (P)}{\partial x_{1}^{\beta_1}\cdots\partial x_{N}^{\beta_N}} C_{\alpha\beta}+O(\e^{N-1+{ k_I}+{ k_L}}).
            \end{aligned}
        \end{equation}
        
          For $I_2$, using \eqref{eqwe} and \eqref{eq1901-1}, we see that for $x\in \partial B(P,\e)$ 
        $$
       u_i(x)= A(N,{ k_I})\e^{N-2+2{ k_I}}	\sum_{|\alpha|={ k_I}}\frac1{\alpha!}
					\frac{\partial^{{ k_I}}u_i(P)}{\partial x_1^{\alpha_1}\cdots\partial x_N^{\alpha_N}}
					\frac{\partial^{{ k_I}}G_{\Omega}(x,P)}{\partial y_1^{\alpha_1}\cdots\partial y_N^{\alpha_N}}+O(\e^{{ k_I}+1}).
        $$
        Combining  with Lemma \ref{lem2810}, we get
        \begin{equation}\label{eq27102025-3}
        \begin{aligned}
            I_2&=-\frac{1}{\la}\int_{\partial B(P,\e)} \left[ A(N,{ k_I})\e^{N-2+2{ k_I}}	\sum_{|\alpha|={ k_I}}\frac1{\alpha!}
					\frac{\partial^{{ k_I}}u_i(P)}{\partial x_1^{\alpha_1}\cdots\partial x_N^{\alpha_N}}
					\frac{\partial^{{ k_I}}G_{\Omega}(x,P)}{\partial y_1^{\alpha_1}\cdots\partial y_N^{\alpha_N}}+O(\e^{{ k_I}+1})\right]\frac{\partial V_{\e,u_l}(x) }{\partial \nu}\ d S_x\\
                    &=-\frac{A(N,{ k_I})}{\la}\e^{N-2+2{k_I}}	\sum_{|\alpha|={ k_I}}\frac1{\alpha!}
					\frac{\partial^{{ k_I}}u_i(P)}{\partial x_1^{\alpha_1}\cdots\partial x_N^{\alpha_N}}\int_{\partial B(P,\e)} \frac{\partial^{{ k_I}}G_{\Omega}(x,P)}{\partial y_1^{\alpha_1}\cdots\partial y_N^{\alpha_N}}\frac{\partial V_{\e,u_l}(x) }{\partial \nu}\ d S_x+O(\e^{N-1+{ k_I}+{ k_L}}).\\
                    \end{aligned}
        \end{equation}
        Then we only need to consider $\int_{\partial B(P,\e)} \frac{\partial^{{ k_I}}G_{\Omega}(x,P)}{\partial y_1^{\alpha_1}\cdots\partial y_N^{\alpha_N}}\frac{\partial V_{\e,u_l}(x) }{\partial \nu}\ d S_x$. Since 
        \begin{equation*}
            \begin{cases}
			\Delta V_{\e,u_l} =0 &\text{in }\Omega\setminus B(P,\e),\vspace{2mm}\\
			V_{\e,u_l} =0 &\text{on }\partial \Omega\vspace{2mm},\\
			V_{\e,u_l} =u_l &\text{on }\partial B(P,\e),
		\end{cases} 
        \quad \text{and}\quad
         \begin{cases}
			\displaystyle \Delta \frac{\partial^{{ k_I}}G_{\Omega}(x,P)}{\partial y_1^{\alpha_1}\cdots\partial y_N^{\alpha_N}} =0 &\text{in }\Omega\setminus B(P,\e),\vspace{2mm}\\
			\displaystyle \frac{\partial^{{ k_I}}G_{\Omega}(x,P)}{\partial y_1^{\alpha_1}\cdots\partial y_N^{\alpha_N}} =0 &\text{on }\partial \Omega,\vspace{2mm}\\
		\end{cases}
        \end{equation*}
    we obtain that
    \begin{equation}\label{eq27102025-4}
        \begin{aligned}
            0&=\int_{\O_\e} V_{\e,u_l}\Delta \frac{\partial^{{ k_I}}G_{\Omega}(x,P)}{\partial y_1^{\alpha_1}\cdots\partial y_N^{\alpha_N}}-\frac{\partial^{{ k_I}}G_{\Omega}(x,P)}{\partial y_1^{\alpha_1}\cdots\partial y_N^{\alpha_N}}\Delta V_{\e,u_l}\ dx\vspace{2mm}\\
            &= \int_{\partial \O_\e} V_{\e,u_l}\frac{\partial}{\partial \nu}\frac{\partial^{{ k_I}}G_{\Omega}(x,P)}{\partial y_1^{\alpha_1}\cdots\partial y_N^{\alpha_N}}-\frac{\partial^{{ k_I}}G_{\Omega}(x,P)}{\partial y_1^{\alpha_1}\cdots\partial y_N^{\alpha_N}} \frac{\partial V_{\e,u_l}}{\partial \nu} \ dS_x\vspace{2mm}\\
            &= -\int_{\partial B(P,\e)} u_l\frac{\partial}{\partial \nu}\frac{\partial^{{ k_I}}G_{\Omega}(x,P)}{\partial y_1^{\alpha_1}\cdots\partial y_N^{\alpha_N}}\ dS_x+\int_{\partial B(P,\e)}\frac{\partial^{{ k_I}}G_{\Omega}(x,P)}{\partial y_1^{\alpha_1}\cdots\partial y_N^{\alpha_N}} \frac{\partial V_{\e,u_l}}{\partial \nu} \ dS_x.\vspace{2mm}\\
        \end{aligned}
    \end{equation}
    Thus we have 
    \begin{equation}
        \begin{aligned}
            I_2
                    =&-\frac{1}{\la}A(N,{ k_I})\e^{N-2+2{ k_I}}	\sum_{|\alpha|={ k_I}}\frac1{\alpha!}
					\frac{\partial^{{ k_I}}u_i(P)}{\partial x_1^{\alpha_1}\cdots\partial x_N^{\alpha_N}}\int_{\partial B(P,\e)} u_l\frac{\partial}{\partial \nu}\frac{\partial^{{ k_I}}\Gamma(x,P)}{\partial y_1^{\alpha_1}\cdots\partial y_N^{\alpha_N}}\ dS_x\\
                    &+O(\e^{2N-3+{ k_L}+2{{ k_I}}})+O(\e^{N-1+{{ k_I}}+{ k_L}})\vspace{2mm}\\
                    =&-\frac{1}{\la}A(N,{ k_I})\e^{N-2+2{ k_I}}	\sum_{\substack{s\\|\alpha|={ k_I}}}\frac1{\alpha!}
					\frac{\partial^{{ k_I}}u_i(P)}{\partial x_1^{\alpha_1}\cdots\partial x_N^{\alpha_N}}\int_{\partial B(P,\e)} u_l\frac{\partial}{\partial x_s}\left(\frac{\partial^{{ k_I}}\Gamma(x,P)}{\partial y_1^{\alpha_1}\cdots\partial y_N^{\alpha_N}}\right)\frac{x_s-P_s}{\e}\ dS_x\vspace{2mm}\\
                    &+O(\e^{N-1+{ k_I}+{ k_L}}).\vspace{2mm}\\
        \end{aligned}
    \end{equation}
   Next, using Lemma \ref{Cor-Q}, we deduce that for any $x\in \partial B(P,\e)$
        \begin{align*}
          &  \sum_{\substack{s\\|\alpha|={ k_I}}}\frac1{\alpha!}
					\frac{\partial^{{ k_I}}u_i(P)}{\partial x_1^{\alpha_1}\cdots\partial x_N^{\alpha_N}}\frac{\partial}{\partial x_s}\left(\frac{\partial^{{ k_I}}\Gamma(x,P)}{\partial y_1^{\alpha_1}\cdots\partial y_N^{\alpha_N}}\right)\frac{x_s-P_s}{\e}
                    \\
                    =&\frac{1}{A(N,{ k_I})}\sum_{\substack{s\\|\alpha|={ k_I}}}\frac1{\alpha!}
					\frac{\partial^{{ k_I}}u_i(P)}{\partial x_1^{\alpha_1}\cdots\partial x_N^{\alpha_N}}\frac{\partial}{\partial x_s}
                    \left[\frac{(x_1-P_1)^{\alpha_1}\cdots (x_N-P_N)^{\alpha_N}+Q_\alpha(x,P)}
					{|x-P|^{N-2+2{ k_I}}}\right]\frac{x_s-P_s}{\e}\\
                  =&-\frac{N-2+2{ k_I}}{A(N,{ k_I})}\sum_{\substack{s\\|\alpha|={ k_I}}}\frac1{\alpha!}
					\frac{\partial^{{ k_I}}u_i(P)}{\partial x_1^{\alpha_1}\cdots\partial x_N^{\alpha_N}}\frac{(x_1-P_1)^{\alpha_1}\cdots (x_N-P_N)^{\alpha_N}+Q_\alpha(x,P)}{|x-P|^{N+2{ k_I}}}\frac{(x_s-P_s)^2}{\e}\\
                  &+\frac{k_I}{A(N,{ k_I})\e}\sum_{|\alpha|={ k_I}}\frac1{\alpha!}
					\frac{\partial^{{ k_I}}u_i(P)}{\partial x_1^{\alpha_1}\cdots\partial x_N^{\alpha_N}}\frac{(x_1-P_1)^{\alpha_1}\cdots (x_N-P_N)^{\alpha_N}}
					{|x-P|^{N-2+2{ k_I}}}\\
                    &+\frac{1}{A(N,{ k_I})}\sum_{\substack{s\\|\alpha|={ k_I}}}\frac1{\alpha!}
					\frac{\partial^{{ k_I}}u_i(P)}{\partial x_1^{\alpha_1}\cdots\partial x_N^{\alpha_N}}\frac{\frac{\partial Q_\alpha(x,P)}{\partial x_s}}
					{|x-P|^{N-2+2{ k_I}}}\frac{x_s-P_s}{\e}\\
                   =&-\frac{N-2+{ k_I}}{A(N,{ k_I})\e} \sum_{|\alpha|={ k_I}}\frac1{\alpha!}
					\frac{\partial^{{ k_I}}u_i(P)}{\partial x_1^{\alpha_1}\cdots\partial x_N^{\alpha_N}}\frac{(x_1-P_1)^{\alpha_1}\cdots (x_N-P_N)^{\alpha_N}}
					{|x-P|^{N-2+2{ k_I}}}.\\
        \end{align*}
        Using Taylor expansion \eqref{eq27102025-1}  and arguing as in the computation of $I_1$, it follows that
        \begin{equation}\label{eq05112025-2}
            I_2=\frac{N-2+{ k_I}}{\lambda (N-2+k_I+k_L) }\e^{N-2+k_I+k_L}	\sum_{\substack{|\alpha|={ k_I}\\
           |\beta|={ k_L}}}\frac{1}{\alpha!\beta!}\frac{\partial^{{ k_I}} u_i (P)}{\partial x_{1}^{\alpha_1}\cdots\partial x_{N}^{\alpha_N}}\frac{\partial^{{ k_L}} u_l (P)}{\partial x_{1}^{\beta_1}\cdots\partial x_{N}^{\beta_N}} C_{\alpha\beta}
            +O(\e^{N-1+{ k_I}+{ k_L}}).
        \end{equation}
    Combining \eqref{eq27102025-0}, \eqref{eq05112025-1} and \eqref{eq05112025-2}, it follows that
    \begin{equation}
    \begin{aligned}
        &\int_{\O_\e}u_i(x) V_{\e,u_l}(x)\ dx\\
               =&\frac{N-2+2{ k_I}}{\lambda (N-2+k_I+k_L) }\e^{N-2+{ k_I}+{ k_L}}	\sum_{\substack{|\alpha|={ k_I}\\
           |\beta|={ k_L}}}\frac{1}{\alpha!\beta!}\frac{\partial^{{ k_I}} u_i (P)}{\partial x_{1}^{\alpha_1}\cdots\partial x_{N}^{\alpha_N}}\frac{\partial^{{ k_L}} u_l (P)}{\partial x_{1}^{\beta_1}\cdots\partial x_{N}^{\beta_N}} C_{\alpha\beta}
            +O(\e^{N-1+{ k_I}+{ k_L}}).\\
              \end{aligned}
    \end{equation}
    By Lemma \ref{lem2212-1}, we see that $\sum_{\substack{|\alpha|={ k_I}}}\frac{1}{\alpha!}\frac{\partial^{{ k_I}} u_i (P)}{\partial x_{1}^{\alpha_1}\cdots\partial x_{N}^{\alpha_N}}x_1^{\alpha_1}\cdots x_N^{\alpha_N}$ and $\sum_{\substack{
           |\beta|={ k_L}}}\frac{1}{\beta!}\frac{\partial^{{ k_L}} u_l (P)}{\partial x_{1}^{\beta_1}\cdots\partial x_{N}^{\beta_N}}x_1^{\beta_1}\cdots x_N^{\beta_N} $ are harmonic polynomials. If ${ k_I}\neq { k_L}$,
        using Theorem 3.2.1 in  \cite{Groemer_1996}, we have 
        \begin{equation}
            \sum_{\substack{|\alpha|={ k_I}\\
           |\beta|={ k_L}}}\frac{1}{\alpha!\beta!}\frac{\partial^{{ k_I}} u_i (P)}{\partial x_{1}^{\alpha_1}\cdots\partial x_{N}^{\alpha_N}}\frac{\partial^{{ k_L}} u_l (P)}{\partial x_{1}^{\beta_1}\cdots\partial x_{N}^{\beta_N}} \int_{\partial B(0,1)} x_{1}^{\alpha_1+\beta_1}\cdots x_{N}^{\alpha_N+\beta_N} \ d S_x=0.
        \end{equation}
         If ${ k_I}={ k_L}=k$,
        using equation 4.634 in \cite{Gradshteyn2007}, we have
        \begin{equation}
            \int_{\partial B(0,1)} x_{1}^{\alpha_1+\beta_1}\cdots x_{N}^{\alpha_N+\beta_N} \ d S_x=
            \begin{cases}
              \displaystyle \frac{2\prod_{i=1}^N\Gamma(\frac{\alpha_i+\beta_i+1}{2})}{\Gamma(k+\frac{N}{2})},  &if\ \alpha_i+\beta_i\ is  \ even, \ i=1,\cdots,N,\vspace{2mm}\\
                0,& otherwise.
            \end{cases}
        \end{equation}
 which gives the claim for $N\ge3$.

        For $N=2$,  repeating above calculations, we obtain the results for ${ k_I}\ge1$. Therefore, we only need to consider ${ k_I}=0$. If ${ k_I}={ k_L}=0$, using equation \eqref{eq10102025-1}, then
        \begin{equation}
            \begin{aligned}
                \int_{\O_\e}u_i(x) V_{\e,u_l}(x)\ dx= &\int_{\O_\e}u_i(x) \left[\frac{2\pi}{|\log \e|}u_l(p)G_\Omega(x,P)+O\left(\frac{1}{|\log \e|^2}G_\Omega(x,P)\right)\right]\ dx\\
                 =&\frac{2\pi}{|\log \e|}u_l(p)\int_{\O}u_i(x) G_\Omega(x,P)\ dx+O\left(\frac{1}{|\log \e|^2}\right)\\
                 =&\frac{2\pi}{\lambda|\log \e|}u_i(P)u_l(P)+O\left(\frac{1}{|\log \e|^2}\right).
            \end{aligned}
        \end{equation}
        If $k_I=0$ and ${ k_L}\ge1$, using equation \eqref{eq10102025-3}, we directly obtain that 
        $$ \int_{\O_\e}u_i(x) V_{\e,u_l}(x)\ dx=O\left(||V_{\e,u_l}||_{L^1(\O_\e)}\right)=O(\e^{1+{ k_L}}).
        $$
      This completes the proof.
    \end{proof}
    \begin{rem} Assume that the vanishing order of $u_l$ at $P$ is equal to $k_L$. As a byproduct of the proof of Lemma \ref{lem1211-1}, we derive the $u_l-$capacity of $B(P,\e)$ with respect to the domain $\O$.

    Indeed, by  definition of  $\hbox{Cap}_\O(B(P,\e),u_l)$, we have 
    $$
    \hbox{Cap}_\O(B(P,\e),u_l)=\int_\O|\nabla V_{\e,u_l}|^2\ dx=\int_{\O_\e}|\nabla V_{\e,u_l}|^2\ dx+ \int_{B(P,\e)}|\nabla u_l|^2\ dx.
    $$
        Using \eqref{eq05112025-2}, we  obtain 
        \begin{equation}
            \begin{aligned}
               & \int_{\O_\e}|\nabla V_{\e,u_l}|^2\ dx=\int_{\partial\O_\e}V_{\e,u_l}\frac{\partial V_{\e,u_l}}{\partial\nu}\ dS_x-\int_{\O_\e}V_{\e,u_l}\Delta V_{\e,u_l}\ dx=-\int_{\partial B(P,\e)}u_l\frac{\partial V_{\e,u_l}}{\partial\nu}\ dS_x\\
            =&\frac{N-2+{ k_L}}{N-2+2 k_L}\e^{N-2+2{ k_L}}	\sum_{\substack{|\alpha|={ k_L}\\
           |\beta|={ k_L}}}\frac{1}{\alpha!\beta!}\frac{\partial^{{ k_L}} u_l (P)}{\partial x_{1}^{\alpha_1}\cdots\partial x_{N}^{\alpha_N}}\frac{\partial^{{ k_L}} u_l (P)}{\partial x_{1}^{\beta_1}\cdots\partial x_{N}^{\beta_N}}C_{\alpha\beta}+O(\e^{N-1+2{ k_L}}).
            \end{aligned}
        \end{equation}
     Using \eqref{eq05112025-1},   a straightforward computation shows that
        \begin{equation}
        \begin{aligned}
            \int_{B(P,\e)}|\nabla u_l|^2\ dx=&\int_{\partial B(P,\e)}u_l\frac{\partial u_l}{\partial\nu}\ dS_x+\lambda\int_{B(P,\e)} u_l^2\ dx\\
            =&\frac{ k_L}{N-2+2 k_L}\e^{N-2+2{ k_L}}	\sum_{\substack{|\alpha|={ k_L}\\
           |\beta|={ k_L}}}\frac{1}{\alpha!\beta!}\frac{\partial^{{ k_L}} u_l (P)}{\partial x_{1}^{\alpha_1}\cdots\partial x_{N}^{\alpha_N}}\frac{\partial^{{ k_L}} u_l (P)}{\partial x_{1}^{\beta_1}\cdots\partial x_{N}^{\beta_N}}C_{\alpha\beta}+O(\e^{N-1+2{ k_L}}).
            \end{aligned}
        \end{equation}
       Thus, 
        \begin{align*}
            \hbox{Cap}_\O(B(P,\e),u_l)
            =&\e^{N-2+2{ k_L}}	\sum_{\substack{|\alpha|={ k_L}\\
           |\beta|={ k_L}}}\frac{1}{\alpha!\beta!}\frac{\partial^{{ k_L}} u_l (P)}{\partial x_{1}^{\alpha_1}\cdots\partial x_{N}^{\alpha_N}}\frac{\partial^{{ k_L}} u_l (P)}{\partial x_{1}^{\beta_1}\cdots\partial x_{N}^{\beta_N}}C_{\alpha\beta}+O(\e^{N-1+2{ k_L}}).
        \end{align*}
    \end{rem}

     We consider the canonical extension of $u_\e$ to $\Omega$, obtained by extending it by zero inside $B(P,\e)$.
     By the definition of $V_{\e,u_i}$, we see that $u_i-V_{\e,u_i}\in H_0^1(\O_\e)$ for $i=1,\cdots,m(\lambda)$.
 Now we will decompose $u_{j,\e} \hbox{ in } span\{u_1-V_{\e,u_1},\cdots,u_{m(\lambda)}-V_{\e,u_{m(\lambda)}}\}$ and its orthogonal space. 
\begin{equation}\label{eq15112025-1}
    u_{j,\e}=\sum_{i=1}^{m(\la)}\alpha_{ij,\e} (u_i-V_{\e,u_i})+ v_{j,\e},
\end{equation}
where 
$v_{j,\e}\in H_0^1(\O_\e)$ and 
\begin{equation}\label{eq15112025-2}
  0=\int_{\O_\e} \nabla v_{j,\e} \cdot\nabla(u_i-V_{\e,u_i})\ dx= -\int_{\O_\e}  v_{j,\e} \Delta (u_i-V_{\e,u_i})\ dx=\lambda\int_{\O_\e}  v_{j,\e} u_i\ dx,\ \ i=1,\cdots,m(\lambda).
\end{equation}
Then $v_{j,\e}$ satisfies the following 
\begin{equation}\label{eq2301-4}
    \begin{cases}
       \displaystyle -\Delta v_{j,\e}-\lambda_{j,\e}v_{j,\e}=f_{j,\e}\ \ &in \ \ \O_\e,\vspace{2mm}\\
       \displaystyle v_{j,\e}=0\ \ &on \ \ \partial\O_\e,\vspace{2mm}\\
      \displaystyle \int_{\O_\e}  v_{j,\e} u_i\ dx= 0,\ \ &i=1,\cdots,m(\lambda),
    \end{cases}
\end{equation}
where $f_{j,\e}=(\lambda_{j,\e}-\lambda)\sum_{i=1}^{m(\la)}\alpha_{ij,\e}  u_i-\lambda_{j,\e}\sum_{i=1}^{m(\la)}\alpha_{ij,\e}  V_{\e,u_i} .$
In the following  lemma, we will give an estimate of $||v_{j,\e}||_{L^2(\O_\e)}$.
     \begin{lemma}\label{lem1511}
        Assume that $v_{j,\e}$ and $\alpha_{ij,\e}$ are determined by \eqref{eq15112025-1} and \eqref{eq15112025-2}. Then,
         \begin{equation}
              ||v_{j,\e}||_{L^2(\O_\e)}\le C|\lambda_{j,\e}-\lambda|+C\sum_{i=1}^{m(\la)}|\alpha_{ij,\e}| ||  V_{\e,u_i}||_{L^2(\O_\e)}.
         \end{equation}

     \end{lemma}
     \begin{proof}
     Since $v_{j,\e}$ satisfies \eqref{eq2301-4},
to obtain the result, we only need to prove
   \begin{equation}\label{eq0531-1}
       ||{v}_{j,\e}||_{L^2(\O_\e)}\leq C||f_{j,\e}||_{L^2(\O_\e)},\ \ j=1,\cdots,m(\lambda).
   \end{equation}
   Indeed, if the above inequality has been proved, then we obtain that
   \begin{equation}
       \begin{aligned}
           ||v_{j,\e}||_{L^2(\O_\e)}
           \le& C||f_{j,\e}||_{L^2(\O_\e)}\\
           \le&C|\lambda_{j,\e}-\lambda|\sum_{i=1}^{m(\la)}|\alpha_{ij,\e} | || u_i||_{L^2(\O_\e)}+C\sum_{i=1}^{m(\la)}|\alpha_{ij,\e}| ||  V_{\e,u_i}||_{L^2(\O_\e)}\\
           \le&C|\lambda_{j,\e}-\lambda|+C\sum_{i=1}^{m(\la)}|\alpha_{ij,\e}| ||  V_{\e,u_i}||_{L^2(\O_\e)}.\\
           \end{aligned}
   \end{equation}

In the following, we prove the  lemma  by contradiction. If not, there would  exist a sequence ${v}_{j_0,\e_n}$, $\lambda_{j_0,\e_n}\to\lambda$ and $f_{j_0,\e_n}$, $\e_n\to 0$ as $n\to \infty$ that satisfy 
   \begin{equation}\label{eq11112025-1}
    \begin{cases}
       \displaystyle -\Delta {v}_{j_0,\e_n}-\lambda_{j_0,\e_n}{v}_{j_0,\e_n}=f_{j_0,\e_n}\ \ &\hbox{in} \ \ \O_{\e_n},\\
       \displaystyle {v}_{j_0,\e_n}=0\ \ &\hbox{on} \ \ \partial\O_{\e_n},\\
    \end{cases}
\end{equation}
and $      \displaystyle \int_{\O_{\e_n}}  v_{j_0,\e_n} u_i\ dx= 0$ for any $i=1,\cdots,m(\lambda),$
but 
\begin{equation*}
     ||{v}_{j_0,\e_n}||_{L^2(\O_{\e_n})}\geq n||f_{j_0,\e_n}||_{L^2(\O_{\e_n})}.
\end{equation*}

Without loss of generality, we assume that $||{v}_{j_0,\e_n}||_{L^2(\O_{\e_n})}=1$. 
Then  $||f_{j_0,\e_n}||_{L^2(\O_{\e_n})}\to 0$ as $n\to \infty$. Multiplying \eqref{eq11112025-1} by ${v}_{j_0,\e_n}$ and integrating on $\O_{\e_n}$, we obtain that
\begin{equation*}
    ||{v}_{j_0,\e_n}||_{H_0^1(\O_{\e_n})}^2\leq \lambda_{j_0,\e_n}||{v}_{j_0,\e_n}||_{L^2(\O_{\e_n})}^2+||f_{j_0,\e_n}||_{L^2(\O_{\e_n})}||{v}_{j_0,\e_n}||_{L^2(\O_{\e_n})}\leq \lambda+o_n(1).
\end{equation*}
Thus there exists a subsequence  which we still denote $\{{v}_{j_0,\e_n}\}$ such that 
$${v}_{j_0,\e_n}\rightharpoonup{v}_{j_0} \ \text{weakly in }H_0^1(\O)\ \  \text{and} \  \ {v}_{j_0,\e_n}\to{v}_{j_0} \text{ in }L^2(\O). $$
Then ${v}_{j_0}$ is a weak solution of 
\begin{equation*}
    \begin{cases}
       \displaystyle -\Delta {v}_{j_0}-\lambda{v}_{j_0}=0\ \ &\hbox{in} \ \ \O,\\
       \displaystyle {v}_{j_0}=0\ \ &\hbox{on} \ \ \partial\O,\\
    \end{cases}
\end{equation*}
which deduces that ${v}_{j_0}\in span\{u_1,\cdots,u_{m(\lambda)}\}$. 
Since ${v}_{j_0,\e_n}\to{v}_{j_0} \text{ in }L^2(\O)$,  it follows from
$
\int_{\O_{\e_n}}{v}_{j_0,\e_n}u_l \ dx=0$  for any  $ l=1,\cdots,m(\lambda)
$
that $\int_{\O}{v}_{j_0}u_l \ dx=0$ for any $l=1,\cdots,m(\lambda)$. Then ${v}_{j_0}=0$. However, by the assumption $||{v}_{j_0,\e_n}||_{L^2(\O_{\e_n})}=1$, we have  $||{v}_{j_0}||_{L^2(\O)}=1$ which  contradicts ${v}_{j_0}=0$. This completes the proof.
     \end{proof}
     \begin{rem}
      In above theorem,   it is essential to retain $\alpha_{ij,\e}$ in the second term of  \eqref{eq0531-1}.  In the proof of  Theorem \ref{sim}, we aim to show that $\int_{\O_\e}v_{j,\e}V_{\e,u_l}\ dx=o(1)|\lambda_{j,\e}-\lambda|$ for any $l,j=1+d_0+\cdots+d_{J-1},\cdots,d_0+\cdots+d_J$.  It follows from \eqref{V3} and \eqref{eq10102025-3} that
      $$
      \begin{aligned}
            \left  |\int_{\O_\e}v_{j,\e}V_{\e,u_l}\ dx\right|
            \leq& ||v_{j,\e}||_{L^2(\O_\e)} ||  V_{\e,u_l}||_{L^2(\O_\e)}\\
            \leq& o(1)|\lambda_{j,\e}-\lambda|+C\sum_{I=0}^p\sum_{i=1+d_0+\cdots+d_{I-1}}^{d_0+\cdots+d_I}|\alpha_{ij,\e}| ||  V_{\e,u_i}||_{L^2(\O_\e)}||  V_{\e,u_l}||_{L^2(\O_\e)}.\\
      \end{aligned}
      $$
 For any $I\ge J$,  $||  V_{\e,u_i}||_{L^2(\O_\e)}||  V_{\e,u_l}||_{L^2(\O_\e)}=o(1)|\lambda_{j,\e}-\lambda|$ for any $i=1+d_0+\cdots+d_{I-1},\cdots,d_0+\cdots+d_I$. In this case,  we only need the boundedness of $\alpha_{ij,\e}$. However, if $I<J$, $||  V_{\e,u_i}||_{L^2(\O_\e)}||  V_{\e,u_l}||_{L^2(\O_\e)}=o(1)|\lambda_{j,\e}-\lambda|$ may fail.  At this stage, a  decay estimate for the coefficients $\alpha_{ij,\e}$  (see Lemma \ref{lem2204}) is needed. Therefore,  it is necessary to keep the coefficients $\alpha_{ij,\e}$ in front of $||  V_{\e,u_i}||_{L^2(\O_\e)}$.
     \end{rem}

    	\section{\bf The simplicity of the approximating eigenvalues}\label{simpl}
        In this section we will study, for small $\e$, the simplicity of the eigenvalue $\la_\e$.

 For any $J=0,1,\cdots,p$, 
 it is shown in \cite{alm2} that for $N\ge3$ there exist positive real numbers $\mu_j\ne0$ for any $j=1+d_0+\cdots+ d_{J-1},\ldots,d_0+\cdots+d_J$ such that the eigenbranches departing from $\lambda$ have the following asymptotic behavior
    \begin{equation}
		\la_{j,\e}=\la+\mu_j\e^{N-2+2k_J} \quad \text{for any }j=1+d_0+\cdots +d_{J-1},\ldots,d_0+\cdots+d_J.
	\end{equation}
In this section we  revisit the above result in such a way so that the numbers $\mu_j$ appear more explicit, so giving clearer information about the possible simplicity of the eigenbranches.
	
Finally, we denote by $u_{1,\e},..,u_{m(\la),\e}$ the  perturbed eigenfunctions associated to $\la_{1,\e},..,\la_{m(\la),\e}$.

	\subsection{\bf Proof of the main theorems}

     We first introduce some notation.
     Let  $n_J:=\#\{\alpha:\ |\alpha|=k_J\}$ denote the number of elements in $\{\alpha:\ |\alpha|=k_J\}$. Then, we label the elements in this set as 
    $$
    \{\alpha:\ |\alpha|=k_J\}=\{\alpha^{(1)},\cdots,\alpha^{(n_J)}\}.
    $$
    We also label $\{\beta:\ |\beta|=k_J\}=\{\beta^{(1)},\cdots,\beta^{(n_J)}\}$ in the same way.
    Denote 
    \begin{equation}\label{defU}
         U_i=\left(D^{\alpha^{(1)}}u_i(P),\cdots,D^{\alpha^{(n_J)}}u_i(P)\right)^T\in M_{n_J\times 1},
    \end{equation}
    where 
    \begin{equation*}
        D^{\alpha^{(m)}}u_i(P)=\frac{\partial^{k_J} u_i (P)}{\partial x_{1}^{\alpha_1^{(m)}}\cdots\partial x_{N}^{\alpha_N^{(m)}}}.
    \end{equation*}
    Denote by
    \begin{equation}
         \mathcal{C}_{k_J}=\left(c_{mj}\right)\in M_{n_J\times n_J},
    \end{equation}
	where
    $$
    c_{mj}=\frac{C_{\alpha^{(m)}\beta^{(j)}}}{\alpha^{(m)}!\beta^{(j)}!}=\frac{(N+2k_J)(N-2+2k_J)}{\alpha^{(m)}!\beta^{(j)}!}\int_{ B(0,1)} x_{1}^{\alpha_1^{(m)}+\beta_1^{(j)}}\cdots x_{N}^{\alpha_N^{(m)}+\beta_N^{(j)}} \ d x.
    $$
    The following lemma provides a more convenient representation of the matrix $G_{k_J}$ in Theorem \ref{sim} , which will be useful in the subsequent computations.
    \begin{lemma}\label{defG}
        Under the same condition as Theorem \ref{sim}, there exists an invertible matrix $B_{k_J}$ such that the matrix $G_{k_J}$ in Theorem \ref{sim} can be equivalently written as
        $$
G_{k_J}=
\begin{bmatrix}
    \left(B_{k_J}U_{1+d_0+\cdots+d_{J-1}}\right)^TB_{k_J}U_{1+d_0+\cdots+d_{J-1}} &\dots & \left(B_{k_J}U_{1+d_0+\cdots+d_{J-1}}\right)^TB_{k_J}U_{d_0+\cdots+d_{J}}\\
    \vdots&  &\vdots\\
     \left(B_{k_J}U_{d_0+\cdots+d_{J}}\right)^TB_{k_J}U_{1+d_0+\cdots+d_{J-1}} &\dots & \left(B_{k_J}U_{d_0+\cdots+d_{J}}\right)^TB_{k_J}U_{d_0+\cdots+d_{J}}\\
\end{bmatrix}.
$$
Moreover, $G_{k_J}$ is a Gram matrix.
    \end{lemma}
    \begin{proof}
        First, we claim that $ \mathcal{C}_{k_J}$ is a positive definite matrix whenever $N=2$ and $k_J\ge 1$, or $N\ge 3$. 

         Assume that $A=\left(a_1,\cdots,a_{n_J}\right)^{T} \in M_{n_J\times 1} $. 
    \begin{equation}\label{eq0512-2}
        \begin{aligned}
            A^T\mathcal{C}_{k_J}A=&(N+2k_J)(N-2+2k_J)\sum_{m,j=1}^{n_J}\frac{1}{\alpha^{(m)}!\alpha^{(j)}!}\int_{ B(0,1)}a_m a_j x_{1}^{\alpha_1^{(m)}+\alpha_1^{(j)}}\cdots x_{N}^{\alpha_N^{(m)}+\alpha_N^{(j)}} \ d x\\
            =&(N+2k_J)(N-2+2k_J)\int_{ B(0,1)}\left[\sum_{m=1}^{n_J}\frac{a_m}{\alpha^{(m)}!}x_{1}^{\alpha_1^{(m)}}\cdots x_{N}^{\alpha_N^{(m)}} \right]^2\ d x\geq 0.
        \end{aligned}
    \end{equation}
	If $A^T\mathcal{C}_{k_J}A=0$, then 
    \begin{equation}\label{eq0512-3}
        \sum_{m=1}^{n_J}\frac{a_m}{\alpha^{(m)}!}x_{1}^{\alpha_1^{(m)}}\cdots x_{N}^{\alpha_N^{(m)}}=0\ \hbox{for any }   x\in B(0,1). 
         \end{equation}
      Thus, we have $a_m=0$ for any $m=1,\cdots,n_J$. Then we conclude that $\mathcal{C}_{k_J}$ is positive definite.
        
        It follows from the above claim that   $\mathcal{C}_{k_J}$  is positive definite  for any $k_J$. Therefore,  there exists an invertible matrix $B_{k_J}$ such that $\mathcal{C}_{k_J}=B_{k_J}^TB_{k_J}$. Thus, we have for any $i,l=1+d_0+\cdots+d_{J-1},\cdots,d_0+\cdots+d_J$,
   \begin{equation}
   \begin{aligned}
       \sum_{\substack{|\alpha|={ k_J}\\
           |\beta|={ k_J}}}\frac{1}{\alpha!\beta!}\frac{\partial^{{ k_J}} u_{i} (P)}{\partial x_{1}^{\alpha_1}\cdots\partial x_{N}^{\alpha_N}}\frac{\partial^{{ k_J}} u_{l} (P)}{\partial x_{1}^{\beta_1}\cdots\partial x_{N}^{\beta_N}} C_{\alpha\beta} =U_i^T\mathcal{C}_{k_J}U_l=\left(B_{k_J}U_i\right)^TB_{k_J}U_l,
            \end{aligned}
   \end{equation}
   where $U_i$ and $U_l$ are defined as \eqref{defU}.
Furthermore, we observe that the matrix in Theorem \ref{sim}
is a Gram matrix. This ends the proof.
    \end{proof}

    \subsection{Proof of Theorem \ref{sim}}
    First, we begin by proving  Corollary \ref{cor4.2} and its corresponding case for $N=2$, Corollary \ref{cor1510-1}, as special cases of Theorem \ref{sim}. 
 \begin{proof}[\bf Proof of Corollary \ref{cor4.2}]
        Recall that $u$ is any eigenfunction associated to $\la$. 
		Using the same notation of the previous section, we have that the function $u-V_{\e,u}$ solves
		\begin{equation*}
			\begin{cases}
				-\Delta(u-V_{\e,u})=\lambda u&in\ \O_\e,\\
				u-V_{\e,u}=0&on\ \partial\O_\e.
			\end{cases}
		\end{equation*}
		Multiplying by $u_{j,\e}$ and integrating twice we get
		\begin{equation}\label{2d2}
			(\la_{j,\e}-\la)\int_{\O_\e}u_{j,\e} u=\la_{j,\e}\int_{\O_\e}u_{j,\e}V_{\e,u}.
		\end{equation}
        We provide the same decomposition as in
		  \eqref{eq15112025-1} and \eqref{eq15112025-2}, as follows:
        \begin{equation}\label{2con}
    u_{j,\e}=\sum_{i=1}^{m(\la)}\alpha_{ij,\e} (u_i-V_{\e,u_i})+ v_{j,\e},
\end{equation}
 and
        \begin{equation}
   \int_{\O_\e} v_{j,\e} u_i\ dx= 0,\ \ i=1,\cdots,m(\lambda),
\end{equation}
		where $\{u_1,u_2,..,u_{m(\la)}\}$ is a system of orthonormal eigenfunctions of $E(\la)$ where $u_1\in E_0$ and $u_l\in E_1$, $l=2,\cdots,m(\la)$.
 Note that this decomposition will be crucial in dealing with the case where the order of vanishing of the $u$ is greater than $2$. In the present case, it can be omitted (see Remark \ref{M1}).

Using Lemma \ref{lem1511}, we obtain that 
\begin{equation}
    u_{j,\e}\to \sum_{i=1}^{m(\la)}\alpha_{ij} u_i \ \ \text{in}\ \ L^2(\O).
\end{equation}
  Multiplying \eqref{2con} by $u_l$  and integrating on $\O_\e$, it follows that $\alpha_{lj,\e}\to \alpha_{lj}$ for any $l,j=1,\cdots,m(\lambda)$.
		Moreover, by $\int_{\O_\e}u_{j,\e}^2=1$  for any $j$, we obtain 
        \begin{equation}
            \int_{\O}|\sum_{i=1}^{m(\la)}\alpha_{ij} u_i|^2=\sum_{i=1}^{m(\la)}\alpha_{ij}^2\int_{\O}u_i^2=\sum_{i=1}^{m(\la)}\alpha_{ij}^2=1,
        \end{equation}
        which  implies that
		\begin{equation}\label{nonz*}
			(\alpha_{1j},..,\alpha_{m(\la)j})\ne(0,..,0).
		\end{equation}
		Choosing $u=u_l$ for $l=1,..,m(\la)$, then  gets that the $LHS$ of \eqref{2d2} becomes
		\begin{equation}\label{2d4}
			LHS=(\la_{j,\e}-\la)\sum_{i=1}^{m(\la)}\alpha_{ij,\e}\int_{\O_\e}(u_i-V_{\e,u_i})u_l=(\la_{j,\e}-\la)(\alpha_{lj}+o(1)).
		\end{equation}
		The estimate of $LHS$ in \eqref{2d4} holds for any eigenvalues $\la_{j,\e}$.
        
        On the other hand, we will see that the computation of $RHS$ of \eqref{2d2} strongly depends on the order of vanishing at $P$ of $u_l$. Let us start the computations of the $RHS$ of \eqref{2d2} for a generic eigenfunction $u_l$. 
        Using Lemma \ref{lem1211-1} and Lemma \ref{lem1511}, we directly obtain the following.
		For $l=1$,
		\begin{equation}\label{2d5}
			\begin{split}
				RHS=&\big(\la+o(1)\big)\int_{\O_\e}\left(\sum_{i=1}^{m(\la)}\alpha_{ij,\e} (u_i-V_{\e,u_i})+ v_{j,\e}\right)V_{\e,u_1}\\
				=&\big(\la+o(1)\big)\left[\sum_{i=1}^{m(\la)}\alpha_{ij,\e}\int_{\O_\e}u_iV_{\e,u_1}-\sum_{i=1}^{m(\la)}\alpha_{ij,\e}\int_{\O_\e}V_{\e,u_i}V_{\e,u_1}+\int_{\O_\e}v_{j,\e}V_{\e,u_1}\right]\\
                =&\big(\la+o(1)\big)\alpha_{1j,\e}\left(\frac{(N-2)N\o_N}{\lambda}\e^{N-2}u_1^2(P)+O(\e^{N-1})\right)+\sum_{i=2}^{m(\la)}\alpha_{ij,\e}O(\e^{N})\\
                & +\sum_{i=1}^{m(\la)}|\alpha_{ij,\e}|||V_{\e,u_i}||_{L^2(\O_\e)}||V_{\e,u_1}||_{L^2(\O_\e)}+||v_{j,\e}||_{L^2(\O_\e)}||V_{\e,u_1}||_{L^2(\O_\e)}\\
				=&\frac{(\la+o(1))}{\la }(N-2)N\o_N\e^{N-2}\Big(\alpha_{1j}u_1^2(P)+o(1)\Big)+o(1)|\lambda_{j,\e}-\lambda|.
			\end{split}
		\end{equation}
		For $l=2,\cdots,m(\lambda)$, we obtain 
		\begin{equation}\label{2d6}
			\begin{split}
				RHS=&\big(\la+o(1)\big)\int_{\O_\e}\left(\sum_{i=1}^{m(\la)}\alpha_{ij,\e}(u_i-V_{\e,u_i})+v_{j,\e}\right)V_{\e,u_l}\\
				=&\big(\la+o(1)\big)\left[\sum_{i=1}^{m(\la)}\alpha_{ij,\e}\int_{\O_\e}u_iV_{\e,u_l}-\sum_{i=1}^{m(\la)}\alpha_{ij,\e}\int_{\O_\e}V_{\e,u_i}V_{\e,u_l}+\int_{\O_\e}v_{j,\e}V_{\e,u_l}\right]\\
                =&\big(\la+o(1)\big)\left[\alpha_{1j,\e}O(\e^{N})+\sum_{i=2}^{m(\la)}(\alpha_{ij}+o(1))\left(\frac{N\o_N}{\lambda}\e^{N}\nabla u_i(P)\cdot\nabla u_l(P)+O(\e^{N+1})\right)\right]\\
                &+\sum_{i=1}^{m(\la)}|\alpha_{ij,\e}|||V_{\e,u_i}||_{L^2(\O_\e)}||V_{\e,u_l}||_{L^2(\O_\e)}+||v_{j,\e}||_{L^2(\O_\e)}||V_{\e,u_l}||_{L^2(\O_\e)}\\
				=&\frac{(\la+o(1))}{\la }N\o_N\e^{N}\Bigg(\sum_{i=2}^{m(\la)}\alpha_{ij}\nabla u_i(P)\cdot\nabla u_l(P)+o(1)\Bigg)+\alpha_{1j,\e}O(\e^{N})+o(1)|\lambda_{j,\e}-\lambda|.
			\end{split}
		\end{equation}
		So by \eqref{2d2}, \eqref{2d4}, \eqref{2d5} and \eqref{2d6} we get
		\begin{equation}\label{eq-2509-1}
			\frac{\lambda_{j,\e}-\la}{\e^{N-2}}(\alpha_{1j}+o(1))=\frac{(\la+o(1))}{\la }(N-2)N\o_N\Bigg(\alpha_{1j}u_1^2(P)+o(1)\Bigg)
		\end{equation}
		and 
		\begin{equation}\label{eq-2509-2}
			\frac{\lambda_{j,\e}-\la}{\e^{N-2}}(\alpha_{lj}+o(1))=\frac{(\la+o(1))}{\la }N\omega_N\e^{2}\left(\sum_{i=2}^{m(\la)}\alpha_{ij}\nabla u_i(P)\cdot\nabla u_l(P)+o(1)\right)+\alpha_{1j,\e}O(\e^{2}),
		\end{equation}
		for any $l=2,\cdots,m(\la)$.
		Denoting by 
		\begin{equation}
			L_j=\lim\limits_{\e \to 0}\frac{\lambda_{j,\e}-\la}{\e^{N-2}},\quad j=1,2,\cdots,m(\la)
		\end{equation}
		and passing to the limit in \eqref{eq-2509-1} and \eqref{eq-2509-2}, we have 
		\begin{equation}\label{eq-2509-3}
			L_j\alpha_{1j}=(N-2)N\o_N\alpha_{1j}u^2_1(P),
		\end{equation}
		and 
		\begin{equation}
			L_j\alpha_{lj}=0, \quad l=2,\cdots,m(\la).
		\end{equation}
		This means that $L_j$ is an eigenvalue of the matrix $A=(a_{ij})$, $i,j=1,2,\cdots,m(\la)$,  where $a_{11}=(N-2)N\o_Nu^2_1(P)$ and $a_{ij}=0$ for other $i,j$. Moreover the vectors $\vec{\alpha}_j=(\alpha_{1j},\cdots,\alpha_{m(\la)j})$, $j=1,\cdots,m(\la)$ are the corresponding eigenvectors. By the orthonormality of the eigenfunctions $u_{j,\e}$, we get that the vectors $\vec{\alpha}_1,..\vec{\alpha}_{m(\la)}$ are orthogonal pairwise and $||\vec{\alpha}_j||=1$. 
It is immediate to verify that the eigenvalues of $A$ are given by
		\begin{equation}
			L_1=(N-2)N\o_Nu^2_1(P)\quad \text{with eigenvector $\vec{\alpha}_1=(1,0,\cdots,0)$}
		\end{equation}
		and 
		\begin{equation}
			L_j=0\quad \text{with eigenvector $\vec{\alpha}_j\perp\vec{\alpha}_1$, which means that $\alpha_{1j}=0,\, j=2,\cdots,m(\la)$.}
		\end{equation}
		This is equivalent to say that
		\begin{equation}\label{eq-2509-5}
			\lim\limits_{\e\to0}\frac{\la_{1,\e}-\la}{\e^{N-2}}
			=L_1=(N-2)N\o_Nu^2_1(P) \text{ with eigenfunction $u_{1,\e}\to u_1\in E_0$ in $L^2(\O)$.}
		\end{equation}
		and
		\begin{equation}\label{eq-2509-6}
			\lim\limits_{\e\to0}\frac{\la_{j,\e}-\la}{\e^{N-2}}
			=0\;\hbox{for any }j=2,..,m(\la) \text{ with eigenfunction $u_{j,\e}\to \sum_{i=2}^{m(\la)}\alpha_{ij}u_i\in E_1$ in $L^2(\O)$ .}
		\end{equation}
        
		Since $\alpha_{1j}=0$ for any $j=2,\cdots,m(\la)$, it follows that $\alpha_{1j,\e}\to 0$ as $\e\to 0$ for any $j=2,\cdots,m(\la)$.
		By \eqref{eq-2509-2}, we get 
		\begin{equation}\label{eq-2509-7}
			\frac{\lambda_{j,\e}-\la}{\e^{N}}(\alpha_{lj}+o(1))=\frac{(\la+o(1))}{\la }N\omega_N\left(\sum_{i=2}^{m(\la)}\alpha_{ij}\nabla u_i(P)\cdot\nabla u_l(P)\right)+o(1)
		\end{equation}

        for any $j,l=2,\cdots,m(\la)$.
		Denoting by 
		\begin{equation} \label{eq-2509-8*}
			\Lambda_j=\lim\limits_{\e \to 0}\frac{\lambda_{j,\e}-\la}{\e^{N}},\quad j=2,\cdots,m(\la),
		\end{equation}
		and passing to the limit in \eqref{eq-2509-7}, we have 
		\begin{equation}\label{eq-2509-8}
			\Lambda_j\alpha_{lj}=N\omega_N\sum_{i=2}^{m(\la)}\alpha_{ij}\nabla u_i(P)\cdot\nabla u_l(P),\quad l,j=2,\cdots,m(\la).
		\end{equation}
		Then we get that $\Lambda_j$ are eigenvalues of the matrix 
		\begin{equation*}
			A=N\omega_N\Big(\nabla u_i(P)\cdot\nabla u_l(P)\Big),\quad i,l=2,\cdots,m(\la).
		\end{equation*}
		So simply eigenvalues of $A$ generate simple eigenvalues $\lambda_{j,\e}$ for $\e$ small enough. In particular, if {\em all} the eigenvalues $\La_j$ of $A$ are simple then the eigenvalues $\lambda_{j,\e}$ are simple for $j=2,..,m(\la)$ and $\e$ small enough.
        This ends the proof.
        \end{proof}
    The next corollary concerns the case $N=2$.
     \begin{cor}\label{cor1510-1}
		 Suppose that $N=2$, and
		$E(\la)=E_0\oplus E_1,$
        with $k_0=0$ and $k_1=1$.
		Then we have that $\la_{1,\e}$ is simple and it verifies 
		\begin{equation}\label{2d11}
			\la_{1,\e}=\la +\frac{2\pi\left(u_1^2(P)+o(1)\right)}{|\log\e|}\hbox{ and }u_{1,\e}\to u_1\in E_0.
		\end{equation}
		Next, 
		\begin{equation}\label{2d12}
			\la_{j,\e}=\la +\big(\La_j+o(1)\big)\e^2\quad\hbox{ for }j=2,..,m(\la),
		\end{equation}
        and
\begin{equation}
			u_{j,\e}\to \sum_{i=2}^{m(\lambda)}\alpha_{ij}u_i\hbox{ in } H^1_0(\O)\quad\hbox{ for }j=2,..,m(\la),
		\end{equation}       
		where $\La_j$, $j=2,..,m(\la),$ are eigenvalues of the matrix 
			\begin{equation}\label{2mat}
				A=2\pi\big(\nabla u_i(P)\cdot\nabla u_l(P)
				\big)=2\pi\begin{pmatrix}
					|\nabla u_2(P)|^2&\cdots&\nabla u_2(P)\cdot\nabla u_{m(\la)}(P)\\
					\vdots&  &\vdots\\
					\nabla u_{m(\la)}(P)\cdot\nabla u_2(P)&\cdots&
					|\nabla u_{m(\la)}(P)|^2
				\end{pmatrix},
			\end{equation}
             and the vector $(\alpha_{2j},\cdots,\alpha_{m(\lambda)j})$, $j=2,..,m(\la),$ are eigenvectors of the matrix $G(P)$.
             
		In particular, if all the eigenvalues of $A$ are simple, then $\la_{j,\e}$ are simple for  $j=2,..,m(\la).$ 
	\end{cor}
    \begin{proof}
   The only  difference between the case $N=2$ with the case $N\ge3$ is the result of RHS for $l=1$. Using \eqref{eq0301-1}, we obtain for $l=1$
         \begin{equation*}
        \begin{aligned}
            LHS=&(\lambda_{j,\e}-\lambda)\left(\alpha_{1j}+o(1)\right),\\
            RHS=&\frac{(\la+o(1))}{\la |\log\e|}2\pi u_1(P)^2\Big(\alpha_{1j}+o(1)\Big)
       +o(1)|\lambda_{j,\e}-\lambda|.
             \end{aligned}
        \end{equation*}
        For $l=2,\cdots,m(\lambda)$, we have
         \begin{equation*}
        \begin{aligned}
            LHS=&(\lambda_{j,\e}-\lambda)\left(\alpha_{lj}+o(1)\right),\\
            RHS=&\frac{2\pi(\lambda+o(1))}{\lambda}\e^2\sum_{i=2}^{m(\lambda)}\Big( \alpha_{ij}\nabla u_i(P)\cdot\nabla u_l(P) +o(1)\Big)+\alpha_{1j,\e}O(\e^{2})
       +o(1)|\lambda_{j,\e}-\lambda|.
             \end{aligned}
        \end{equation*}
          Repeating the above computations used in the proof for $N\ge3$, we  obtain the result.
    \end{proof}

     \begin{rem}\label{M1}
  Note that  in the previous case where $E(\lambda)=E_0\oplus E_1$  with $k_0=0$ and $k_1=1$,   the estimates $\alpha_{ij,\e}-\alpha_{ij}=o(1)$, $i,j=1,\cdots,m(\lambda)$ are enough to complete the proof. But if $E(\lambda)= E_0\oplus\cdots\oplus E_p,\ k_p\geq k_0+ 2$,  we need more refined estimates of  $\alpha_{ij,\e}-\alpha_{ij}$.  Indeed,  when we repeat the above  Steps in proving Corollary \ref{cor4.2} and Corollary \ref{cor1510-1}, we see that $\alpha_{1j}=0$ for $j=2,\cdots,m(\lambda)$. Naturally, if we choose $u=u_l\in E_p$ in \eqref{2d2} with $k_p\ge k_0+2$, we conjecture to have $ RHS = C\e^{N-2+2k_p}(1+o(1))$.    However, $RHS$ of \eqref{2d2} contains the term  $\alpha_{1j,\e}\int_{\O_\e}u_1 V_{\e,u_l}$ and,  by Lemma \ref{lem1211-1}, $\int_{\O_\e}u_1 V_{\e,u_l}=O(\e^{N-1+k_0+k_p})$, which turns out not to be of lower order.   Luckily, by a refined computation of  $LHS$ and $RHS$, allow to get a sharp estimate of $\alpha_{ij,\e}$ (see Lemma \ref{thsa} and Lemma \ref{2thsa}). 
         \end{rem}
In view of the previous discussion, the proof of Theorem \ref{sim} in the case where the order of vanishing of $u$ at $P$ exceeds one reduces to estimating the coefficients $\alpha_{ij,\e}$. 

Since the notation is rather heavy, we believe it is better for the reader to first prove the result in the particular case where the space $E_i$ is one-dimensional, i.e. $dim(E_i)=1$, $i=0,1,\cdots p$ and $E_i=span\{u_i\}$. The general case follows along the same lines. So for any $j=1,2,\cdots,p$, the following estimate of $\alpha_{lj,\e}$, $l=0,1,\cdots,j-1$ holds.
 
\begin{lemma}\label{thsa}
If $N\ge3$, then
for any $j=1,\cdots,p$ and  any $l=0,1,\cdots,j-1$
    \begin{equation*}
    \alpha_{lj,\e}=\sum_{i=j}^{ p}\alpha_{ij,\e}O(\e^{1+{ k_i}-{ k_l}})+O(\e^{\frac{N-1}{2}+2{ k_j}-{ k_l}}).
\end{equation*}
\end{lemma}
\begin{proof}
Using  Lemma \ref{lem1211-1}, Lemma \ref{lem1511} and \eqref{eq15112025-2}, we obtain that  the $LHS$ and $RHS$ of \eqref{2d2} becomes
\begin{equation}\label{eq0612-1}
            \begin{aligned}
                LHS=&(\lambda_{j,\e}-\lambda)\int_{\O_\e}\left[ \sum_{i=1}^{p}\alpha_{ij,\e} (u_i-V_{\e,u_i})+ v_{j,\e}\right]u_l\\
                =&(\lambda_{j,\e}-\lambda)\left[\alpha_{lj,\e}\int_{\O_\e}u_l^2- \alpha_{lj,\e}\int_{\O_\e}V_{\e,u_l}u_l+\sum_{i\neq l}\alpha_{ij,\e} \int_{\O_\e}(u_i-V_{\e,u_i})u_l\right]\\
                =&(\lambda_{j,\e}-\lambda)\left[\alpha_{lj,\e}\left(1-\int_{B(P,\e)}u_l^2\right)- \alpha_{lj,\e}\int_{\O_\e}V_{\e,u_l}u_l-\sum_{i\neq l}\alpha_{ij,\e} \left(\int_{B(P,\e)}u_iu_l+\int_{\O_\e}V_{\e,u_i}u_l\right)\right]\\
                =&(\lambda_{j,\e}-\lambda)\left[\alpha_{lj,\e}\left(1+o(1)\right)+\sum_{i\neq l}\alpha_{ij,\e} O(\e^{N-1+{ k_i}+{ k_l}})\right],\\
            \end{aligned}
        \end{equation}
        and
        \begin{equation}\label{eq0612-2}
            \begin{aligned}
                RHS=&\big(\la+o(1)\big)\int_{\O_\e}\left[ \sum_{i=1}^{p}\alpha_{ij,\e} (u_i-V_{\e,u_i})+ v_{j,\e}\right]V_{\e,u_l}\\
				=&\big(\la+o(1)\big)\left[ \sum_{i=1}^{p}\alpha_{ij,\e}\int_{\O_\e}u_iV_{\e,u_l}- \sum_{i=1}^{p}\alpha_{ij,\e}\int_{\O_\e}V_{\e,u_i}V_{\e,u_l}+\int_{\O_\e}v_{j,\e}V_{\e,u_l}\right]\\
				=&\frac{(\la+o(1))}{\la }\e^{N-2+2{ k_l}}\alpha_{lj,\e}\Bigg((B_{ k_l}U_l)^TB_{ k_l}U_l+o(1)\Bigg)+\sum_{i\neq l}\alpha_{ij,\e}O(\e^{N-1+{ k_i}+{ k_l}})
                +|\lambda_{j,\e}-\lambda|O(\e^{\frac{N-1}{2}+{ k_l}}),
            \end{aligned}
        \end{equation}
       where $U_l$ and $B_{ k_l}$ are defined by \eqref{defU} and Lemma \ref{defG}, respectively. 
Seeing the fact that $B_{ k_l}^TB_{ k_l}$ is positive definite, it follows that $(B_{ k_l}U_l)^TB_{ k_l}U_l$ is a positive number. Combining $\lambda_{j,\e}-\lambda=C\e^{N-2+2{ k_j}}(1+o(1))$, we obtain 
\begin{equation}\label{eqRelationa}
    \alpha_{lj,\e}=\sum_{i\neq l}\alpha_{ij,\e}O(\e^{1+{ k_i}-{ k_l}})+O(\e^{\frac{N-1}{2}+2{ k_j}-{ k_l}}).
\end{equation}
For  $l=0$, we have
\begin{equation}\label{eq0712-3}
    \alpha_{0j,\e}= \sum_{i=1}^{p}\alpha_{ij,\e}O(\e^{1+{ k_i-k_0}})+O(\e^{\frac{N-1}{2}+2{ k_j-k_0}}),
\end{equation}
and for $l=1$
\begin{equation}\label{eq0712-4}
    \alpha_{1j,\e}=\alpha_{0j,\e}O(\e^{1+{ k_0}-{ k_1}})+{ \sum_{i=2}^{p}}\alpha_{ij,\e}O(\e^{1+{ k_i}-{ k_1}})+O(\e^{\frac{N-1}{2}+2{ k_j}-{ k_1}}).
\end{equation}
Substituting \eqref{eq0712-3} into \eqref{eq0712-4}, it follows that
\begin{equation}\label{eq0712-5}
    \alpha_{1j,\e}={ \sum_{i=2}^{p}}\alpha_{ij,\e}O(\e^{1+{ k_i}-{ k_1}})+O(\e^{\frac{N-1}{2}+2{ k_j}-{ k_1}}).
\end{equation}
And then, plugging \eqref{eq0712-5} into \eqref{eq0712-3}, we get
\begin{equation}\label{eq0712-6}
   \alpha_{0j,\e}={ \sum_{i=2}^{p}}\alpha_{ij,\e}O(\e^{1+{k_i-k_0}})+O(\e^{\frac{N-1}{2}+2{ k_j-k_0}}).
\end{equation}

 Repetition of the above step  several times leads to the conclusion that
\begin{equation}
    \alpha_{lj,\e}=\sum_{i=j}^{ p}\alpha_{ij,\e}O(\e^{1+{ k_i}-{ k_l}})+O(\e^{\frac{N-1}{2}+2{ k_j}-{ k_l}}),\ \ for \ \ any\ \ l=0,\cdots,j-1.
\end{equation}
\end{proof}
     \begin{lemma}\label{2thsa}
If $N=2$,
for any $j=1,\cdots,p$, then it follows that
   \begin{equation}\label{eq0301-2}
 \alpha_{0j,\e}=
    \begin{cases}
       \displaystyle{ \sum_{i=j}^{p}}\alpha_{ij,\e}O(\e^{1+{ k_i}}|\log\e|)+O(\e^{2{ k_j}}), &if\ k_0=0,\\
      \displaystyle  { \sum_{i=j}^{p}}\alpha_{ij,\e}O(\e^{1+{k_i-k_0}})+O(\e^{2{ k_j}-{ k_0}+1}|\log\e|), &if\ k_0\ge 1,\\
    \end{cases}
\end{equation}
and
\begin{equation}
    \alpha_{lj,\e}={ \sum_{i=j}^{p}}\alpha_{ij,\e}O(\e^{1+{ k_i-k_l}})+O(\e^{2{ k_j}-{ k_l}+1}|\log\e|),\ \ l=1,\cdots,j-1.
\end{equation}
\end{lemma}
\begin{proof}
Using  Lemma \ref{lem1211-1}, Lemma \ref{lem1511} and  \eqref{eq15112025-2},  we obtain that for any $l=0,\cdots,j-1$, the $LHS$ of \eqref{2d2} becomes 
\begin{equation}\label{eq0201-1}
            \begin{aligned}
                LHS=&(\lambda_{j,\e}-\lambda)\int_{\O_\e}\left[\sum_{i=0}^{p}\alpha_{ij,\e} (u_i-V_{\e,u_i})+ v_{j,\e}\right]u_l\\
                =&(\lambda_{j,\e}-\lambda)\left[\alpha_{lj,\e}\int_{\O_\e}u_l^2- \alpha_{lj,\e}\int_{\O_\e}V_{\e,u_l}u_l+\sum_{i\neq l}\alpha_{ij,\e} \int_{\O_\e}(u_i-V_{\e,u_i})u_l\right]\\
                =&(\lambda_{j,\e}-\lambda)\left[\alpha_{lj,\e}\left(1-\int_{B(P,\e)}u_l^2\right)- \alpha_{lj,\e}\int_{\O_\e}V_{\e,u_l}u_l-\sum_{i\neq l}\alpha_{ij,\e} \left(\int_{B(P,\e)}u_iu_l+\int_{\O_\e}V_{\e,u_i}u_l\right)\right]\\
                 =&(\lambda_{j,\e}-\lambda)\left[\alpha_{lj,\e}\left(1+o(1)\right)+\sum_{i\neq l}\alpha_{ij,\e} O(\e^{1+{ k_i+k_l}})\right].\\
            \end{aligned}
        \end{equation}
        In the case  $N=2$, it is necessary to compute  $RHS$ of \eqref{2d2} separately, when  $k_0=0$ and $l=0$. If  $k_0=0$ and $l=0$, then
          \begin{equation}\label{eq0201-2}
            \begin{aligned}
                RHS=&\big(\la+o(1)\big)\int_{\O_\e}\left({ \sum_{i=0}^{p}}\alpha_{ij,\e} (u_i-V_{\e,u_i})+ v_{j,\e}\right)V_{u_l}\\
				=&\frac{(\la+o(1))}{\la |\log\e|}2\pi u_l^2(P)\alpha_{lj,\e}\left(1+o(1)\right)
        +{ \sum_{i=1}^{p}}\alpha_{ij,\e}O(\e^{1+{ k_i}})+O\left(\frac{\e^{2{ k_j}}}{|\log\e|}\right).\\
            \end{aligned}
        \end{equation}
        In all other cases, namely when  $k_0=0$ and  $l=1,\cdots,j-1$, or when $ k_0\ge1$ and  $l=0,\cdots,j-1$, we have
        \begin{equation}\label{eq0201-R2}
            \begin{aligned}
                RHS=&\big(\la+o(1)\big)\int_{\O_\e}\left({ \sum_{i=1}^{p}}\alpha_{ij,\e} (u_i-V_{\e,u_i})+ v_{j,\e}\right)V_{\e,u_l}\\
				=&\frac{(\la+o(1))}{\la }\e^{2{ k_l}}\alpha_{lj,\e}\Bigg((B_{ k_l}U_l)^TB_{ k_l}U_l+o(1)\Bigg)+\sum_{i\neq l}\alpha_{ij,\e}O(\e^{1+{ k_i}+{ k_l}})+O(\e^{2{ k_j}+{ k_l}+1}|\log\e|),\\
            \end{aligned}
        \end{equation}
          where $U_l$ and $B_{ k_l}$ are defined by \eqref{defU} and Lemma \ref{defG}, respectively.
 Seeing the fact that $B_{ k_l}^TB_{ k_l}$ is positive definite, it follows that $(B_{ k_l}U_l)^TB_{ k_l}U_l$ is a positive number. Then, we obtain 
\begin{equation}\label{eq0201-3}
    \alpha_{0j,\e}=
    \begin{cases}
       \displaystyle{ \sum_{i=1}^{p}}\alpha_{ij,\e}O(\e^{1+{ k_i}}|\log\e|)+O(\e^{2{ k_j}}), &if\ k_0=0,\\
      \displaystyle  { \sum_{i=1}^{p}}\alpha_{ij,\e}O(\e^{1+{ k_i-k_0}})+O(\e^{2{ k_j}-{ k_0}+1}|\log\e|), &if\ k_0\ge 1,\\
    \end{cases}
\end{equation}
\begin{equation}
    \alpha_{lj,\e}=\sum_{i\neq l}\alpha_{ij,\e}O(\e^{1+{ k_i-k_l}})+O(\e^{2{ k_j}-{ k_l}+1}|\log\e|),\ \ l=1,\cdots,j-1.
\end{equation}
For $l=1$,
\begin{equation}\label{eq0201-4}
    \alpha_{1j,\e}=\alpha_{0j,\e}O(\e^{1+{ k_0-k_1}})+{\sum_{i=2}^{p}}\alpha_{ij,\e}O(\e^{1+{ k_i-k_1}})+O(\e^{2{ k_j}-{ k_1}+1}|\log\e|).
\end{equation}
Substituting \eqref{eq0201-3} into \eqref{eq0201-4}, it follows that
\begin{equation}\label{eq0201-5}
    \alpha_{1j,\e}={ \sum_{i=2}^{p}}\alpha_{ij,\e}O(\e^{1+{ k_i-k_1}})+O(\e^{2{ k_j}-{ k_1}+1}|\log\e|).
\end{equation}
And then plugging \eqref{eq0201-5} into \eqref{eq0201-3}, we get
\begin{equation}\label{eq0201-6}
   \alpha_{0j,\e}=
    \begin{cases}
       \displaystyle{ \sum_{i=2}^{p}}\alpha_{ij,\e}O(\e^{1+{ k_i}}|\log\e|)+O(\e^{2{ k_j}}), &if\ k_0=0,\\
      \displaystyle  { \sum_{i=2}^{p}}\alpha_{ij,\e}O(\e^{1+{ k_i-k_0}})+O(\e^{2{ k_j}-{ k_0}+1}|\log\e|), &if\ k_0\ge 1,\\
    \end{cases}
\end{equation}

By repeating the above step  several times, we conclude that
\begin{equation*}
 \alpha_{0j,\e}=
    \begin{cases}
       \displaystyle{ \sum_{i=j}^{p}}\alpha_{ij,\e}O(\e^{1+{ k_i}}|\log\e|)+O(\e^{2{ k_j}}), &if\ k_0=0,\\
      \displaystyle  { \sum_{i=j}^{p}}\alpha_{ij,\e}O(\e^{1+{ k_i-k_0}})+O(\e^{2{ k_j}-{ k_0}+1}|\log\e|), &if\ k_0\ge 1,\\
    \end{cases}
\end{equation*}
and
\begin{equation*}
    \alpha_{lj,\e}={ \sum_{i=j}^{p}}\alpha_{ij,\e}O(\e^{1+{ k_i-k_l}})+O(\e^{2{ k_j}-{ k_l}+1}|\log\e|),\ \ l=1,\cdots,j-1.
\end{equation*}
\end{proof}

If $dim(E_{k_l})=1$, then from  the above calculations, we see that, for example in the case $N\ge 3$, the leading term of  $RHS$  is $\alpha_{lj,\e}(B_{ k_l}U_l)^TB_{ k_l}U_l\e^{N-2+2{ k_l}}$. Since $\mathcal{C}_{k_l}=B_{k_l}^TB_{k_l}$ is positive definite and $U_l\neq (0,\cdots,0)$, it follows that $(B_{ k_l}U_l)^TB_{ k_l}U_l$  is a positive number. Then we get the relationship \eqref{eqRelationa} between $\alpha_{ij,\e}$.  For the usual case, the leading term of $RHS$ becomes $\sum_{i=1+d_0+\cdots+d_{L-1}}^{d_0+\cdots+d_L}\alpha_{ij,\e}(B_{k_L}U_i)^TB_{k_L}U_l \e^{N-2+2k_L}$. At this point, the invertibility of the matrix $G_{k_L}=\Big((B_{k_L}U_i)^TB_{k_L}U_l\Big)$ plays a key role in deriving the relations among $\alpha_{ij,\e}$.

\begin{lemma}\label{lem2204}
 Assume that $J=1,\cdots,p$ with $1+d_0+\cdots+d_{J-1}\leq j\leq d_0+\cdots d_{J}$ and $0\leq L\leq J-1$ with $1+d_0+\cdots+d_{L-1}\leq l\leq d_0+\cdots d_{L}$. Then we have
    \begin{equation}\label{eq2301-3}
    \begin{cases}
        \displaystyle\alpha_{lj,\e}=\sum_{I=J}^p\sum_{i=1+d_0+\cdots+d_{I-1}}^{d_0+\cdots+d_I}\alpha_{ij,\e}O(\e^{1+k_I-k_L})+O(\e^{\frac{N-1}{2}+2k_J-k_L})&\ for \ \ N\ge3,\\
\displaystyle\alpha_{lj,\e}=\sum_{I=J}^p\sum_{i=1+d_0+\cdots+d_{I-1}}^{d_0+\cdots+d_I}\alpha_{ij,\e}O(\e^{1+k_I}|\log\e|)+O(\e^{2k_J}) &\ for \ \ N=2\ and\ k_L=0,\\
    \displaystyle\alpha_{lj,\e}=\sum_{I=J}^p\sum_{i=1+d_0+\cdots+d_{I-1}}^{d_0+\cdots+d_I}\alpha_{ij,\e}O(\e^{1+k_I-k_L})+O(\e^{2k_J-k_L+1}|\log\e|)&\ for \ \ N=2\ and\  k_L\ge 1.\\
    \end{cases}
\end{equation}
\end{lemma}
\begin{proof}

Now, we only consider $N\ge3$. The estimates of $N=2$  follow by the same argument.
By a similar computation as above Lemmas, we obtain that for $1+d_0+\cdots+d_{L-1}\leq l\leq d_0+\cdots d_{L}$
    \begin{equation*}
     \begin{aligned}
    LHS =&(\lambda_{j,\e}-\lambda)\left[\alpha_{lj,\e}\left(1+o(1)\right)+\sum_{i=1+d_0+\cdots+d_{L-1}}^{d_0+\cdots+d_L}\alpha_{ij,\e} O(\e^{N-2+2k_L})+\sum_{I\neq L}^p\sum_{i=1+d_0+\cdots+d_{I-1}}^{d_0+\cdots+d_I}\alpha_{ij,\e} O(\e^{N-1+k_I+k_L})\right]\\
                RHS=
				&\frac{(\la+o(1))}{\la }\e^{N-2+2k_L}\sum_{i=1+d_0+\cdots+d_{L-1}}^{d_0+\cdots+d_L}\alpha_{ij,\e}\Bigg((B_{k_L}U_i)^TB_{k_L}U_l+o(1)\Bigg)\\
                &+\sum_{I\neq L}^p\sum_{i=1+d_0+\cdots+d_{I-1}}^{d_0+\cdots+d_I}\alpha_{ij,\e} O(\e^{N-1+k_I+k_L})
                +O(\e^{N-2+2k_J+\frac{N-1}{2}+k_L}),\\
                \end{aligned}
\end{equation*} 
 where $U_l$ and $B_{ k_l}$ are defined by \eqref{defU} and Lemma \ref{defG}, respectively.
If the matrix $G_{k_L}=\Big((B_{k_L}U_i)^TB_{k_L}U_l\Big)$ is invertible, then we also obtain  
$$
\alpha_{ij,\e}=\sum_{I\neq L}^p\sum_{i=1+d_0+\cdots+d_{I-1}}^{d_0+\cdots+d_I}\alpha_{ij,\e}O(\e^{1+k_I-k_L})+O(\e^{2k_J+\frac{N-1}{2}-k_L}).
$$
The results are obtained by the same method as in Lemma \ref{thsa}.

Now, let us focus on the matrix $G_{k_L}$. First, $G_{k_L}$ is a Gram matrix. By the property of the Gram matrix, we only need to prove that $\{B_{k_L}U_{1+d_0+\cdots+d_{L-1}},\cdots,B_{k_L}U_{d_0+\cdots+d_{L}}\}$ are linearly independent. Indeed, suppose that there exist $c_{1+d_0+\cdots+d_L},\cdots,c_{d_0+\cdots+d_L}\in \R$ such that
        \begin{equation}\label{eq0414-4}
        c_{1+d_0+\cdots+d_{L-1}}B_{k_L}U_{1+d_0+\cdots+d_{L-1}}+\cdots+c_{d_0+\cdots+d_{L}}B_{k_L}U_{d_0+\cdots+d_{L}}=0.  
        \end{equation}
        Let 
        $$
        f_{L}(x)=c_{1+d_0+\cdots+d_{L-1}}u_{1+d_0+\cdots+d_{L-1}}(x)+\cdots+c_{d_0+\cdots+d_{L}}u_{d_0+\cdots+d_{L}}(x).
        $$
        Then $f_{L}\in E_L$. By  assumption \eqref{eq0414-4}, the definition of $U_i$ and the fact that $B_{k_L}$ is invertible, we see that $D^\alpha f_{L}(P)=0$ for any $|\alpha|\leq k_L$. 
       Thus,
         $f_{L}$ must be zero by the definition of $E_{L}$. Then $c_{1+d_0+\cdots+d_{L-1}}=\cdots=c_{d_0+\cdots+d_L}=0$, which deduces that $\{B_{k_L}U_{1+d_0+\cdots+d_{L-1}},\cdots,B_{k_L}U_{d_0+\cdots+d_{L}}\}$ are linearly independent.  This completes the proof.
\end{proof}
\begin{proof}[Proof of Theorem \ref{sim}]
We only  consider the case $N\ge3$, the result of $N=2$ can be obtained by the same procedures. 
    If  $E(\lambda)= E_0\oplus\cdots\oplus E_p,\ p\ge 1$, we repeat the proof of   Corollaries \ref{cor4.2} and \ref{cor1510-1}.
     We obtain that, for $1+d_0+\cdots+d_{L-1}\leq l\leq d_0+\cdots d_{L}$, 
      \begin{equation}
     \begin{aligned}
    LHS =&(\lambda_{j,\e}-\lambda)\left[\alpha_{lj,\e}\left(1+o(1)\right)+\sum_{i=1+d_0+\cdots+d_{L-1}}^{d_0+\cdots+d_L}\alpha_{ij,\e} O(\e^{N-2+2k_L})+\sum_{I\neq L}^p\sum_{i=1+d_0+\cdots+d_{I-1}}^{d_0+\cdots+d_I}\alpha_{ij,\e} O(\e^{N-1+k_I+k_L})\right]\\
                RHS=
				&\frac{(\la+o(1))}{\la }\e^{N-2+2k_L}\sum_{i=1+d_0+\cdots+d_{L-1}}^{d_0+\cdots+d_L}\alpha_{ij,\e}\Bigg((B_{k_L}U_i)^TB_{k_L}U_l+o(1)\Bigg)\\
                &+\sum_{I\neq L}^p\sum_{i=1+d_0+\cdots+d_{I-1}}^{d_0+\cdots+d_I}\alpha_{ij,\e} O(\e^{N-1+k_I+k_L})
                +O(\e^{N-2+2k_J+\frac{N-1}{2}+k_L})\\
                \end{aligned}
\end{equation}
First,  choosing  $l=1,\cdots,d_0$, and let $\e\to 0$, it follows that
\begin{equation}
    \lim_{\e\to0}\frac{\lambda_{j,\e}-\lambda}{\e^{N-2{ +2k_0}}}\alpha_{lj}=\alpha_{lj}\sum_{i=1}^{d_0}(B_{k_0}U_i)^TB_{k_0}U_l.
\end{equation}
  Then, for $ l={ 1+d_0},\cdots,m(\lambda)$, 
\begin{equation}
    \lim_{\e\to0}\frac{\lambda_{j,\e}-\lambda}{\e^{N-2{ +2k_0}}}\alpha_{lj}=0.
\end{equation}
Since the vectors $\vec{\alpha}_1,..\vec{\alpha}_{m(\la)}$ are orthogonal pairwise and $||\vec{\alpha}_j||=1$,   we may assume without loss of generality that the submatrix $(\alpha_{lj})_{l,j=1,\cdots d_0}$ is invertible. Then we see that $\lim_{\e\to0}\frac{\lambda_{j,\e}-\lambda}{\e^{N-2+2k_0}}$
are eigenvalues of the invertible matrix $G_{k_0}=\left((B_{k_0}U_i)^TB_{k_0}U_l\right)$ and  $(\alpha_{1j},\cdots,\alpha_{d_0j}),$ $ j=1,\cdots,d_0$, are the corresponding eigenvectors. Moreover $\alpha_{lj}=0,\  l=1+d_0,\cdots,m(\lambda),\  j=1,\cdots,d_0.$ Thus, \eqref{L2approx} holds for $j=1,\cdots,d_0.$ In particular, if $(u_1(P),..,u_l(P))\neq(0,..,0)$, then $k_0=0$, which yields $d_0=1$. It follows that $u_{1,\e}\to u_1$ in $L^2(\O)$.
 
Next, we consider $l=1+d_0,\cdots,d_0+d_1$ and $j=1+d_0,\cdots,m(\lambda)$. 
Using \eqref{eq2301-3}, we see
\begin{equation}\label{eq1205-3}
\begin{aligned}
      RHS=&
				\frac{(\la+o(1))}{\la }\e^{N-2+2k_1}\sum_{i=1+d_0}^{d_0+d_1}\alpha_{ij,\e}\Bigg((B_{k_1}U_i)^TB_{k_1}U_l+o(1)\Bigg)\\
                &+\sum_{I=2}^p\sum_{i=1+d_0+\cdots+d_{I-1}}^{d_0+\cdots+d_I}\alpha_{ij,\e}O(\e^{N-1+{ k_I}+{ k_1}})+O(\e^{N-2+2k_J+\frac{N-1}{2}+k_1}),\\
                \end{aligned}
\end{equation}
which  implies that 
\begin{equation}
    \lim_{\e\to0}\frac{\lambda_{j,\e}-\lambda}{\e^{N-2+2k_1}}\alpha_{lj}=\sum_{i=1+d_0}^{d_0+d_1}\alpha_{ij}(B_{k_1}U_i)^TB_{k_1}U_l,\  l=1+d_0,\cdots, d_0+d_1,
\end{equation}
and 
\begin{equation}
    \lim_{\e\to0}\frac{\lambda_{j,\e}-\lambda}{\e^{N-2+2k_1}}\alpha_{lj}=0,\ \ l=1+d_0+d_1,\cdots,m(\lambda).
\end{equation}
 WLOG, we assume  the submatrix $(\alpha_{lj})_{l,j=1+d_0,\cdots d_0+d_1}$ is invertible. Then we see that $\lim_{\e\to0}\frac{\lambda_{j,\e}-\lambda}{\e^{N-2+2k_1}}$
 are eigenvalues of the invertible matrix $G_{k_1}:=\left((B_{k_1}U_i)^TB_{k_1}U_l\right)$ and  $(\alpha_{(1+d_0)j},\cdots,\alpha_{(d_0+d_1)j}),$ $ j=1+d_0,\cdots,d_0+d_1$, are the corresponding eigenvectors. Finally, $\alpha_{lj}=0$ for any $  l=1+d_0+d_1,\cdots,m(\lambda),\  j=1+d_0,\cdots,d_0+d_1.$  Furthermore, combining Lemma \ref{lem2204}, we obtain that $\alpha_{lj}=0$, for any $l=1,\cdots,d_0,\  j=1+d_0,\cdots,d_0+d_1.$ Thus, \eqref{L2approx} holds for $j=1+d_0,\cdots,d_0+d_1.$
 
   Repeating the above procedures, we obtain the result step by step.
\end{proof}

In next lemma we show that the number $\Lambda_j$ appearing in \eqref{d12} of Theorem \ref{sim}  are independent of the choice of the eigenspace $E(\la)$.

 \begin{lemma}\label{ind}
     In Theorem \ref{sim},  $\Lambda_j$  are independent of the 
     choice of the orthonormal basis $\{u_1,\cdots,u_{m(\la)}\}$ of $E(\lambda)$. 
     \end{lemma}
     \begin{proof}
    Assume that \[E(\lambda)=E_0\oplus \dots \oplus E_p,\] 
    and $\{e_{1+d_0+\cdots+d_{J-1}},\cdots,e_{d_0+\cdots+d_J}\}$ is the orthonormal basis of the subspace $E_J$ for any $J=0,1,\cdots,p$.  The eigenvalues of the matrix   $\tilde{G}_{k_J}$
    $$	\!\!\!\!\!\!\!\!\!\!\!\!\!\!\!\!\!\!\!\!\!\!\!
\tilde{G}_{k_J}=\begin{bmatrix}
\displaystyle\sum_{\substack{|\alpha|={ k_J}\\
           |\beta|={ k_J}}}\frac{1}{\alpha!\beta!}\frac{\partial^{{ k_J}} e_{1+..+d_{J-1}} (P)}{\partial x_{1}^{\alpha_1}\cdots\partial x_{N}^{\alpha_N}}\frac{\partial^{{ k_J}} e_{1+..+d_{J-1}} (P)}{\partial x_{1}^{\beta_1}\cdots\partial x_{N}^{\beta_N}} C_{\alpha\beta}&\dots & \displaystyle\sum_{\substack{|\alpha|={ k_J}\\
           |\beta|={ k_J}}}\frac{1}{\alpha!\beta!}\frac{\partial^{{ k_J}} e_{1+..+d_{J-1}} (P)}{\partial x_{1}^{\alpha_1}\cdots\partial x_{N}^{\alpha_N}}\frac{\partial^{{ k_J}} e_{d_0+..+d_J} (P)}{\partial x_{1}^{\beta_1}\cdots\partial x_{N}^{\beta_N}} C_{\alpha\beta}\\
    \vdots&  &\vdots\\
  \displaystyle \sum_{\substack{|\alpha|={ k_J}\\
           |\beta|={ k_J}}}\frac{1}{\alpha!\beta!}\frac{\partial^{{ k_J}} e_{1+..+d_{J-1}} (P)}{\partial x_{1}^{\alpha_1}\cdots\partial x_{N}^{\alpha_N}}\frac{\partial^{{ k_J}} e_{d_0+..+d_J} (P)}{\partial x_{1}^{\beta_1}\cdots\partial x_{N}^{\beta_N}} C_{\alpha\beta} &\dots & \displaystyle\sum_{\substack{|\alpha|={ k_J}\\
           |\beta|={ k_J}}}\frac{1}{\alpha!\beta!}\frac{\partial^{{ k_J}} e_{d_0+..+d_J} (P)}{\partial x_{1}^{\alpha_1}\cdots\partial x_{N}^{\alpha_N}}\frac{\partial^{{ k_J}} e_{d_0+..+d_J} (P)}{\partial x_{1}^{\beta_1}\cdots\partial x_{N}^{\beta_N}} C_{\alpha\beta}\\
\end{bmatrix}.
$$
         
         Since $\{e_{1+d_0+\cdots+d_{J-1}},\cdots,e_{d_0+\cdots+d_J}\}$ and $\{u_{1+d_0+\cdots+d_{J-1}},\cdots,u_{d_0+\cdots+d_J}\}$ are the  orthonormal basis of the subspace $E_J$, there exists an Orthogonal matrix $O_{k_J}$ such that 
       $$
       (u_{1+d_0+\cdots+d_{J-1}},\cdots,u_{d_0+\cdots+d_J})^T=O_{k_J}(e_{1+d_0+\cdots+d_{J-1}},\cdots,e_{d_0+\cdots+d_J})^T,
       $$
which follows that 
$$
G_{k_J}=O_{k_J}\tilde{G}_{k_J}O_{k_J}^T.
$$
Then the eigenvalues of $G_{k_J}$ and $\tilde{G}_{k_J}$ are the same.
 \end{proof}   

\begin{rem}\label{rem-withoutalm}
    To simplify the computation, we use the convergent rate of $\lambda_{j,\e}-\lambda$ obtained by \cite{alm1,alm2}. Actually, without using this tool, using our method, Theorem \ref{sim} also holds. Here, we only give a brief sketch of the proof.  

As an example, consider the case $N=3$.
    Recall that for any $j=1,\cdots,d_0$, we obtain the result  even if we do not know  the convergent rate of $\lambda_{j,\e}-\lambda$ and we see that $\lambda_{i,\e}-\lambda=o(\e^{N-2+2k_0})$ for any $i=1+d_0,\cdots,m(\lambda)$. Then, we see that the result holds for $J=0$. 

    In the next step, we consider $J=1$ and $j=1+d_0,\cdots,d_0+d_1$. Following the same computations in the proof of Lemma \ref{lem2204} and keeping $\lambda_{j,\e}-\lambda$, we obtain the decay rate of $\alpha_{lj,\e}$ for any $l=1,\cdots,d_0$ as following 
    \begin{equation}\label{eq1205-1}
         \alpha_{lj,\e}=\sum_{I=1}^p\sum_{i=1+d_0+\cdots+d_{I-1}}^{d_0+\cdots+d_I}\alpha_{ij,\e}O(\e^{1+k_I-k_0})+|\lambda_{j,\e}-\lambda|O(\e^{\frac{3-N}{2}-k_0}).
    \end{equation}
  Plugging \eqref{eq1205-1} into $RHS$ of \eqref{2d2}, we obtain 
    \begin{equation}\label{eq1205-2}
\begin{aligned}
      RHS=&
				\frac{(\la+o(1))}{\la }\e^{N-2+2k_1}\sum_{i=1+d_0}^{d_0+d_1}\alpha_{ij,\e}\Bigg((B_{k_1}U_i)^TB_{k_1}U_l+o(1)\Bigg)\\
                &+\sum_{I=2}^p\sum_{i=1+d_0+\cdots+d_{I-1}}^{d_0+\cdots+d_I}\alpha_{ij,\e}O(\e^{N-1+{ k_I}+{ k_1}})+|\lambda_{j,\e}-\lambda|O(\e^{\frac{N-1}{2}+k_1}),\\
                \end{aligned}
\end{equation}
Substituting \eqref{eq1205-2} for \eqref{eq1205-3}, the result holds for $J=1$.  After that, the proof continues as before.
\end{rem}

	\section{Quantitative estimates of eigenfunctions}\label{sec-function}
As a preliminary step, we begin by proving our results in the case where $\la$ is a simple eigenvalue. The case of a multiple eigenvalue will be handled using the same approximation as in the previous section.

\subsection{The case of simple eigenvalues}
    In this Section, we assume that $\lambda$ is simple and the space $H_0^1(\O_\e)$ will be regarded as a subspace of $H_0^1(\O)$. 	  

Let us give an idea of the proof. 

Given $u\in E(\lambda)$, the function $u-V_{\e,u}$ is the best approximation of $u$ in the space $H^1_0(\Omega_\e)$. Hence we write
\begin{equation}\label{ort}
u=\psi_\e+V_{\e,u}
\end{equation}
then, by the definition of $V_{\e,u}$ we get that $\int_\O\nabla\psi_\e\nabla V_{\e,u}=0$.

So \eqref{ort} provides an {\em orthogonal decomposition} of $u$. Next aim is to gives some properties of $\psi_\e$. 

Moving from $H^1_0(\Omega_\e)$ into $E(\lambda_\e)$, we take the orthogonal projection
	\[
	\Pi_\e: \ L^2(\Omega_\e) \to E(\lambda_\e),
	\]
so that $\Pi_\e(u-V_{\e,u})$ is the best approximation of $u-V_{\e,u}$ in the space $E(\lambda_\e)$.

Using the Spectral Theorem, the distance of $\psi_\e$ to $\Pi_\e\psi_\e$ can be computed. Since $dim (E(\lambda_\e))=1$, it is easy to see that 
      \begin{equation}\label{into-u-e}
          u_\e=\pm \frac{\Pi_\e\psi_\e}{||\Pi_\e\psi_\e||_{L^2(\Omega_\e)}}.
      \end{equation} Thus, for the simple case, if the estimates of $V_{\e,u}$ are given, then the estimates of corresponding perturbed eigenfunction $u_\e$ follows.\\
\begin{tikzpicture}[scale=1.2]

\fill[gray!10] (-3,-1) -- (3,-1) -- (4,1.5) -- (-2,1.5) -- cycle;
\draw[thick] (-3,-1) -- (3,-1) -- (4,1.5) -- (-2,1.5) -- cycle;
\node at (0,1) {$H_0^1(\O_\e)$};
\node at (-2,-0.32) {$E(\lambda_\varepsilon)$};

\draw[line width=1.2pt] (-1.5,-0.32) -- (2.7,-0.32);

\coordinate (p) at (1.2,-0.32);
\fill (p) circle (2.2pt);
\node[below] at (p) {$\Pi_\e \psi_\varepsilon$};

\coordinate (ue) at (1.9,0.25);
\fill (ue) circle (2.2pt);
\node[right] at (ue) {$\psi_\varepsilon$};

\coordinate (u) at (1.9,2.2);
\fill (u) circle (2.2pt);
\node[right] at (u) {$u$};

\draw[dashed] (u) -- (ue);
\draw[dashed] (ue) -- (p);


\end{tikzpicture}\\

We start by proving the following.
	\begin{prop}\label{p:main}
	Denote $\psi_\e:=u-V_{\e,u}$.	We obtain  
\begin{equation}\label{est1}
	|| u-\Pi_\e(\psi_\e) ||_{L^2(\Omega_\e)}=O(|| V_{\e,u} ||_{L^2(\Omega_\e)}) \quad \text{as }\e\to0,
\end{equation}
and
\begin{equation}\label{eq:normL2pepsie}
||\Pi_\e(\psi_\e)||_{L^2(\Omega_\e)} = 1 +O(|| V_{\e,u} ||_{L^2(\Omega_\e)}) \quad \text{as }\e\to0.
\end{equation}
\end{prop}
\begin{proof}
		We recall here the argument contained in \cite{afhl}.
		Let us consider the so-called $u$-capacitary potential $V_{\e,u}$ defined in \eqref{eq:Veproblem}. By \eqref{eq:eigenvalueproblem}
		and \eqref{eq:Veproblem} we have that $\psi_\e$ is a solution to 
		\begin{equation*}
			\begin{cases}
				-\Delta \psi_\e= \lambda u &\text{in }\Omega_\e\\
				\psi_\e=0 &\text{on }\partial \Omega_\e.
			\end{cases}
		\end{equation*}
		From now on, let us denote $-\Delta_\e$ the standard Dirichlet Laplacian over $\Omega_\e$. Recalling that $u=\psi_\e + V_{\e,u}$, we have
		\[
		(-\Delta_\e - \lambda)\psi_\e = \lambda V_{\e,u}, 
		\]
		so that 
		\begin{equation}\label{eq:oppsie}
			|| (-\Delta_\e - \lambda)\psi_\e ||_{L^2(\Omega_\e)} = 
			\lambda||  V_{\e,u} ||_{L^2(\Omega_\e)} .
		\end{equation}
		In order to prove that $\psi_\e$ is very close to the space $E(\lambda_\e)$, we take into account the projection $\Pi_\e$ defined above and study 
		\[
		\Tilde{u}_\e:= \psi_\e - \Pi_\e(\psi_\e).
		\]
		Firstly, we note that $ \Pi_\e(\psi_\e)$ is in fact an eigenfunction of $\lambda_\e$, so that 
		\[
		(-\Delta_\e - \lambda_\e) \Tilde{u}_\e = (-\Delta_\e - \lambda_\e) \psi_\e
		\]
		and, adding and subtracting $\lambda\psi_\e$, by \eqref{eq:oppsie} we obtain 
		\begin{align*}
			|| (-\Delta_\e - \lambda_\e)\Tilde{u}_\e ||_{L^2(\Omega_\e)} 
			&\leq|\lambda - \lambda_\e|\, || \psi_\e||_{L^2(\Omega_\e)}
			+ || (-\Delta_\e - \lambda)\psi_\e ||_{L^2(\Omega_\e)}\\
			&=|\lambda - \lambda_\e|\, || \psi_\e||_{L^2(\Omega_\e)}
			+ \lambda || V_{\e,u} ||_{L^2(\Omega_\e)}\\
			&=O(|| V_{\e,u} ||_{L^2(\Omega_\e)}) \quad \text{as }\e\to0,
		\end{align*}
		where in the last equality we have taken into account \eqref{eq:eigenvaluesbehavior}, \eqref{eq:N=2eigenvaluesbehavior}, Lemma \ref{l:improvedL2Ve} and Lemma \ref{2l:improvedL2Ve}.  
		The function $\Tilde{u}_\e \in \mathrm{ker}\Pi_\e$ and the spectrum $\sigma({-\Delta_\e}|_{ \mathrm{ker}\Pi_\e})=\sigma(-\Delta_\e) \setminus \{\lambda_\e\}$, in such a way that $\mathrm{dist}\big(\sigma({-\Delta_\e}|_{ \mathrm{ker}\Pi_\e}), \lambda_\e\big)\geq\delta>0$ for $\e$ is sufficiently small, due to the simplicity assumption on $\lambda$.
		Then by the Spectral Theorem (see Theorem \ref{th_spectral})
		\begin{equation}\label{eq:normtildeue}
			|| \Tilde{u}_\e ||_{L^2(\Omega_\e)} \leq \dfrac{|| (-\Delta_\e - \lambda_\e)\Tilde{u}_\e ||_{L^2(\Omega_\e)}}{\mathrm{dist}\big(\sigma({-\Delta_\e}|_{ \mathrm{ker}\Pi_\e}), \lambda_\e\big)}
			=O(|| V_{\e,u} ||_{L^2(\Omega_\e)}) \quad \text{as }\e\to0.
		\end{equation}
		Finally, from the latter equation we obtain
		\begin{align*}
			|| u-\Pi_\e(\psi_\e) ||_{L^2(\Omega_\e)} &\leq || u-\psi_\e||_{L^2(\Omega_\e)} + ||\psi_\e- \Pi_\e(\psi_\e) ||_{L^2(\Omega_\e)} \\
			&= ||V_{\e,u}||_{L^2(\Omega_\e)} + ||\Tilde{u}_\e||_{L^2(\Omega_\e)} \\
			&=O(|| V_{\e,u} ||_{L^2(\Omega_\e)}) \quad \text{as }\e\to0,
		\end{align*}
		which in particular implies 
        	\begin{equation*}
            \begin{aligned}
                ||\Pi_\e(\psi_\e)||_{L^2(\Omega_\e)} = &||u||_{L^2(\Omega_\e)} + O(|| V_{\e,u} ||_{L^2(\Omega_\e)})\\
                =&1+O(||u||_{L^2( B(P,\e))}) + O(|| V_{\e,u} ||_{L^2(\Omega_\e)})\quad \text{as }\e\to0,
            \end{aligned}
		\end{equation*}
meaning in particular that the eigenfunction $\Pi_\e (u-V_{\e,u})\not\equiv0$. 

Finally, by Lemma \ref{l:improvedL2Ve} and \ref{2l:improvedL2Ve},  we obtain  that $\|u\|_{L^2(B(P,\e))} = O(\e^{k+\frac{N}{2}})$ we immediately see that 
\begin{equation}\label{eq:negligible}
\|u\|_{L^2(B(P,\e))} = O(\|V_{\e,u}\|_{L^2(\Omega_\e)}) \quad \text{as }\e\to0.
\end{equation}

The last estimate completes the proof of \eqref{est1} and \eqref{eq:normL2pepsie}. 
\end{proof}

Thanks to Proposition \ref{p:main}, we know that $\Pi_\e(u-V_{\e,u})$ is a nontrivial function generating $E(\lambda_\e)$. From now on, for $\e$ small enough, we denote
\begin{equation}\label{eq1405-1}
    u_\e := \dfrac{\Pi_\e (u-V_{\e,u})}{\|\Pi_\e (u-V_{\e,u})\|_{L^2(\Omega_\e)}}.
\end{equation}

\begin{cor}\label{l:s2}
As $\e\to0$, we have
\begin{equation}\label{s2}
			||u_\e-u||_{L^2(\Omega_\e)}=O(||V_{\e,u}||_{L^2(\Omega_\e)}).
\end{equation}
\end{cor}
\begin{proof}
		The proof relies essentially on the Triangle inequality:
		\begin{align*}
			||u_\e-u||_{L^2(\Omega_\e)}
			&= \Big\|\dfrac{\Pi_\e(\psi_\e)}{||\Pi_\e(\psi_\e)||_{L^2(\Omega_\e)}}-u\Big\|_{L^2(\Omega_\e)} 
			= \dfrac{\Big\|\Pi_\e(\psi_\e)- ||\Pi_\e(\psi_\e)||_{L^2(\Omega_\e)} u\Big\|_{L^2(\Omega_\e)}}{||\Pi_\e(\psi_\e)||_{L^2(\Omega_\e)}}\\
			&\leq \dfrac{1}{||\Pi_\e(\psi_\e)||_{L^2(\Omega_\e)}} \left(\left\|\Pi_\e(\psi_\e)- u \right\|_{L^2(\Omega_\e)} + \|u\|_{L^2(\Omega_\e)} \big|1- ||\Pi_\e(\psi_\e)||_{L^2(\Omega_\e)}\big|\right)\\
			&= O(\left\| V_{\e,u} \right\|_{L^2(\Omega_\e)})+O(||u||_{L^2( B(P,\e))}).
		\end{align*}
By \eqref{eq:negligible} we conclude the proof.
\end{proof}

\begin{cor}\label{c:1}
    As $\e\to0$, we have 
    \begin{equation}\label{s1}
			||u_\e-u+V_{\e,u}||_{L^2(\O_\e)}=O(||V_{\e,u}||_{L^2(\Omega_\e)}).
    \end{equation}
\end{cor}
\begin{proof}
    This is a straightforward consequence of Corollary \ref{l:s2}, obtained by  applying the Triangle inequality.
\end{proof}

Thanks to the Corollary \ref{c:1} and by means of regularity theory we are in position to prove a uniform estimate on the difference between $u_\e$ and $u - V_{\e,u}$. We are now in a position to prove Theorem \ref{intr-T1}.

\begin{proof}[Proof of Theorem \ref{intr-T1}]
We first observe that the function $u_\e-u+V_{\e,u}$ is a solution to
\begin{equation}\label{eq}
		\begin{cases}
				-\Delta(u_\e-u+V_{\e,u})-\la(u_\e-u+V_{\e,u})=(\la_\e-\la)u_\e-\la V_{\e,u}
				&\text{in }\Omega_\e,\\
				u_\e-u+V_{\e,u}=0 &\text{on }\partial \Omega_\e.
		\end{cases}
	\end{equation}
		To simplify the exposition, we set $\phi_\e=u_\e-u+V_{\e,u}$ and
		$f_\e=(\la_\e-\la)u_\e-\la V_{\e,u}$. So \eqref{eq} now reads
		\begin{equation}
			\begin{cases}
				-\Delta\phi_\e-\la\phi_\e=f_\e
				&\text{in }\Omega_\e\\
				\phi_\e=0 &\text{on }\partial \Omega_\e.
			\end{cases}
		\end{equation}
		By standard regularity theory, $u_\e$, $u$ and $V_{\e,u}$ are in $L^p(\Omega_\e)$ for any $p\in[1,+\infty]$. 
       In order to apply Lemma \ref{t1}, we choose $q$ in a proper way, according to the dimension. If $N\geq4$, we set $q=\frac N2 + \delta$, with $\delta\in(0,\frac12)$. In this way we have both $q>2$ and $q>\frac{N}{N-2}$. In dimension $N=3$ and $N=2$, we simply choose $\frac{N}{2}<q<2$.
        	 By Lemma \ref{t1}, there exists a constant $C=C(N,|\O|)$, which is independent of $\e$ such that 
\begin{align}\label{fi}
\|\phi_\e\|_{L^\infty(\Omega_\e)}&\le C\left(
\|\phi_\e\|_{L^2(\Omega_\e)}+\|f_\e\|_{L^q(\Omega_\e)}\right)\notag\\
&\le C\|V_{\e,u}\|_{L^2(\Omega_\e)}
+C|\la-\la_\e|\,\|u_\e\|_{L^q(\Omega_\e)}+C\|V_{\e,u}\|_{L^q(\Omega_\e)},
		\end{align}
        where the last inequality holds due to Corollary \ref{c:1}.
		Thanks to our choice of $q$, Lemma \ref{l:improvedL2Ve} and Lemma \ref{2l:improvedL2Ve}  apply, yielding \begin{equation}
        ||V_{\e,u}||_{L^q(\Omega_\e)}=
            \begin{cases}
                O\left(\frac{1}{|\log\e|}\right)\quad& if  \ N=2  \ and\ k=0\\
                O\big(\e^2\big) & if  \ N=2  \ and\ k=1\\
                O\big(\e^{k+\frac{N}{q}}\big) & if  \ N=2  \ and\ k\ge 2\\
                O\big(\e\big) & if  \ N=3  \ and\ k=0\\
                O\big(\e^{k+\frac{N}{q}}\big) & if  \ N=3  \ and\ k\ge 1\\
                 O\big(\e^{k+\frac{N}{q}}\big) & if  \ N\ge 4  \ and\ k\ge 0.\\
            \end{cases}
        \end{equation}
        As already mentioned, in dimension $N=3$ and $N=2$ we have $\frac{N}{2}<q<2$, so that $k+\frac{N}{q}>k+\frac{N}{2}$. On the other hand, for $N\ge 4$ we have $q=\frac{N}{2}+\delta$, so that 
		\[
		k+\frac{N}{\frac N2+\delta} >k+ 2(1-\delta).
		\]
		In conclusion, we have \begin{equation}
        ||V_{\e,u}||_{L^q(\Omega_\e)}=
            \begin{cases}
                O\left(\frac{1}{|\log\e|}\right)\quad& if  \ N=2  \ and\ k=0\\
                O\big(\e^2\big) & if  \ N=2  \ and\ k=1\\
                O\big(\e^{k+\frac{N}{2}}\big) & if  \ N=2  \ and\ k\ge 2\\
                O\big(\e\big) & if  \ N=3  \ and\ k=0\\
                O\big(\e^{k+\frac{N}{2}}\big) & if  \ N=3  \ and\ k\ge 1\\
                 O\big(\e^{k+ 2(1-\delta)}\big) & if  \ N\ge 4  \ and\ k\ge 0.\\
            \end{cases}
        \end{equation}
		We now recall that
		\begin{itemize}
         	\item $|\la-\la_\e|= \begin{cases}
            O\left(\frac{1}{|\log\e|}\right),& if  \ N=2  \ and\ k=0,\\
                O\big( \e^{N-2+2k}\big), & otherwise.
            \end{cases}$
			\item $\|u_\e\|_{L^q(\Omega_\e)} \le \|u_\e\|_{L^\infty(\Omega_\e)}^{\frac{q-1}q} \,\|u_\e\|_{L^1(\Omega_\e)}^\frac1q \le C(\Omega)\, 
			\|u_\e\|_{L^\infty(\Omega_\e)}^\frac{q-1}q\,\|u_\e\|_{L^2(\Omega_\e)}^\frac1q\le C(\Omega)$ where in the first inequality we used H\"older inequality, in the second one the well-known embedding of $L^p$-spaces and in the last one we used Lemma \ref{t1} with $\phi_\e=u_\e$, $\la_0=\la_\e$ and $f_\e=0$.
		\end{itemize}
		Finally \eqref{fi} becomes 
		\begin{equation}\label{U1}
			||\phi_\e||_\infty= \begin{cases}
                O\left(\frac{1}{|\log\e|}\right)\quad& if  \ N=2  \ and\ k=0\\
                O\big(\e^2|\log\e|\big) & if  \ N=2  \ and\ k=1\\
                O\big(\e^{k+\frac{N}{2}}\big) & if  \ N=2  \ and\ k\ge 2\\
                O\big(\e\big) & if  \ N=3  \ and\ k=0\\
                O\big(\e^{k+\frac{N}{2}}\big) & if  \ N=3  \ and\ k\ge 1\\
                 O\big(\e^{k+ 2(1-\delta)}\big) & if  \ N\ge 4  \ and\ k\ge 0,\\
            \end{cases}
		\end{equation}
		which ends the proof.
	\end{proof}
\vskip0.1cm
	\begin{proof}[Proof of Corollary \ref{T2'}]
		 By Lemma \ref{l:improvedL2Ve} and Lemma \ref{2l:improvedL2Ve}, it follows that on compact set $K\subset \Omega$,
        \begin{equation}\label{eq13102025-1}
            ||V_{\e,u}||_{L^\infty(K)}= \begin{cases}
          \displaystyle  O\left(\frac{1}{|\log\e|}\right),\quad& if  \ N=2  \ and\ k=0,\\
          \displaystyle O\left(\e^{2k}\right)+O\left(\frac{\e^{2+k}}{|\log\e|}\right),&if\ N=2\ and\ k\ge 1,\\
           \displaystyle     O\big( \e^{N-2+2k}\big)+O\left(\e^{N+k}\right), & otherwise.
            \end{cases}
        \end{equation}
        Combining \eqref{eq13102025-1} and \eqref{U1},  we get

			\begin{equation*}
				||u_\e-u||_{L^\infty(K)}\le||\phi_\e||_{L^\infty(K)}+||V_{\e,u}||_{L^\infty(K)}=
				\begin{cases}
                O\left(\frac{1}{|\log\e|}\right)\quad& if  \ N=2  \ and\ k=0\\
                O\big(\e^2|\log\e|\big) & if  \ N=2  \ and\ k=1\\
                O\big(\e^{k+\frac{N}{2}}\big) & if  \ N=2  \ and\ k\ge 2\\
                O\big(\e\big) & if  \ N=3  \ and\ k=0\\
                O\big(\e^{k+\frac{N}{2}}\big) & if  \ N=3  \ and\ k\ge 1\\
                 O\big(\e^{k+ 2(1-\delta)}\big) & if  \ N\ge 4  \ and\ k\ge 0,\\
            \end{cases}
			\end{equation*}
		which gives the claim.
	\end{proof}
\begin{proof}[Proof of Corollary \ref{T2}]
		It is a straightforward consequence of \eqref{V1}, \eqref{eq10102025-1}, \eqref{eq10102025-2} and \eqref{U1}. Indeed, we have
			\begin{align*}
				u_\e(x)&=u(x)-V_{\e,u}(x)+ 	\begin{cases}
                O\left(\frac{1}{|\log\e|}\right)\quad& if  \ N=2  \ and\ k=0\\
                O\big(\e^2|\log\e|\big) & if  \ N=2  \ and\ k=1\\
                O\big(\e^{k+1}\big) & otherwise\\
            \end{cases}\\
            &\left(\hbox{using \eqref{V1}, \eqref{eq10102025-1} and \eqref{eq10102025-2}}\right)=\\
            &u(x)-\begin{cases}
              \displaystyle A(N,0)\frac{1}{|\log \e|}u(p)G_\Omega(x,P)+ O\left(\frac{1}{|\log\e|}\right)\quad& if  \ N=2  \ and\ k=0\\
            \displaystyle A(N,k)\e^{N-2+2k}\sum_{|\alpha|=k}\frac{\partial^ku(P)}{\partial x_1^{\alpha_1}\cdots\partial x_N^{\alpha_k}}  \frac{\partial^kG(x,P)}{\partial y_1^{\alpha_1}\cdots\partial y_N^{\alpha_N}} +O\big(\e^2|\log\e|\big) & if  \ N=2  \ and\ k=1\\
             \displaystyle   A(N,k)\e^{N-2+2k}\sum_{|\alpha|=k}\frac{\partial^ku(P)}{\partial x_1^{\alpha_1}\cdots\partial x_N^{\alpha_k}}  \frac{\partial^kG(x,P)}{\partial y_1^{\alpha_1}\cdots\partial y_N^{\alpha_N}} + O\big(\e^{k+1}\big) & otherwise\\
            \end{cases}\\
            &\left(\hbox{using \eqref{eq10102025-4} and Lemma \ref{Cor-Q}}\right)=\\
             &u(x)-\begin{cases}
              \displaystyle -\frac{\log|x-P|}{|\log \e|}u(p)+ O\left(\frac{1}{|\log\e|}\right)\quad  if  \ N=2  \ and\ k=0\\
            \displaystyle \frac{\e^{2k}}{|x-P|^{2k}}\sum_{|\alpha|=k}\frac{\partial^ku(P)}{\partial x_1^{\alpha_1}\partial x_2^{\alpha_2}}  (x_1-P_1)^{\alpha_1}(x_2-P_2)^{\alpha_2} +O\big(\e^2|\log\e|\big) \quad  if  \ N=2  \ and\ k=1\\
             \displaystyle   \frac{\e^{N-2+2k}}{|x-P|^{N-2+2k}}\sum_{|\alpha|=k}\frac{\partial^ku(P)}{\partial x_1^{\alpha_1}\cdots\partial x_N^{\alpha_k}}  (x_1-P_1)^{\alpha_1}\cdots(x_N-P_N)^{\alpha_N} + O\big(\e^{k+1}\big) \quad  otherwise\\
            \end{cases}\\
            &\left(\hbox{using the Taylor expansion for $u$ }\right)=\\
             & \begin{cases}
              \displaystyle\left  (1+\frac{\log |x-P|}{|\log\e|}\right)u(x)
              +O\left(\frac{1}{|\log\e|}\right), &if \ N=2  \ and\ k=0,\\
				 \displaystyle \left(1-\frac{\e^{2k}}{|x-P|^{2k}}\right)u(x)+O(\e^{2}|\log\e|),& if  \ N=2  \ and\ k=1,\\
                  \displaystyle  \left(1-\frac{\e^{N-2+2k}}{|x-P|^{N-2+2k}}
				\right)u(x)+O(\e^{k+1}), &otherwise,
                \end{cases} \\
			\end{align*}
			which gives the claim.
	\end{proof}
	\begin{rem}
	    Observe that the estimate \eqref{d2} holds everywhere in $\O_\e$ but it is meaningful only when $|x-P|\in [\e,C\e]$ for some $C>1$.
	\end{rem}

\subsection{The case of multiple eigenvalues}

	 In Theorem \ref{intr-T1}, we establish the $L^\infty-$estimate of $u_\e-u+V_{\e,u}$ in the case where $\lambda$ is simple. The key  step is to estimate $\Pi_\e (u-V_{\e,u})$, where  $\Pi_\e (u-V_{\e,u})$ is the best approximation of $u$ in the space $E(\lambda_\e)$ (see Proposition \ref{p:main}). However, when $\lambda$ is a multiple eigenvalue, even with the same  orthogonal projection 
\[
	\Pi_{j,\e}: \ L^2(\Omega_\e) \to E(\lambda_{j,\e}),
	\] 
    the argument used in the proof of Proposition \ref{p:main} does not lead to the same conclusion. 
   Indeed, with this definition, we still have
    $\sigma(-\Delta_\e\big|_{\mathrm{ker}\Pi_{j,\e}})=\sigma(-\Delta_\e)\setminus
    \{\lambda_{j,\e}\}$. However, since $\lambda$ is  a multiple eigenvalue, the eigenvalues $\lambda_{j,\e}$  bifurcate   (see Theorem \ref{sim}), and hence $\hbox{dist}\big(\sigma(-\Delta_\e\big|_{\mathrm{ker}\Pi_{j,\e}}),\lambda_{j,\e}\big)=o(1)$. At this stage, the Spectral Theorem no longer yields an appropriate estimate of $\|u_j-V_{\e,u_j}-\Pi_{j,\e}(u_j-V_{\e,u_j})\|_{L^2(\Omega_\e)}$. 

    On the other hand, let $E_\e=E(\lambda_{1,\e})\oplus\cdots\oplus E(\lambda_{m(\lambda),\e})$ and let $\Pi_\e: L^2(\O_\e)\to E_\e$. Then  $||u_j-V_{\e,u_j}-\Pi_\e(u_j-V_{\e,u_j})||_{L^2(\O_\e)}\leq C||V_{\e,u_j}||_{L^2{\O_\e)}}$ (see  \cite[Lemma 4]{alm2026}). Thus, it  suffices to prove that $\hat{u}_{j,\e}:=\frac{\Pi_\e(u_j-V_{\e,u_j})}{||\Pi_\e(u_j-V_{\e,u_j})||_{L^2(\O_\e)}}$, $j=1,\cdots,m(\lambda)$,  forms a basis of $E_\e$. If $m(\lambda)=1$, the $\hat{u}_{1,\e}$ must be the eigenfunction of $\lambda_{1,\e}$. However, if $m(\lambda)\ge 2$, it is not clear whether $\hat{u}_{j,\e}$ is an eigenfunction associated with $\lambda_{j,\e}$. Moreover, it is  not even clear whether $\hat{u}_{j,\e}$ are pairwise orthogonal. Therefore,  when $\lambda$ is a multiple eigenvalue, a different argument is required.
    Recall the decomposition introduced in \eqref{2con},
$$u_{j,\e}=\sum_{i=1}^{m(\la)}\alpha_{ij,\e} (u_i-V_{\e,u_i})+ v_{j,\e}.$$
Combining  Lemma \ref{lem1511} and  Lemma \ref{lem2204}, we obtain a precise  estimate of $\|v_{j,\e}\|_{L^2(\O_\e)}$. Consequently, $\sum_{i=1}^{m(\la)}\alpha_{ij,\e} (u_i-V_{\e,u_i})$ provides a good approximation of   $u_{j,\e}$. Thus, one may expect an analogue of Theorem \ref{intr-T1} to hold when $\lambda$ is a multiple eigenvalue.
 
Next Theorem establishes  a uniform estimate for the difference between $u_{j,\e}$ and a suitable linear combination of $u_i-V_{\e,u_i}$, $i=1,\cdots,m(\lambda)$.  The proof  remains valid when $\lambda$ is simple and in this case, the conclusions coincide with those of  Theorem \ref{intr-T1}. 
\begin{teo}\label{th-2204}
    Under the same assumption as Theorem \ref{sim}, for any $J=0,1,\cdots,p$, if $j=1+d_0+..+d_{J-1},\cdots,d_0+..+d_J$, then  for any $\delta\in (0,\frac{1}{2})$, we have 
     \begin{equation}\label{M2}
     u_{j,\e}-\sum_{i=1}^{m(\lambda)}\alpha_{ij,\e}(u_i-V_{\e,u_i})=\begin{cases}
                O\left(\frac{1}{|\log\e|}\right)\quad& if  \ N=2  \ and\ k_J=0\\
                O\big(\e^2|\log\e|\big) & if  \ N=2  \ and\ k_J=1\\
                O\big(\e^{k_J+\frac{N}{2}}\big) & if  \ N=2  \ and\ k_J\ge 2\\
                O\big(\e\big) & if  \ N=3  \ and\ k_J=0\\
                O\big(\e^{k_J+\frac{N}{2}}\big) & if  \ N=3  \ and\ k_J\ge 1\\
                 O\big(\e^{k_J+ 2(1-\delta)}\big) & if  \ N\ge 4  \ and\ k_J\ge 0\\
			\end{cases}
\end{equation}
as $\e\to 0$, uniformly for $x\in\O_\e$,
where $\alpha_{ij,\e}$ is the same as \eqref{eq15112025-1}.
\end{teo}
\begin{proof}
    By \eqref{eq15112025-1}, we see that 
    $
    u_{j,\e}=\sum_{i=1}^{m(\lambda)}\alpha_{ij,\e}(u_i-V_{\e,u_i})+v_{j,\e},
    $
    and $v_{j,\e}$ satisfies
    \begin{equation}
    \begin{cases}
       \displaystyle -\Delta v_{j,\e}-\lambda_{j,\e}v_{j,\e}=f_{j,\e}\ \ &in \ \ \O_\e,\vspace{2mm}\\
       \displaystyle v_{j,\e}=0\ \ &on \ \ \partial\O_\e,\vspace{2mm}
    \end{cases}
\end{equation}
where $f_{j,\e}=(\lambda_{j,\e}-\lambda)\sum_{i=1}^{m(\la)}\alpha_{ij,\e}  u_i-\lambda_{j,\e}\sum_{i=1}^{m(\la)}\alpha_{ij,\e}  V_{\e,u_i} .$
Using Theorem \ref{t1}, it follows that 
$
||v_{j,\e}||_{L^\infty(\O_\e)}\leq C(||v_{j,\e}||_{L^2(\O_\e)}+||f_{j,\e}||_{L^q(\O_\e)}),
$
for any $q>\frac{N}{2}$. We choose $q$  the same way as in Theorem \ref{t1}. Combining Lemma \ref{l:improvedL2Ve}, Lemma \ref{2l:improvedL2Ve}, Theorem \ref{sim},  Lemma \ref{lem1511} and Lemma \ref{lem2204}, we obtain
\begin{align*}
   ||v_{j,\e}||_{L^\infty(\O_\e)}\leq &C|\lambda_{j,\e}-\lambda|+C\sum_{i=1}^{m(\la)}|\alpha_{ij,\e}|  ||V_{\e,u_i}||_{L^2(\O_\e)}+C||f_{j,\e}||_{L^q(\O_\e)} \\
   \leq &C|\lambda_{j,\e}-\lambda|+C\sum_{i=1}^{m(\la)}|\alpha_{ij,\e}|  ||V_{\e,u_i}||_{L^2(\O_\e)}+C\sum_{i=1}^{m(\la)}|\alpha_{ij,\e}|  ||V_{\e,u_i}||_{L^q(\O_\e)} \\
   =& 
   \begin{cases}
                O\left(\frac{1}{|\log\e|}\right)\quad& if  \ N=2  \ and\ k_J=0\\
                O\big(\e^2|\log\e|\big) & if  \ N=2  \ and\ k_J=1\\
                O\big(\e^{k_J+\frac{N}{2}}\big) & if  \ N=2  \ and\ k_J\ge 2\\
                O\big(\e\big) & if  \ N=3  \ and\ k_J=0\\
                O\big(\e^{k_J+\frac{N}{2}}\big) & if  \ N=3  \ and\ k_J\ge 1\\
                 O\big(\e^{k_J+ 2(1-\delta)}\big) & if  \ N\ge 4  \ and\ k_J\ge 0.\\
			\end{cases}
\end{align*}
This completes the proof.
\end{proof}
\begin{rem}\label{rem2204}
Using Lemma \ref{lem2204}, it follows that $\alpha_{lj,\e}\to 0$ as $\e\to 0$, if $j=1+d_0+\cdots+d_{J-1},\cdots,d_0+\cdots+d_J$ and  $l=1,\cdots,d_0+\cdots+d_{J-1}$.
    In the proof of Theorem \ref{sim}, we see that $\alpha_{lj,\e}\to 0$, if $j=1+d_0+\cdots+d_{J-1},\cdots,d_0+\cdots+d_J$ and $l=1+d_0+\cdots+d_J,\cdots,m(\lambda)$. Therefore,  the following estimate holds.
    \begin{equation}
        u_{j,\e}=\sum_{i=1+d_1+\cdots+d_{J-1}}^{d_1+\cdots+d_J}\alpha_{ij,\e}(u_i-V_{\e,u_i})+o(1), \hbox{ as } \e\to 0,
    \end{equation}
    uniformly for $x\in\O_\e$.
\end{rem}
Using Lemma \ref{lem2204} again, we obtain a boundary estimate around $\partial B(P,\e)$. 
\begin{cor}\label{cor2204}
     Under the same assumption as Theorem \ref{sim}, we have 
    $$
u_{j,\e}(x)=\begin{cases}
              \displaystyle\left  (1+\frac{\log |x-P|}{|\log\e|}\right)\sum_{i=1+d_0+\cdots+d_{J-1}}^{d_0+\cdots+d_J}\alpha_{ij,\e}u_i(x)
              +O\left(\frac{1}{|\log\e|}\right),\  if \ N=2  \ and\ k_J=0,\\
				 \displaystyle \left(1-\frac{\e^{2k_J}}{|x-P|^{2k_J}}\right)\sum_{i=1+d_0+\cdots+d_{J-1}}^{d_0+\cdots+d_J}\alpha_{ij,\e}u_i(x)+O(\e^{k_J+1}|\log\e|),\  if  \ N=2  \ and\ k_J\ge1,\\
                 \displaystyle   \left(1-\frac{\e^{N-2+2k_J}}{|x-P|^{N-2+2k_J}}
				\right)\sum_{i=1+d_0+\cdots+d_{J-1}}^{d_0+\cdots+d_J}\alpha_{ij,\e}u_i(x)+O(\e^{k_J+1}), \ if\ N\ge 3,
                \end{cases} 
$$
   as $\e\to 0$,  uniformly for $|x-P|\leq C\e$.
\end{cor}
\begin{proof}
For $|x-P|\leq C\e$ and for any $i=1+d_0+\cdots+d_{I-1},\cdots,d_0+\cdots+d_I$, $u_i-V_{\e,u_i}=O(\e^{k_I})$. In the proof of Theorem \ref{sim}, we see that $\alpha_{lj,\e}\to 0$, if $j=1+d_0+\cdots+d_{J-1},\cdots,d_0+\cdots+d_J$ and $l=1+d_0+\cdots+d_J,\cdots,m(\lambda)$. It follows that
\begin{equation}
    \sum_{I=J+1}^{p}\sum_{i=1+d_0+\cdots+d_{I-1}}^{d_0+\cdots+d_I}\alpha_{ij,\e} (u_i-V_{\e,u_i})=O(\e^{1+k_J}).
\end{equation}
Plugging \eqref{eq2301-3} into $ \sum_{I=1}^{J-1}\sum_{i=1+d_0+\cdots+d_{I-1}}^{d_0+\cdots+d_I}\alpha_{ij,\e} (u_i-V_{\e,u_i})$,  we have
    \begin{equation}
      \sum_{I=1}^{J-1}\sum_{i=1+d_0+\cdots+d_{I-1}}^{d_0+\cdots+d_I}\alpha_{ij,\e} (u_i-V_{\e,u_i})
      = \begin{cases}
\displaystyle O(\e^{1+k_J}|\log\e|) &\ for \ \ N=2\ and\  k_0=0,\\
    \displaystyle O(\e^{1+k_J})&\ for \ \ N=2\ and\  k_0\ge 1,\\
     \displaystyle O(\e^{1+k_J})&\ for \ \ N\ge3.\\
    \end{cases}
\end{equation}
Combining Theorem \ref{th-2204}, we obtain that 
 \begin{equation}\label{eq1707}
      u_{j,\e}-\sum_{i=1+d_0+\cdots+d_{J-1}}^{d_0+\cdots+d_J}\alpha_{ij,\e} (u_i-V_{\e,u_i})=
      \begin{cases}
      \displaystyle O\left(\frac{1}{|\log\e|}\right),&if\ N=2\ and \ k_J=0,\\
      \displaystyle O(\e^{k_J+1}|\log\e|), &if \ N=2\ and \ k_J\ge 1,\\
       \displaystyle O(\e^{k_J+1})&\ if \ \ N\ge3.\\
      \end{cases}
 \end{equation}
     It should be noted that, in the simple-eigenvalue case (see \eqref{d2}), if $N=2$, the remainder term is of order $O(\e^{k+1}|\log\e|)$ for $k=1$, whereas it is of order $O(\e^{k+1})$ for $k\ge2$.

The conclusion follows by arguing as in the proof of Corollary \ref{T2}. Indeed, since the vanishing order of  $u_i$ is   $k_J$ for any $i=1+d_0+\cdots+d_{J-1},\cdots,d_0+\cdots+d_{J}$, combining \eqref{V1}, \eqref{eq10102025-1}, \eqref{eq10102025-2} and \eqref{eq1707} with the Taylor expansion of $u_i$ we get the following results,
$$
u_{j,\e}(x)=\begin{cases}
              \displaystyle\left  (1+\frac{\log |x-P|}{|\log\e|}\right)\sum_{i=1+d_0+\cdots+d_{J-1}}^{d_0+\cdots+d_J}\alpha_{ij,\e}u_i(x)
              +O\left(\frac{1}{|\log\e|}\right),\  if \ N=2  \ and\ k_J=0,\\
				 \displaystyle \left(1-\frac{\e^{2k_J}}{|x-P|^{2k_J}}\right)\sum_{i=1+d_0+\cdots+d_{J-1}}^{d_0+\cdots+d_J}\alpha_{ij,\e}u_i(x)+O(\e^{k_J+1}|\log\e|),\  if  \ N=2  \ and\ k_J\ge1,\\
                 \displaystyle   \left(1-\frac{\e^{N-2+2k_J}}{|x-P|^{N-2+2k_J}}
				\right)\sum_{i=1+d_0+\cdots+d_{J-1}}^{d_0+\cdots+d_J}\alpha_{ij,\e}u_i(x)+O(\e^{k_J+1}), \ if\ N\ge 3,
                \end{cases} 
$$
uniformly for $|x-P|\leq C\e$.
\end{proof}

\section{Nodal lines of eigenfunctions $u_\e$}\label{sec-application}
	In this section we want to compare the nodal set of a perturbed eigenfunction $u_\e$, namely
	\begin{equation}
		\mathcal{N_\e}=\overline{\{x\in\O_\e\hbox{ such that }u_\e=0\}}
	\end{equation}
	with the nodal set of a limit eigenfunction $u$
	\begin{equation}
		\mathcal{N}=\overline{\{x\in\O\hbox{ such that }u=0\}}
	\end{equation}
      First we recall that by Theorem \ref{intr-T1}, we get
\begin{equation}\label{eq2901-1}
    u_\e(x)=u(x)-V_{\e,u}(x)+o(1), \ \hbox{as } \e\to0,
\end{equation}
 uniformly for $x\in \O_\e$. 

Next we write the boundary estimates of $u_\e$ in Corollary \ref{T2} as follows. For any fixed $C>0$ we have 
\begin{equation}\label{eq2204-02}
    u_\e(x)  =\begin{cases}
              \displaystyle  \left(1+\frac{\log |x-P|}{|\log\e|}\right)g(x)
              +O\left(\frac{1}{|\log\e|}\right), &if \ N=2  \ and\ k=0,\vspace{2mm}\\
			\displaystyle	  \left(1-\frac{\e^{2k}}{|x-P|^{2k}}\right)g(x)+O(\e^{2}|\log\e|),& if  \ N=2  \ and\ k=1,\vspace{2mm}\\
            \displaystyle        \left(1-\frac{\e^{N-2+2k}}{|x-P|^{N-2+2k}}
				\right)g(x)+O(\e^{k+1}), &otherwise,
                \end{cases}
\end{equation}
as $\e\to 0$, uniformly for $|x-P|\leq C\e$,
where
     \begin{equation}\label{forg}
        g(x)= \sum_{|\alpha|=k}\frac1{\alpha!}
					\frac{\partial^ku(P)}{\partial x_1^{\alpha_1}\cdots\partial x_N^{\alpha_N}}(x_1-P_1)^{\alpha_1}\cdots(x_N-P_N)^{\alpha_N}.
     \end{equation}
These are the basic ingredients that will be required in the subsequent proofs.

  For sake of simplicity, we only consider the case where $\lambda$ is a simple eigenvalue. The case of multiple eigenvalues can be treated in a similar way; however, in order not to make the paper excessively heavy, we briefly outline below how this situation can be handled.

 When $\lambda$ has multiplicity $m(\lambda)\ge2$, recalling the decomposition in \eqref{eq15112025-1}, by Remark \ref{rem2204}, for any $j=1+d_0+\cdots+d_{J-1},\cdots,d_0+\cdots+d_J$, we have 
  \begin{equation}\label{eq2204-03}
        u_{j,\e}=\sum_{i=1+d_0+\cdots+d_{J-1}}^{d_0+\cdots+d_J}\alpha_{ij,\e}(u_i-V_{u_i})+o(1), \hbox{ as } \e\to 0,
    \end{equation}
  uniformly for $x\in\O_\e$. By Corollary \ref{cor2204}, it follows that 
\begin{equation}\label{eq2204-04}
    u_{j,\e}(x)=\begin{cases}
              \displaystyle\left  (1+\frac{\log |x-P|}{|\log\e|}\right)\sum_{i=1+d_0+\cdots+d_{J-1}}^{d_0+\cdots+d_J}\alpha_{ij,\e}g_i(x)
              +O\left(\frac{1}{|\log\e|}\right),\  if \ N=2  \ and\ k_J=0,\\
				 \displaystyle \left(1-\frac{\e^{2k_J}}{|x-P|^{2k_J}}\right)\sum_{i=1+d_0+\cdots+d_{J-1}}^{d_0+\cdots+d_J}\alpha_{ij,\e}g_i(x)+O(\e^{k_J+1}|\log\e|),\  if  \ N=2  \ and\ k_J\ge1,\\
                 \displaystyle   \left(1-\frac{\e^{N-2+2k_J}}{|x-P|^{N-2+2k_J}}
				\right)\sum_{i=1+d_0+\cdots+d_{J-1}}^{d_0+\cdots+d_J}\alpha_{ij,\e}g_i(x)+O(\e^{k_J+1}), \ if\ N\ge 3,
                \end{cases} 
\end{equation}
as $\e\to 0$, uniformly for $|x-P|\leq C\e$, where for any $i=1+d_0+\cdots+d_{J-1},\cdots,d_0+\cdots+d_J$,
     $$
   g_i(x):= \sum_{|\alpha|=k_J}\frac1{\alpha!}
					\frac{\partial^{k_J}u_i(P)}{\partial x_1^{\alpha_1}\cdots\partial x_N^{\alpha_N}}(x_1-P_1)^{\alpha_1}\cdots(x_N-P_N)^{\alpha_N}.
    $$
Formulas \eqref{eq2204-03} and \eqref{eq2204-04} when $\lambda$ is multiple play the same role as \eqref{eq2901-1} and \eqref{eq2204-02} when $\lambda$ is simple. 

 \begin{rem}
    If $\lambda$ is multiple, the perturbed eigenvalues bifurcate by $\la$ with suitable rates. At this time, we study the nodal sets of eigenfunctions on each eigenbranch. Thus,
    we discuss the properties of the nodal set    $$\mathcal{N}_{j}=\overline{\{x\in\O\hbox{ such that }u_{j}=0\}}\hbox{ and } \mathcal{N}_{j,\e}=\overline{\{x\in\O_\e\hbox{ such that }u_{j,\e}=0\}}.$$
    The idea of proof is the same as that of the simple one. We only need to use \eqref{eq2204-03} to replace \eqref{eq2901-1} and use \eqref{eq2204-04} to replace \eqref{eq2204-02}. 
     After that, all the proofs follow without any changes.
\end{rem}
 Let us now return to the case of a simple eigenvalue.	We split the proof in the different cases $k=0$ and $k>0$.
	\vskip0.2cm
	\subsection{Case $k=0$}
	\vskip0.2cm
    In this case $u(P)\ne0$ and so $P\notin\mathcal{N}$. A first important consequence of the Corollary is that the nodal set $\mathcal{N_\e}$ does not hit $\partial B(P,\e)$.
	\begin{prop}\label{p0}
		 If $P\notin\mathcal{N}$, there exists $\delta>0$ such that  $\mathcal{N_\e}\cap B(P,\delta)=\emptyset$.
	\end{prop}
	\begin{proof}
		By contradiction assume that there exist points $r_\e\in\O_\e$ with dist$(r_\e,\partial B(P,\e))\to0$ such that $u_\e(r_\e)=0$. Next fix $\delta>0$ and observe that  by Corollary \ref{T2'} we have 
$u_\e(x)>\frac{u(P)}2>0$ on $\partial B(P,\delta)$. Moreover, by the strong maximum principle, it is not possible that $u_\e\ge0$ in $B(P,\delta)\setminus B(P,\e)$. So there exits $s_\e\in B(P,\delta)\setminus B(P,\e)$ such that $u_\e(s_\e)<0$ and this allows to build a domain 
 $\mathcal{D}_\e\subset \overline{B(P,\delta)\setminus B(P,\e)}$ such that for any $\e$ small enough we have that $u_\e(x)=0$ for any $x\in\partial \mathcal{D}_\e$.
        
		Now we apply Lemma \ref{sml} with $a(x)=\lambda$ and $\mathcal{D}=\mathcal{D}_\e$ so that $meas(\mathcal{D}_\e)<C_N\delta^N$. So choosing $\delta$ small such that $$\lambda C_N^\frac2N\delta^2<
		\begin{cases}
			S&if\ N>2\\
			1&if\ N=2,
		\end{cases}$$ 
		we derive by Lemma \ref{sml} that $u_\e\equiv0$ in $\mathcal{D}_\e$ and by the unique continuation we have that $u_\e\equiv0$ in $\Omega_\e$, a contradiction.
	\end{proof}
    \begin{teo}\label{th-0804}
    If $P\notin \mathcal{N}$, then 
          $$
    \sharp\{\hbox{nodal regions of } u_{\e}\}\leq \sharp \{\hbox{nodal regions of } u\}.
    $$
    \end{teo}
   \begin{proof}
       Since $P\notin \mathcal{N}$,  by Proposition \ref{p0}, there exists $\delta>0$ such that 
       $$
       \mathcal{N}_\e\cap B(P,\delta)=\mathcal{N}\cap B(P,\delta)=\emptyset.
       $$
       Then we only need to consider the nodal set $\mathcal{N}_\e \cap \big(\O\setminus B(P,\delta)\big)$.
      As a consequence of Corollary \ref{T2'}  we see that 
    \begin{equation}\label{eq0104-1}
        u_\e=u+\begin{cases}
            O\left(\frac{1}{|\log\e|}\right) &if \ N=2,\\
            O(\e)&if\ N\ge3,\\
        \end{cases}
    \end{equation}
    uniformly  for $x\in \O\setminus B(P,\delta)$.
    For any connected compact set $\mathcal{K}\subset \O\setminus \big(\mathcal{N}\cup B(P,\delta)\big)$  with $u\neq 0$ in $\mathcal{K}$, it follows from \eqref{eq0104-1} that $u_\e \neq 0$ in  the interior of $\mathcal{K}$. Then we see that the nodal set $\mathcal{N}_\e$ is  contained in a neighborhood of   $\mathcal{N}\cup \partial\O$. 
     In next claim we show that the nodal set $\mathcal{N}_\e$ is  contained just in the neighborhood of $\mathcal{N}$.
    
\textbf{Claim:} \emph{For any fixed $\tau>0$, there exists $\e_0$ such that for $\e<\e_0$, the nodal set $\mathcal{N}_\e$ is contained in the $\tau-$tubular neighborhood of $\mathcal{N}$.}
    
We prove the Claim by contradiction.
    Assume that there exist $\tau_0>0$ and $\e_i\to 0$ as $i\to\infty$,   such that $u_{\e_i}(x_{\e_i})=0$, $x_{\e_i}\in \O$, $dist(x_{\e_i},\mathcal{N})\ge \tau_0$ and $dist(x_{\e_i},\partial \O)\to0$. Then there exists $x_0\in \partial\O$ such that $x_{\e_i}\to x_0$ and $dist(x_0,\mathcal{N})\ge \tau_0$. Since $x_0\in\partial\O$ and $dist(x_0,\mathcal{N})\ge \tau_0$,  it follows from the Hopf lemma that $\frac{\partial u}{\partial\nu}(x_0)\neq 0$, where $\frac{\partial u}{\partial\nu}$ is the outnormal derivative of $u$ on the boundary $\partial\O$. 
    Since 
    \begin{equation}
        \begin{cases}
            -\Delta(u_\e-u)-\lambda_\e(u_\e-u)=(\lambda_\e-\lambda)u\ \hbox{in}\ \O_\e,\\
            u_\e-u=0\ \hbox{on}\ \partial\O,\\
            \end{cases}
    \end{equation}
    it follows from Lemma \ref{GTlem6.5} that for some $\delta>0$ there is a ball $B=B(x_0,\delta)$ such that 
    \begin{equation}\label{eq0904-1}
        u_{\e_i}\to u \ \hbox{in}\ C^2(B\cap\O)
    \end{equation}
    Without loss of generality, we assume that $x_0=0$ and $\nu=(0,\cdots,0,1)$.  Then $\frac{\partial u}{\partial x_i}(0)=0$, $i=1,\cdots,N-1$ and $\frac{\partial u}{\partial x_N}(0)\neq 0$.  By the approximation \eqref{eq0904-1}, we get  $\frac{\partial u_{\e_i}}{\partial x_i}(0)=o(1)$, $i=1,\cdots,N-1$ and $\frac{\partial u_{\e_i}}{\partial x_N}(0)= \frac{\partial u}{\partial x_N}(0)+o(1)\neq 0$.  Let $x'=(x_1,\cdots,x_{N-1})$. There exists $\delta>0$ such that if  $|x|\leq \delta$, then $u_{\e_i}(x)=o(1)x_1+\cdots+o(1)x_{N-1}+\left(\frac{\partial u}{\partial x_N}(0)+o(1)\right)x_N+O(|x'|^2+x_N^2)$. Then by the implicit function theorem,  given $x'$ such that $|x'|<\delta$, there exists a unique $x_{N,\e_i}(x')$ satisfying $|x_{N,\e_i}(x')|< \delta$ and  $u_{\e_i}(x',x_{N,\e_i})=0$. Since $u_{\e_i}=0$ on $\partial\O\cap B$ and by the uniqueness of $x_{N,\e_i}(x')$, it deduce that if $x\in \Bar{\O}\cap B$ and $u_{\e_i}(x)=0$, then $x\in \partial\O$ which contradicts to the assumption that there exists $x_{\e_i}\in\O$ and $u_{\e_i}(x_{\e_i})=0$. The proof of the Claim is completed.

     Next, we  prove that 
    $$
    \sharp\{\hbox{nodal regions of } u_{\e}\}\leq \sharp \{\hbox{nodal regions of } u\}.
    $$
    
      Assume that there exists $\e_i\to 0$ such that
      $$
    \sharp\{\hbox{nodal regions of } u_{\e_i}\}> \sharp \{\hbox{nodal regions of } u\}.
    $$
     By precompactness in Hausdorff metric (See Theorem 2.2.25 of \cite{Henrot2018}), there exists a subsequence called $\mathcal{N}_{\e_i}$  which converges to a set $X\subset\O\setminus B(P,\delta)$. It follows from  the Claim  that  for some small $\tau$, if $i$ large enough,  then $\mathcal{N}_{\e_i}$ is contained in the $\tau-$tubular neighborhood of $\mathcal{N}$. By the assumption that  
     $$
    \sharp\{\hbox{nodal regions of } u_{\e_i}\}> \sharp \{\hbox{nodal regions of } u\},
    $$
    it can be seen that there exists a nodal region $\mathcal{D}_{\e_i}$ of $u_{\e_i}$  which is contained in the $\tau-$tubular neighborhood of $\mathcal{N}$. Let $\tau\to 0$, we derive by Lemma \ref{sml} that $u_{\e_i}\equiv 0$ in $\mathcal{D}_{\e_i}$ and by the unique continuation we have that $u_{\e_i}\equiv0$ in $\Omega_{\e_i}$, a contradiction.
   \end{proof}
    An immediate consequence of the previous result is the following
   \begin{cor}
 If $\sharp \{\hbox{nodal regions of } u\}=2$, then $\sharp \{\hbox{nodal regions of } u_\e\}= 2$.
   \end{cor}
    \begin{proof}
In fact,  if $\sharp \{\hbox{nodal regions of } u\}=2$, then $u$ must change sign in $\O$.  Since $u(P)\neq 0$, it follows from \eqref{eq0104-1} that $u_\e$ must change sign in $\O\setminus B(P,\delta)$ where $\delta$ is defined in Proposition \ref{p0}. Then $\sharp \{\hbox{nodal regions of } u_\e\}\ge 2$. Combining Theorem \ref{th-0804}, we have $\sharp \{\hbox{nodal regions of } u_\e\}= 2$.
   \end{proof}
   The above statement states a special case that the equality holds. In the case of $N=2$, if $\sharp \{\hbox{nodal regions of } u\}=2$, we consider the following 2 configurations of $\mathcal{N}\subset\R^2$,
   \begin{itemize}
       \item [(I)] $\mathcal{N}$ is a simple curve  which intersects the boundary,
       \item[(II)] $\mathcal{N}$ is a Jordan curve and is compactly contained in $\O$.
   \end{itemize}
   By the Claim  established in Theorem \ref{th-0804}, we will deduce the same result as \cite[Proposition 2.6, $N=2$]{MukherjeeSaha2025}.
   More precisely, if $(\hbox{I})$ holds, then   the nodal set $\mathcal{N}_\e$ of $u_\e$ intersects $\partial \O$ at exactly 2 points. If $(\hbox{II})$ holds, then $\mathcal{N}_\e$ is also a Jordan curve  compactly contained in $\O_\e$..
   \begin{cor}\label{cor-intersect}
     Assume that $P\notin \mathcal{N}$ and $\sharp \{\hbox{nodal regions of } u\}=2$.  If the nodal set $\mathcal{N}$ of $u$ intersects $\partial \O$ at exactly 2 points, then the nodal set $\mathcal{N}_\e$ of $u_\e$ intersects $\partial \O$ at exactly 2 points. If the nodal set $\mathcal{N}$ of $u$ does not intersect $\partial \O$, then the nodal set $\mathcal{N}_\e$ of $u_\e$ does not intersect $\partial \O$. 
   \end{cor}
   \begin{proof}
      If  the nodal set $\mathcal{N}$ of $u$ intersects $\partial \O$ at exactly 2 points, by the Claim stated in Theorem \ref{th-0804},  the nodal set $\mathcal{N}_\e$ is contained in the $\tau-$tubular neighborhood of $\mathcal{N}$. Then  the nodal set $\mathcal{N}_\e$ of $u_\e$ intersects $\partial \O$ at exactly 2 points. 
      
       Observe that in this case $\mathcal{N}_\e$ cannot be a Jordan curve. Indeed, in this case there exists a nodal region $\mathcal{D}_\e$ is contained in   the $\tau-$tubular neighborhood of $\mathcal{N}$. Then using Lemma \ref{sml} and the unique continuation again, we obtain that $u_\e\equiv 0$ in $\O_\e$, a contradiction. 

       On the other hand if $\mathcal{N}$ is a Jordan curve and is compactly contained in $\O$, also by the Claim in Theorem \ref{th-0804}, $\mathcal{N}_\e$ is  contained in the $\tau-$tubular neighborhood of $\mathcal{N}$ which is contained in a compact set of $\O_\e$ for $\tau$ small. Then $\mathcal{N}_\e$ is also a Jordan curve and is compactly contained in $\O_\e$.

       This completes the proof.
   \end{proof}

	\subsection{ Case $k>0$, $N\ge3$}
	\vskip0.2cm
	First of all let us show that $\mathcal{N_\e}$  hits the boundary of $B(P,\e)$.
	\begin{prop}\label{zon}
		We have that $\mathcal{N_\e}\cap \partial B(P,\e)\ne\emptyset$.
	\end{prop}
	\begin{proof}
Let $R>1$ and $x\in B(P,R\e)$. Using Theorem \ref{T2} we get as $\e\to0$
\begin{align}	
u_\e(x)=&\left(1-\frac{\e^{N-2+2k}}{|x-P|^{N-2+2k}}		\right)u(x)+O(\e^{k+1})\\=
&\underbrace{\left(1-\frac1{R^{N-2+2k}}		\right)}_{>0}u(x)+O(\e^{k+1}).
\end{align}
Since $u(P)=0$, using Lemma \ref{lem1}, we have that $u$ behaves like an homogeneous harmonic polynomial of degree $k$ and by $k>0$ we get that $u$ does change sign for $\e$ small and any $R>1$. So also $u_\e$ does it and this proves the claim.
	\end{proof}

   \subsection{ Case $k>0$, $N=2$}
     Since $ u(x)=g(x)+O(|x-P|^{k+1})$ and $g$ is harmonic in $\R^2$,  the shape of the nodal line of $u$ in a neighborhood of $P$ is  well understood. 
    \begin{prop}\label{prop2612}\cite[Theorem 2.5 (ii)]{cheng1976}
        There exists $\delta_0>0$ such that the nodal line $\mathcal{N}\cap B(P,\delta_0)$ is a union of $k$ curves that only intersect at point $P$ and the angle at which two adjacent curves intersect at point $P$ is $\frac{\pi}{k}$. (See Picture below for $k=2$.) 
        \end{prop} 
  \centerline{ \begin{tikzpicture}[scale=0.6, transform shape, font=\large]

    \draw[thick] plot [smooth cycle, tension=0.7] coordinates {
        (-3.5, 0.5)
        (-2, 3)
        (2, 3.2)
        (3.5, 1)
        (2.5, -2.5)
        (-1.5, -3)
    };

    \node at (2.5, 2) {$\Omega$};

    \node at (0.2,0.4) {$P$};

    \draw[thick] (0,0) circle (1.5cm);

    \node at (0.2, 1.9) {$B(P, \delta_0)$};

    \filldraw (0,0) circle (2pt);

    \draw[thick] (210:1.5) -- (30:1.5);
    \draw[thick] (135:1.5) -- (-45:1.5);

\end{tikzpicture}}
By  Lemma 1.2 in \cite{CSLin1987}, the proof of Theorem \ref{apmain} is reduced to computing $\frac{\partial u_\e}{\partial\nu}=0$. Therefore, we give a delicate estimate of $\frac{\partial u_\e}{\partial\nu}$ as follows. 
  \begin{lemma}\label{lem2901-1}
  For some $\delta>0$, we obtain that at each $x_0\in \partial B(P,\e)$,
    \begin{equation}\label{eq0304-2}
       \e|D(u_\e-h_\e)|_{0;B(x_0,\e\delta)\cap \O_\e }+\e^2|D^2(u_\e-h_\e)|_{0;B(x_0,\e\delta)\cap \O_\e }=O(\e^{k}|\log\e|),
    \end{equation}  
    where 
    \begin{equation}
      h_\e(x)=  \left(1-\frac{\e^{2k}}{|x-P|^{2k}}\right)g(x).
    \end{equation}
    Furthermore,
         \begin{equation}\label{eq0304-3}
        \frac{\partial u_\e(x)}{\partial\nu}=\frac{2k}{\e}g(x)+O(\e^{k}|\log\e|)\ \hbox{on } \partial B(P,\e),
    \end{equation}
     (see \eqref{forg} for the definition of $g$),
    where 
    $\frac{\partial u_\e}{\partial\nu}$ is the outnormal derivative of $u_\e$ on the boundary $\partial B(P,\e)$.
    \end{lemma}
    \begin{proof}
    By direct computation we have that 
    \begin{equation}
      h_\e(x)=  \left(1-\frac{\e^{2k}}{|x-P|^{2k}}\right)g(x),
    \end{equation}
    verifies 
    \begin{equation}\label{eq29012026-3}
    \begin{cases}
         \Delta h_\e=0 &\hbox{in }\ \O_\e,\\
         h_\e=0&\hbox{on }\  \partial B(P,\e).\\
         
    \end{cases}
    \end{equation}
    Indeed,
    \begin{equation}
         \Delta h_\e(x)=\left(1-\frac{\e^{2k}}{|x-P|^{2k}}\right)\Delta g+g\Delta \left(1-\frac{\e^{2k}}{|x-P|^{2k}}\right)+2\nabla g\cdot \nabla \left(1-\frac{\e^{2k}}{|x-P|^{2k}}\right).
    \end{equation}
    By Lemma \ref{lem2212-1}, it follows that
    $\Delta g=0$ in $\R^N$  and then
    \begin{equation}
         \Delta \frac{1}{|x-P|^{2k}}=\frac{4k^2}{|x-P|^{2+2k}}
    \end{equation}
    and
    \begin{equation}
        \nabla  \frac{1}{|x-P|^{2k}}\cdot\nabla g(x)=\frac{-2k^2}{|x-P|^{2+2k}}g(x).
    \end{equation}
 So \eqref{eq29012026-3} follows.
    
    Next, we consider 
    \begin{equation}\label{eq29012026-4}
        \begin{cases}
               -\Delta (u_\e-h_\e)-\lambda_\e (u_\e-h_\e)=\lambda_\e h_\e &\hbox{in }\ \O_\e,\\
        u_\e- h_\e=0&\hbox{on }\  \partial B(P,\e).\\
        \end{cases}
    \end{equation}
    Applying  Lemma \ref{FurGT} to equation \eqref{eq29012026-4}, we obtain that there exists $\delta>0$ such that for each $x_0\in \partial B(P,\e)$, the following estimate holds
    \begin{equation}
        \e|D(u_\e-h_\e)|_{0;B(x_0,\e\delta)\cap \O_\e }+\e^2|D^2(u_\e-h_\e)|_{0;B(x_0,\e\delta)\cap \O_\e }\leq C\left(|u_\e-h_\e|_{0;\mathcal{O}_\e}+\e^2|h_\e|_{0;\mathcal{O}_\e}+\e^{2+\alpha}[h_\e]_{\alpha;\mathcal{O}_\e}\right),
    \end{equation}
    where $\mathcal{O}_\e:=B(P,c\e)\setminus B(P,\e)$ for some $c>1$  defined in Lemma \ref{FurGT}. Using \eqref{eq2204-02},  we see $|u_\e-h_\e|_{0;\mathcal{O}_\e}=O(\e^{1+k}|\log\e|)$,  and by direct computations, it follows that $|h_\e|_{1;\mathcal{O}_\e}=O(\e^{k-1})$.
    Then we  get \eqref{eq0304-2}.  

   By \eqref{eq0304-2}, it is  straightforward  that
    \begin{equation}
        \frac{\partial u_\e}{\partial\nu}=\frac{\partial h_\e}{\partial\nu}+O(\e^{k}|\log\e|)=\frac{2k}{\e}g(x)+O(\e^{k}|\log\e|).
    \end{equation}
   For any $x\in \partial B(P,\e)$,  let us compute
    \begin{equation}
        \begin{aligned}
            \frac{\partial h_\e(x)}{\partial\nu}=&\nabla\left[\left(1-\frac{\e^{2k}}{|x-P|^{2k}}\right)g(x)\right]\cdot \frac{x-P}{\e}\\
            =&g(x)\nabla\left(1-\frac{\e^{2k}}{|x-P|^{2k}}\right)\cdot \frac{x-P}{\e}+\left(1-\frac{\e^{2k}}{|x-P|^{2k}}\right)\nabla g(x)\cdot \frac{x-P}{\e}\\
            =&2k\frac{\e^{2k-1}}{|x-P|^{2k}}g(x)+\frac{k}{\e}\left(1-\frac{\e^{2k}}{|x-P|^{2k}}\right)g(x)\\
            =&\frac{k}{\e}\left(1+\frac{\e^{2k}}{|x-P|^{2k}}\right)g(x)\\
            =&\frac{2k}{\e}g(x).\\
        \end{aligned}
    \end{equation}
     This ends the proof.
     \end{proof}
      \begin{teo}\label{apmain}
        $\mathcal{N}_{\e}$  intersects $\partial B(P,\e)$ at $2k$ points.
    \end{teo}
    \begin{proof}
    Let us introduce (using polar coordinates and assuming $P=0$)
\[
v_\varepsilon(\rho,\theta)=u_\varepsilon(\varepsilon\rho,\theta),
\qquad \rho\ge1,\ \theta\in[0,2\pi].
\]
We have that
\[
v_\varepsilon(1,\theta)=u_\varepsilon(\varepsilon,\theta)=0
\qquad \forall \theta\in[0,2\pi],
\]
and
\[
\frac{\partial v_\varepsilon}{\partial \rho}(1,\theta)
=
\varepsilon \frac{\partial u_\varepsilon}{\partial \nu}(x),
\qquad x\in \partial B(0,\varepsilon).
\]
By \eqref{eq0304-3} we obtain, for $x\in \partial B(0,\varepsilon)$,
\[
\frac{\partial v_\varepsilon}{\partial \rho}(1,\theta)
=
\varepsilon \frac{\partial u_\varepsilon}{\partial \nu}(x)
=
2k\, g(x)+O(\varepsilon^{k+1} |\log \varepsilon|)
,
\]
for any $\theta\in[0,2\pi]$.
Since $g$ is a harmonic homogeneous polynomial of degree $k$ in $\R^2$, then for $x\in\partial B(0,\e)$, using polar coordinates, 
$$
g(x)=a\e^k\sin(k\theta+\theta_0).
$$
Thus, we have, for $x\in\partial B(0,\e)$,
$$
\frac{\partial v_\varepsilon}{\partial \rho}(1,\theta)
=2k a\, \varepsilon^{k}\sin(k\theta+\theta_0)
+O(\varepsilon^{k+1} |\log \varepsilon|),
$$
for any $\theta\in[0,2\pi]$.
Hence, if we introduce
\[
w_\varepsilon(\rho,\theta)=\frac{1}{\varepsilon^{k}} \frac{\partial v_\varepsilon}{\partial \rho}(\rho,\theta),
\]
 by \eqref{eq0304-2} in Lemma \ref{lem2901-1}, it can be seen that 
\[
w_\varepsilon(1,\theta)\to 2ka\,\sin(k\theta+\theta_0)
\quad  \hbox{in}\ C^1([0,2\pi]).
\]
Next denoting by $\theta_1,\dots,\theta_{2k}$ the solutions of
\[
\sin(k\theta+\theta_0)=0,
\]
we derive that 
$$ \frac{\partial w_\e}{\partial \theta}(1,\theta_i)\to 2k^2c\cos(k\theta_i+\theta_0)=\pm 2k^2a\quad \text{(according to the point $\theta_i$)}.
$$
By the implicit function theorem, this implies that there exists the unique $\theta_{i,\varepsilon}$ around $\theta_i$
such that
\[
w_\varepsilon(1,\theta_{i,\varepsilon})=0,
\qquad i=1,\dots,2k,
\]
and therefore there exist exactly 2k points $P_{i,\varepsilon}=(\e \cos \theta_{i,\varepsilon}, \e\sin \theta_{i,\varepsilon})$, $i=1,\ldots,2k$, such that 
\[
\frac{\partial u_\varepsilon}{\partial \nu}(P_{i,\varepsilon})=0.
\]
Hence, by  Lemma \ref{LCS}, we obtain that $\mathcal{N}_\varepsilon$ intersects
$\partial B(0,\varepsilon)$  exactly with $2k$ points.
    \end{proof}

	\section{Appendix}\label{sec-appendix}
    \subsection{Useful estimates for $Q_\alpha(x,y)$}
    In this section, we establish a  Lemma for $Q_\alpha(x,y)$, which will be useful in the proofs of  Lemma \ref{l:improvedL2Ve} and Lemma \ref{2l:improvedL2Ve}. 
    
 Recall that $\alpha=(\alpha_1,\cdots,\alpha_N) $ and $|\alpha|=\alpha_1+\cdots+\alpha_N=k$. For $N\ge 3$, let $C(N,0)=1$ and  $C(N,k)=(N-2)N\cdots(N-4+2k)$ for $k\ge 1$. 
 Define that
 $Q_\alpha(x,y)$  is a polynomial of degree $k$ depending on $\alpha $  such that the following equality holds
				\begin{equation*}
					\frac{\partial^k\left(|x-y|^{2-N}\right)}{\partial y_1^{\alpha_1}\cdots\partial y_N^{\alpha_N}}=C(N,k)\frac{(x_1-y_1)^{\alpha_1}\cdots(x_N-y_N)^{\alpha_N}+Q_\alpha(x,y)}
					{|x-y|^{N-2+2k}}.
				\end{equation*}
              For the convenience of representation, we define 
             $Q_{i_1,i_2,\cdots,i_k}(x,y)$ which  is a polynomial of degree $k$ depending on $i_1,\cdots,i_k$ such that the following equality holds
                	\begin{equation}\label{eq10102025-4}
					\frac{\partial^k\left(|x-y|^{2-N}\right)}{\partial y_{i_1} \partial y_{i_2}\cdots \partial y_{i_k}}=C(N,k)\frac{(x_{i_1}-y_{i_1})(x_{i_2}-y_{i_2})\cdots (x_{i_k}-y_{i_k})+Q_{i_1,i_2,\cdots,i_k}(x,y)}
					{|x-y|^{N-2+2k}}.
				\end{equation}
	 We see that 
        \begin{equation}\label{eq2904}
               \sum_{|\alpha|=k}\frac1{\alpha!}
					\frac{\partial^ku(P)}{\partial x_1^{\alpha_1}\cdots\partial x_N^{\alpha_N}}Q_\alpha(x,P) =\frac{1}{k!}\sum_{i_1,i_2,\cdots,i_k=1}^{N}\frac{\partial^k u(P)}{\partial x_{i_1} \partial x_{i_2}\cdots \partial x_{i_k}}Q_{i_1,i_2,\cdots,i_k}(x,P).
        \end{equation}
          For $N=2$, let   $C(N,0)=C(N,1)=1$ and $C(N,k)=2\times\cdots\times2(k-1)$ for any $k\ge 2$. For $k\ge1$, define  $Q_\alpha(x,y)$ as a polynomial of degree $k$ depending  on $\alpha_1$ and $\alpha_2$ such that the following equality holds,
                    \begin{equation}
                        \frac{\partial^k(-\log|x-y|)}{\partial y_1^{\alpha_1}\partial y_2^{\alpha_2}} = C(N,k)\frac{(x_1 - y_1)^{\alpha_1}(x_2 - y_2)^{\alpha_2}+Q_\alpha(x,y)}{|x - y|^{2k}},
                    \end{equation}
                   and  $Q_{i_1,i_2,\cdots,i_k}(x,y)$ is a polynomial of degree $k$ depending on $i_1,\cdots,i_k$ such that the following equality holds
                	\begin{equation}\label{eq2206-1}
					\frac{\partial^k(-\log|x-y|)}{\partial y_{i_1} \partial y_{i_2}\cdots \partial y_{i_k}}=C(N,k)\frac{(x_{i_1}-y_{i_1})(x_{i_2}-y_{i_2})\cdots (x_{i_k}-y_{i_k})+Q_{i_1,i_2,\cdots,i_k}(x,y)}
					{|x-y|^{2k}}.
				\end{equation}
              We see that   \eqref{eq2904} also holds for $N=2$ and $k\ge 1$.
              
      Then we only need to compute $\sum_{i_1,i_2,\cdots,i_k=1}^{N}\frac{\partial^k u(P)}{\partial x_{i_1} \partial x_{i_2}\cdots \partial x_{i_k}}Q_{i_1,i_2,\cdots,i_k}(x,P)$ for any $N\ge 2$. By direct observation the case of $k=2$ and $k=3$, we find that \eqref{eq1609-1} holds.
    Using mathematical induction, we prove the  Lemma \ref{lem10102025-1}.
       \begin{lemma}\label{lem10102025-1}
             Let $N\ge 2$.   We claim that, for any  $k\ge2$,
					\begin{equation}\label{eq1609-1}
						\sum_{i_1,i_2,\cdots,i_k=1}^{N}\frac{\partial^k u(x)}{\partial x_{i_1} \partial x_{i_2}\cdots \partial x_{i_k}}Q_{i_1,i_2,\cdots,i_k}(x,y)=	\sum_{i_1,i_2,\cdots,i_{k-2}=1}^{N}\frac{\partial^{k-2}(\Delta u) (x)}{\partial x_{i_1} \partial x_{i_2}\cdots \partial x_{i_{k-2}}} \widetilde{Q}_{i_1,i_2,\cdots,i_{k-2}}(x,y),
					\end{equation}
					where $\widetilde{Q}_{i_1,i_2,\cdots,i_{k-2}}(x,y)$ is a polynomial of degree k depending on $i_1,i_2,\cdots,i_{k-2}$.
            \end{lemma}
   \begin{proof}
   We prove this lemma by mathematical induction. We only prove the case that $N\ge3$. For $N=2$, the proof is the same.

By  a straightforward computation, it follows that
   \begin{itemize}
				\item For $k = 1$, $Q_{i_1}(x,y) = 0$.
				
				\item For $k = 2$,  $Q_{i_1,i_2}(x,y) = -\frac{1}{N}\delta_{i_1}^{i_2}|x - y|^2$ and $\widetilde{Q}(x,y)=-\frac{1}{N} |x - y|^2$.
				
				\item For $k = 3$,
					$
					Q_{i_1,i_2,i_3}(x,y) = -\frac{1}{N+2}|x - y|^2\left( \delta_{i_1}^{i_2}(x_{i_3} - y_{i_3}) + \delta_{i_1}^{i_3}(x_{i_2} - y_{i_2}) + \delta_{i_2}^{i_3}(x_{i_1} - y_{i_1}) \right),
					$
                    and 
                    $\widetilde{Q}_{i_1}=-\frac{3}{N+2}|x - y|^2(x_{i_1} - y_{i_1}) $.
                    \end{itemize}
                
                Now we suppose that \eqref{eq1609-1} holds for $k$, where $k\geq 2$ and $k\in \N$. We will prove that \eqref{eq1609-1} is valid for $k+1$.
				By assumption, we have
				\begin{equation*}
					\frac{\partial^k\left(|x-y|^{2-N}\right)}{\partial y_{i_1} \partial y_{i_2}\cdots \partial y_{i_k}}=C(N,k)\frac{(x_{i_1}-y_{i_1})(x_{i_2}-y_{i_2})\cdots (x_{i_k}-y_{i_k})+Q_{i_1,i_2,\cdots,i_k}(x,y)}
					{|x-y|^{N-2+2k}},
				\end{equation*}
				then 
				\begin{align*}
					&\frac{\partial^{k+1}\left(|x-y|^{2-N}\right)}{\partial y_{i_1} \partial y_{i_2}\cdots \partial y_{i_{k+1}}}=\frac{\partial}{\partial y_{i_{k+1}}}\frac{\partial^k\left(|x-y|^{2-N}\right)}{\partial y_{i_1} \partial y_{i_2}\cdots \partial y_{i_k}}\\
					=&-C(N,k)\frac{(x_{i_2}-y_{i_2})\cdots (x_{i_k}-y_{i_k})\delta_{i_{k+1}}^{i_1}+\cdots+(x_{i_1}-y_{i_1})\cdots (x_{i_{k-1}}-y_{i_{k-1}})\delta_{i_{k+1}}^{i_k}-\frac{\partial}{\partial y_{i_{k+1}}}Q_{i_1,i_2,\cdots,i_k}(x,y)}
					{|x-y|^{N-2+2k}}\\
					&+C(N,k)(N-2+2k)\frac{\Big[(x_{i_1}-y_{i_1})(x_{i_2}-y_{i_2})\cdots (x_{i_k}-y_{i_k})+Q_{i_1,i_2,\cdots,i_k}(x,y)\Big](x_{i_{k+1}}-y_{i_{k+1}})}
					{|x-y|^{N-2+2(k+1)}}.
				\end{align*}
				Hence, 
				\begin{equation}\label{eq1709-5}
					\begin{aligned}
						&Q_{i_1,i_2,\cdots,i_{k+1}}(x,y)\\
						=&-\frac{1}{N-2+2k}|x-y|^2\Big[(x_{i_2}-y_{i_2})\cdots (x_{i_k}-y_{i_k})\delta_{i_{k+1}}^{i_1}+\cdots+(x_{i_1}-y_{i_1})\cdots (x_{i_{k-1}}-y_{i_{k-1}})\delta_{i_{k+1}}^{i_k}\Big]\\
						&+\frac{1}{N-2+2k}|x-y|^2\frac{\partial}{\partial y_{i_{k+1}}}Q_{i_1,i_2,\cdots,i_k}(x,y)+Q_{i_1,i_2,\cdots,i_k}(x,y)(x_{i_{k+1}}-y_{i_{k+1}}).
					\end{aligned}
				\end{equation}
				Next we multiply each term by $\frac{\partial^{k+1} u(x)}{\partial x_{i_1} \cdots \partial x_{i_{k+1}}}$ and sum over indices $i_1,\cdots,i_{k+1}$.
				\begin{equation}\label{eq1709-4}
					\begin{aligned}
						&\sum_{i_1,\cdots,i_{k+1}=1}^{N}\frac{\partial^{k+1} u(x)}{\partial x_{i_1} \cdots \partial x_{i_{k+1}}}|x-y|^2\Big[(x_{i_2}-y_{i_2})\cdots (x_{i_k}-y_{i_k})\delta_{i_{k+1}}^{i_1}+\cdots+(x_{i_1}-y_{i_1})\cdots (x_{i_{k-1}}-y_{i_{k-1}})\delta_{i_{k+1}}^{i_k}\Big]\\
						&=k\sum_{i_1,\cdots,i_{k-1}=1}^{N}\frac{\partial^{k-1} (\Delta u)(x)}{\partial x_{i_1} \cdots \partial x_{i_{k-1}} }|x-y|^2(x_{i_1}-y_{i_1})\cdots(x_{i_{k-1}}-y_{i_{k-1}}).\\
					\end{aligned}
				\end{equation}
				By differentiating both sides of equation \eqref{eq1609-1} with respect to $x_{i_{k+1}}$, we obtain that 
				\begin{equation}\label{eq1609-2}
					LHS=	\sum_{i_1,i_2,\cdots,i_k=1}^{N}\frac{\partial^{k+1} u(x)}{\partial x_{i_1} \partial x_{i_2}\cdots \partial x_{i_{k+1}}}Q_{i_1,i_2,\cdots,i_k}(x,y)+\sum_{i_1,i_2,\cdots,i_k=1}^{N}\frac{\partial^{k} u(x)}{\partial x_{i_1} \partial x_{i_2}\cdots \partial x_{i_{k}}}\frac{\partial Q_{i_1,i_2,\cdots,i_k}(x,y)}{\partial x_{i_{k+1}}},
				\end{equation}
				\begin{equation}\label{eq1609-3}
					\begin{aligned}
						RHS&=	\sum_{i_1,i_2,\cdots,i_{k-2}=1}^{N}\frac{\partial^{k-1}(\Delta u)(x)}{\partial x_{i_1}\cdots \partial x_{i_{k-2}}\partial x_{i_{k+1}}}  \widetilde{Q}_{i_1,i_2,\cdots,i_{k-2}}(x,y)\\
						&+\sum_{i_1,i_2,\cdots,i_{k-2}=1}^{N}\frac{\partial^{k-2}(\Delta u)(x)}{\partial x_{i_1}\cdots \partial x_{i_{k-2}}} \frac{\partial\widetilde{Q}_{i_1,i_2,\cdots,i_{k-2}}(x,y)}{\partial x_{i_{k+1}}}.
					\end{aligned}
				\end{equation}
				Making the same with respect to $y_{i_{k+1}}$, we have
				\begin{equation}\label{eq1609-4}
					\sum_{i_1,i_2,\cdots,i_k=1}^{N}\frac{\partial^k u(x)}{\partial x_{i_1} \partial x_{i_2}\cdots \partial x_{i_k}}\frac{\partial Q_{i_1,i_2,\cdots,i_k}(x,y)}{\partial y_{i_{k+1}} }=	\sum_{i_1,i_2,\cdots,i_{k-2}=1}^{N}\frac{\partial^{k-2}}{\partial x_{i_1} \partial x_{i_2}\cdots \partial x_{i_{k-2}}} (\Delta u)(x) \frac{\partial\widetilde{Q}_{i_1,i_2,\cdots,i_{k-2}}(x,y)}{\partial y_{i_{k+1}} }.
				\end{equation}
			By the structure of $Q_{i_1,i_2,\cdots,i_k}(x,y)$ and $\widetilde{Q}_{i_1,i_2,\cdots,i_{k-2}}(x,y)$, we can see that
				\begin{equation}\label{eq1609-5}
					\sum_{i_1,i_2,\cdots,i_k=1}^{N}\frac{\partial^k u(x)}{\partial x_{i_1} \partial x_{i_2}\cdots \partial x_{i_k}}\frac{\partial Q_{i_1,i_2,\cdots,i_k}(x,y)}{\partial y_{i_{k+1}} }=	-\sum_{i_1,i_2,\cdots,i_k=1}^{N}\frac{\partial^k u(x)}{\partial x_{i_1} \partial x_{i_2}\cdots \partial x_{i_k}}\frac{\partial Q_{i_1,i_2,\cdots,i_k}(x,y)}{\partial x_{i_{k+1}} },
				\end{equation}
				and 
				\begin{equation}\label{eq1609-6}
					\sum_{i_1,\cdots,i_{k-2}=1}^{N}\frac{\partial^{k-2}}{\partial x_{i_1} \cdots \partial x_{i_{k-2}}} (\Delta u)(x) \frac{\partial\widetilde{Q}_{i_1,\cdots,i_{k-2}}(x,y)}{\partial y_{i_{k+1}} }=-\sum_{i_1,\cdots,i_{k-2}=1}^{N}\frac{\partial^{k-2}}{\partial x_{i_1} \cdots \partial x_{i_{k-2}}} (\Delta u)(x) \frac{\partial\widetilde{Q}_{i_1,\cdots,i_{k-2}}(x,y)}{\partial x_{i_{k+1}} }.
				\end{equation}
				By \eqref{eq1609-2}--\eqref{eq1609-6}, we obtain  
				\begin{equation}\label{eq1609-7}
					\sum_{i_1,i_2,\cdots,i_k=1}^{N}\frac{\partial^{k+1} u(x)}{\partial x_{i_1} \partial x_{i_2}\cdots \partial x_{i_{k+1}}}Q_{i_1,i_2,\cdots,i_k}(x,y)
					=\sum_{i_1,i_2,\cdots,i_{k-2}=1}^{N}\frac{\partial^{k-1}(\Delta u)(x)}{\partial x_{i_1}\cdots \partial x_{i_{k-2}}\partial x_{i_{k+1}}}  \widetilde{Q}_{i_1,i_2,\cdots,i_{k-2}}(x,y).
				\end{equation}
				Then 
				\begin{equation}\label{eq1709-3}
					\begin{aligned}
						&\sum_{i_1,i_2,\cdots,i_{k+1}=1}^{N}\frac{\partial^{k+1} u(x)}{\partial x_{i_1} \partial x_{i_2}\cdots \partial x_{i_{k+1}}}Q_{i_1,i_2,\cdots,i_k}(x,y)(x_{i_{k+1}}-y_{i_{k+1}})\\
						&=\sum_{i_1,i_2,\cdots,i_{k-1}=1}^{N}\frac{\partial^{k-1}(\Delta u)(x)}{\partial x_{i_1}\cdots \partial x_{i_{k-2}}\partial x_{i_{k-1}}}  \widetilde{Q}_{i_1,i_2,\cdots,i_{k-2}}(x,y)(x_{i_{k-1}}-y_{i_{k-1}}),\\
					\end{aligned}
				\end{equation}
				where the equality holds due to index rearrangement.
				By differentiating both sides of equation \eqref{eq1609-7} with respect to $y_{i_{k+1}}$, we have
				\begin{equation}\label{eq1709-1}
					\sum_{i_1,i_2,\cdots,i_k=1}^{N}\frac{\partial^{k+1} u(x)}{\partial x_{i_1} \partial x_{i_2}\cdots \partial x_{i_{k+1}}}\frac{\partial Q_{i_1,i_2,\cdots,i_k}(x,y)}{\partial y_{i_{k+1}}}
					=\sum_{i_1,i_2,\cdots,i_{k-2}=1}^{N}\frac{\partial^{k-1}(\Delta u)(x)}{\partial x_{i_1}\cdots \partial x_{i_{k-2}}\partial x_{i_{k+1}}}  \frac{\partial\widetilde{Q}_{i_1,i_2,\cdots,i_{k-2}}(x,y)}{\partial y_{i_{k+1}}}.
				\end{equation}
				Then by index rearrangement, we obtain 
				\begin{equation}\label{eq1709-2}
					\begin{aligned}
						&\sum_{i_1,i_2,\cdots,i_{k+1}=1}^{N}\frac{\partial^{k+1} u(x)}{\partial x_{i_1} \partial x_{i_2}\cdots \partial x_{i_{k+1}}}\frac{\partial Q_{i_1,i_2,\cdots,i_k}(x,y)}{\partial y_{i_{k+1}}}|x-y|^2\\
						&=\sum_{i_1,i_2,\cdots,i_{k-1}=1}^{N}\frac{\partial^{k-1}(\Delta u)(x)}{\partial x_{i_1}\cdots \partial x_{i_{k-2}}\partial x_{i_{k-1}}}  \frac{\partial\widetilde{Q}_{i_1,i_2,\cdots,i_{k-2}}(x,y)}{\partial y_{i_{k-1}}}|x-y|^2.\\
					\end{aligned}
				\end{equation}
				By \eqref{eq1709-5}, \eqref{eq1709-4}, \eqref{eq1709-3} and \eqref{eq1709-2}, we have 
				\begin{equation}
					\sum_{i_1,i_2,\cdots,i_{k+1}=1}^{N}\frac{\partial^{k+1} u(x)}{\partial x_{i_1} \partial x_{i_2}\cdots \partial x_{i_{k+1}}}Q_{i_1,i_2,\cdots,i_{k+1}}(x,y)=\sum_{i_1,i_2,\cdots,i_{k-1}=1}^{N}\frac{\partial^{k-1}(\Delta u)(x)}{\partial x_{i_1}\cdots \partial x_{i_{k-2}}\partial x_{i_{k-1}}} \widetilde{Q}_{i_1,i_2,\cdots,i_{k-1}}(x,y),
				\end{equation}
				where $$
				\begin{aligned}
					&\widetilde{Q}_{i_1,i_2,\cdots,i_{k-1}}(x,y)\\
                    =&-\frac{k}{N-2+2k}|x-y|^2(x_{i_1}-y_{i_1})\cdots(x_{i_{k-1}}-y_{i_{k-1}})\\
					&+\frac{1}{N-2+2k}\frac{\partial\widetilde{Q}_{i_1,i_2,\cdots,i_{k-2}}(x,y)}{\partial y_{i_{k-1}}}|x-y|^2+\widetilde{Q}_{i_1,i_2,\cdots,i_{k-2}}(x,y)(x_{i_{k-1}}-y_{i_{k-1}}).
				\end{aligned}
				$$
				Thus the proof is complete.
                 \end{proof}
                 \begin{cor}\label{Cor-Q}
                 Assume that $k\in\N$ is the vanishing order of $u$ at $P$.
                     If $u$ is an eigenfunction of Dirichlet Laplacian on $\O$ satisfying \eqref{eq:eigenvalueproblem}, then
                      \begin{equation}\label{eq2904-2}
					\sum_{|\alpha|=k}\frac1{\alpha!}
					\frac{\partial^ku(P)}{\partial x_1^{\alpha_1}\cdots\partial x_N^{\alpha_N}}Q_\alpha(x,P)=0\ \ \hbox{and}\ \ 	\sum_{|\beta|=k+1}\frac1{\beta!}
					\frac{\partial^{k+1}u(P)}{\partial x_1^{\beta_1}\cdots\partial x_N^{\beta_N}}Q_\beta(x,P)=0.
				\end{equation}   
                 \end{cor}  
                 \begin{proof}
                     By direct computation, we see that \eqref{eq2904-2} holds for $k=0$ and $k=1$. 
                     For $k\ge2$, combining \eqref{eq2904} and \eqref{eq1609-1}, we have
                     \begin{align*}
                         &\sum_{|\alpha|=k}\frac1{\alpha!}
					\frac{\partial^ku(P)}{\partial x_1^{\alpha_1}\cdots\partial x_N^{\alpha_N}}Q_\alpha(x,P)\\
                    =&\frac{1}{k!}\sum_{i_1,i_2,\cdots,i_{k-2}=1}^{N}\frac{\partial^{k-2}}{\partial x_{i_1} \partial x_{i_2}\cdots \partial x_{i_{k-2}}} (\Delta u)(P) \widetilde{Q}_{i_1,i_2,\cdots,i_{k-2}}(x,P)\\
                   =& -\frac{\lambda}{k!}\sum_{i_1,i_2,\cdots,i_{k-2}=1}^{N}\frac{\partial^{k-2}}{\partial x_{i_1} \partial x_{i_2}\cdots \partial x_{i_{k-2}}}  u(P) \widetilde{Q}_{i_1,i_2,\cdots,i_{k-2}}(x,P)\\
                     \end{align*}
                     and 
                      \begin{align*}
                         &\sum_{|\beta|=k+1}\frac1{\beta!}
					\frac{\partial^{k+1}u(P)}{\partial x_1^{\beta_1}\cdots\partial x_N^{\beta_N}}Q_\beta(x,P)\\
                   =& -\frac{\lambda}{(k+1)!}\sum_{i_1,i_2,\cdots,i_{k-1}=1}^{N}\frac{\partial^{k-1}}{\partial x_{i_1} \partial x_{i_2}\cdots \partial x_{i_{k-1}}}  u(P) \widetilde{Q}_{i_1,i_2,\cdots,i_{k-1}}(x,P).\\
                     \end{align*}
                     Since the vanishing order of $u$ at $P$ is $k$, \eqref{eq2904-2} holds.
                 \end{proof}
            
  \textbf{Acknowledgments:}
  
Massimo Grossi is supported by Indam-GNAMPA
  
 Ying Li is supported by the Program of China Scholarship Council (Grant: 202506770037).

	\bibliographystyle{abbrv}
	\bibliography{LauraMaxYing.bib}

@article {Uhlenbeck,
    AUTHOR = {Uhlenbeck, K.},
     TITLE = {Generic properties of eigenfunctions},
   JOURNAL = {Amer. J. Math.},
  FJOURNAL = {American Journal of Mathematics},
    VOLUME = {98},
      YEAR = {1976},
    NUMBER = {4},
     PAGES = {1059--1078},
      ISSN = {0002-9327,1080-6377},
   MRCLASS = {58G99 (58C25)},
  MRNUMBER = {464332},
MRREVIEWER = {A.\ J.\ Tromba},
       DOI = {10.2307/2374041},
       URL = {https://doi.org/10.2307/2374041},
}

@article {Micheletti,
    AUTHOR = {Micheletti, Anna Maria},
     TITLE = {Perturbazione dello spettro dell'operatore di {L}aplace, in
              relazione ad una variazione del campo},
   JOURNAL = {Ann. Scuola Norm. Sup. Pisa Cl. Sci. (3)},
  FJOURNAL = {Annali della Scuola Normale Superiore di Pisa. Classe di
              Scienze. Serie III},
    VOLUME = {26},
      YEAR = {1972},
     PAGES = {151--169},
      ISSN = {0391-173X},
   MRCLASS = {35P05},
  MRNUMBER = {367480},
MRREVIEWER = {W.\ M.\ Greenlee},
}

@article{alm2026,
  title={Corrigendum to" Ramification of multiple eigenvalues for the Dirichlet-Laplacian in perforated domains" published in Journal of Functional Analysis 283 (2022) 109718},
  author={Abatangelo, Laura and L{\'e}na, Corentin and Musolino, Paolo},
  journal={arXiv preprint arXiv:2606.03996},
  year={2026}
}

@article {MukherjeeSaha2025,
    AUTHOR = {Mukherjee, Mayukh and Saha, Soumyajit},
     TITLE = {On the effects of small perturbation on low energy {L}aplace
              eigenfunctions},
   JOURNAL = {J. Spectr. Theory},
  FJOURNAL = {Journal of Spectral Theory},
    VOLUME = {15},
      YEAR = {2025},
    NUMBER = {3},
     PAGES = {1045--1087},
      ISSN = {1664-039X,1664-0403},
   MRCLASS = {58J50 (35J05)},
  MRNUMBER = {4949372},
       DOI = {10.4171/jst/570},
       URL = {https://doi.org/10.4171/jst/570},
}

@article{cheng1976,
title = "Eigenfunctions and nodal sets",
author = "Cheng, \{Shiu Yuen\}",
year = "1976",
month = dec,
doi = "10.1007/BF02568142",
language = "English",
volume = "51",
pages = "43--55",
journal = "Commentarii Mathematici Helvetici",
issn = "0010-2571",
publisher = "European Mathematical Society Publishing House",
number = "1",
}

@misc{milne1945,
  title={A treatise on the theory of Bessel functions},
  author={Milne-Thomson, LM},
  year={1945},
  publisher={Nature Publishing Group UK London}
}

@article{BCM2021,
author = {Beck, Thomas and Canzani, Yaiza and Marzuola, Jeremy},
year = {2021},
month = {03},
pages = {323-353},
title = {Nodal line estimates for the second Dirichlet eigenfunction},
volume = {11},
journal = {Journal of Spectral Theory},
doi = {10.4171/JST/342}
}

@article{MHTONN1997,
author = {M. Hoffmann-Ostenhof and T. Hoffmann-Ostenhof and N. Nadirashvili},
title = {{The nodal line of the second eigenfunction of the Laplacian in $\mathbb{R}^2$ can be closed}},
volume = {90},
journal = {Duke Mathematical Journal},
number = {3},
publisher = {Duke University Press},
pages = {631 -- 640},
year = {1997},
doi = {10.1215/S0012-7094-97-09017-7},
URL = {https://doi.org/10.1215/S0012-7094-97-09017-7}
}

@article{melas1992,
  title={On the nodal line of the second eigenfunction of the Laplacian in R2},
  author={Melas, Antonios D},
  journal={J. Differential Geom},
  volume={35},
  number={1},
  pages={255--263},
  year={1992}
}

@article{PU1990,
    AUTHOR = {P\"utter, Rolf},
     TITLE = {On the nodal lines of second eigenfunctions of the fixed
              membrane problem},
   JOURNAL = {Comment. Math. Helv.},
  FJOURNAL = {Commentarii Mathematici Helvetici},
    VOLUME = {65},
      YEAR = {1990},
    NUMBER = {1},
     PAGES = {96--103},
      ISSN = {0010-2571,1420-8946},
   MRCLASS = {35P05 (35J05 58G25)},
  MRNUMBER = {1036131},
       DOI = {10.1007/BF02566596},
       URL = {https://doi.org/10.1007/BF02566596},
}

@article{payne1967,
author = {Payne, L. E.},
title = {Isoperimetric Inequalities and Their Applications},
journal = {SIAM Review},
volume = {9},
number = {3},
pages = {453-488},
year = {1967},
doi = {10.1137/1009070},

URL = { 
    
        https://doi.org/10.1137/1009070
    
    

},
eprint = { 
    
        https://doi.org/10.1137/1009070
    
    

}

}

@article{Courtois1995,
title = {Spectrum of Manifolds with Holes},
journal = {Journal of Functional Analysis},
volume = {134},
number = {1},
pages = {194-221},
year = {1995},
issn = {0022-1236},
doi = {https://doi.org/10.1006/jfan.1995.1142},
url = {https://www.sciencedirect.com/science/article/pii/S0022123685711421},
author = {G. Courtois}
}

@article {MMSS2022,
    AUTHOR = {Mukherjee, Mayukh and Saha, Soumyajit},
     TITLE = {Nodal sets of {L}aplace eigenfunctions under small
              perturbations},
   JOURNAL = {Math. Ann.},
  FJOURNAL = {Mathematische Annalen},
    VOLUME = {383},
      YEAR = {2022},
    NUMBER = {1-2},
     PAGES = {475--491},
      ISSN = {0025-5831,1432-1807},
   MRCLASS = {58J50},
  MRNUMBER = {4444128},
MRREVIEWER = {Maxime\ Ingremeau},
       DOI = {10.1007/s00208-021-02144-3},
       URL = {https://doi.org/10.1007/s00208-021-02144-3},
}

@book {helffer2013,
    AUTHOR = {Helffer, Bernard},
     TITLE = {Spectral theory and its applications},
    SERIES = {Cambridge Studies in Advanced Mathematics},
    VOLUME = {139},
 PUBLISHER = {Cambridge University Press, Cambridge},
      YEAR = {2013},
     PAGES = {vi+255},
      ISBN = {978-1-107-03230-9},
   MRCLASS = {47-02 (34L05 35Pxx 47A10 47A25 47B40 47J10 81Q05)},
  MRNUMBER = {3027462},
MRREVIEWER = {Pavel\ V.\ Exner},
}

@article{hormander1963linear,
  title={Linear partial differential operators Springer-Verlag},
  author={H{\"o}rmander, L},
  journal={New York},
  year={1963}
}

@article {DeCarliHudson,
    AUTHOR = {De Carli, L. and Hudson, S. M.},
     TITLE = {Geometric remarks on the level curves of harmonic functions},
   JOURNAL = {Bull. Lond. Math. Soc.},
  FJOURNAL = {Bulletin of the London Mathematical Society},
    VOLUME = {42},
      YEAR = {2010},
    NUMBER = {1},
     PAGES = {83--95},
      ISSN = {0024-6093,1469-2120},
   MRCLASS = {31B05 (31A05)},
  MRNUMBER = {2586969},
MRREVIEWER = {Steven\ M.\ Deckelman},
       DOI = {10.1112/blms/bdp099},
       URL = {https://doi.org/10.1112/blms/bdp099},
}

@article {GrossiMolle,
    AUTHOR = {Grossi, Massimo and Molle, Riccardo},
     TITLE = {On the shape of the solutions of some semilinear elliptic
              problems},
   JOURNAL = {Commun. Contemp. Math.},
  FJOURNAL = {Communications in Contemporary Mathematics},
    VOLUME = {5},
      YEAR = {2003},
    NUMBER = {1},
     PAGES = {85--99},
      ISSN = {0219-1997,1793-6683},
   MRCLASS = {35J60 (35B33 35J25)},
  MRNUMBER = {1958020},
MRREVIEWER = {Nichiro\ Kawano},
       DOI = {10.1142/S0219199703000914},
       URL = {https://doi.org/10.1142/S0219199703000914},
}

@article {CSLin1987,
    AUTHOR = {Lin, Chang Shou},
     TITLE = {On the second eigenfunctions of the {L}aplacian in {${\bf
              R}^2$}},
   JOURNAL = {Comm. Math. Phys.},
  FJOURNAL = {Communications in Mathematical Physics},
    VOLUME = {111},
      YEAR = {1987},
    NUMBER = {2},
     PAGES = {161--166},
      ISSN = {0010-3616,1432-0916},
   MRCLASS = {35P05 (35J05)},
  MRNUMBER = {899848},
MRREVIEWER = {J.-P.\ Gossez},
       URL = {http://projecteuclid.org/euclid.cmp/1104159536},
}

@book {Gradshteyn2007,
    AUTHOR = {Gradshteyn, I. S. and Ryzhik, I. M.},
     TITLE = {Table of integrals, series, and products},
   EDITION = {Seventh},
      NOTE = {Translated from the Russian,
              Translation edited and with a preface by Alan Jeffrey and
              Daniel Zwillinger,
              With one CD-ROM (Windows, Macintosh and UNIX)},
 PUBLISHER = {Elsevier/Academic Press, Amsterdam},
      YEAR = {2007},
     PAGES = {xlviii+1171},
      ISBN = {978-0-12-373637-6; 0-12-373637-4},
   MRCLASS = {00A22 (33-00 65-00 65A05)},
  MRNUMBER = {2360010},
}

@inproceedings{Henrot2018,
  title={Shape Variation and Optimization: A Geometrical Analysis},
  author={Antoine Henrot and Michel Pierre},
  year={2018},
  url={https://api.semanticscholar.org/CorpusID:126344585}
}

@book{Groemer_1996, place={Cambridge}, series={Encyclopedia of Mathematics and its Applications}, title={Geometric Applications of Fourier Series and Spherical Harmonics}, publisher={Cambridge University Press}, author={Groemer, Helmut}, year={1996}, collection={Encyclopedia of Mathematics and its Applications}}

@article {alm1,
    AUTHOR = {Abatangelo, Laura and L\'ena, Corentin and Musolino, Paolo},
     TITLE = {Ramification of multiple eigenvalues for the
              {D}irichlet-{L}aplacian in perforated domains},
   JOURNAL = {J. Funct. Anal.},
  FJOURNAL = {Journal of Functional Analysis},
    VOLUME = {283},
      YEAR = {2022},
    NUMBER = {12},
     PAGES = {Paper No. 109718, 50},
      ISSN = {0022-1236,1096-0783},
   MRCLASS = {35P20 (31B10 31C15 35B25 35C20)},
  MRNUMBER = {4489278},
       DOI = {10.1016/j.jfa.2022.109718},
       URL = {https://doi.org/10.1016/j.jfa.2022.109718},
}

@article {alm2,
    AUTHOR = {Abatangelo, Laura and L\'ena, Corentin and Musolino, Paolo},
     TITLE = {Asymptotic behavior of generalized capacities with
              applications to eigenvalue perturbations: the higher
              dimensional case},
   JOURNAL = {Nonlinear Anal.},
  FJOURNAL = {Nonlinear Analysis. Theory, Methods \& Applications. An
              International Multidisciplinary Journal},
    VOLUME = {238},
      YEAR = {2024},
     PAGES = {Paper No. 113391, 34},
      ISSN = {0362-546X,1873-5215},
   MRCLASS = {35P15 (31B10 31C15 35B25 35C20)},
  MRNUMBER = {4648492},
MRREVIEWER = {Dmitriy\ Karp},
       DOI = {10.1016/j.na.2023.113391},
       URL = {https://doi.org/10.1016/j.na.2023.113391},
}

@article {afhl,
    AUTHOR = {Abatangelo, Laura and Felli, Veronica and Hillairet, Luc and
              L\'ena, Corentin},
     TITLE = {Spectral stability under removal of small capacity sets and
              applications to {A}haronov-{B}ohm operators},
   JOURNAL = {J. Spectr. Theory},
  FJOURNAL = {Journal of Spectral Theory},
    VOLUME = {9},
      YEAR = {2019},
    NUMBER = {2},
     PAGES = {379--427},
      ISSN = {1664-039X,1664-0403},
   MRCLASS = {35P20 (31C15 35J10 35P15)},
  MRNUMBER = {3950657},
MRREVIEWER = {Rodica\ Luca},
       DOI = {10.4171/JST/251},
       URL = {https://doi.org/10.4171/JST/251},
}

@book {at,
    AUTHOR = {Abate, Marco and Tovena, Francesca},
     TITLE = {Curves and surfaces},
    SERIES = {Unitext},
    VOLUME = {55},
      NOTE = {Translated from the 2006 Italian original by Daniele A.
              Gewurz},
 PUBLISHER = {Springer, Milan},
      YEAR = {2012},
     PAGES = {xiv+390},
      ISBN = {978-88-470-1940-9; 978-88-470-1941-6},
   MRCLASS = {53-01 (53A04 53A05)},
  MRNUMBER = {2964051},
MRREVIEWER = {Andrew Bucki},
       DOI = {10.1007/978-88-470-1941-6},
       URL = {https://doi.org/10.1007/978-88-470-1941-6},
}

@InCollection{b,
  author     = {Bandle, Catherine},
  title      = {Asymptotic behaviour of large solutions of quasilinear elliptic problems},
  year       = {2003},
  note       = {Special issue dedicated to Lawrence E. Payne},
  number     = {5},
  pages      = {731--738},
  volume     = {54},
  doi        = {10.1007/s00033-003-3207-0},
  fjournal   = {Zeitschrift f\"{u}r Angewandte Mathematik und Physik. ZAMP. Journal of Applied Mathematics and Physics. Journal de Math\'{e}matiques et de Physique Appliqu\'{e}es},
  issn       = {0044-2275},
  journal    = {Z. Angew. Math. Phys.},
  mrclass    = {35J60 (35B40)},
  mrnumber   = {2019176},
  mrreviewer = {Pierpaolo Esposito},
  url        = {https://doi.org/10.1007/s00033-003-3207-0},
}

@article{fl,
  title     = {Payne's nodal line conjecture fails on doubly-connected planar domains},
  author    = {Freitas, P. and Leylekian, Romeo},
  journal   = {arXiv},
  year      = {2025},
  eprint    = {CodiceArXiv},
  archivePrefix = {arXiv:2510.24436},
  primaryClass = {math.AP}
}

@book {f,
    AUTHOR = {Friedman, Avner},
     TITLE = {Variational principles and free-boundary problems},
    SERIES = {A Wiley-Interscience Publication},
 PUBLISHER = {John Wiley \& Sons, Inc., New York},
      YEAR = {1982},
     PAGES = {ix+710},
      ISBN = {0-471-86849-3},
   MRCLASS = {35R35 (35J85 35K85 35Q20 49A29 76D25 76S05)},
  MRNUMBER = {679313},
MRREVIEWER = {Emmanuele di Benedetto},
}

@book {gt,
    AUTHOR = {Gilbarg, David and Trudinger, Neil S.},
     TITLE = {Elliptic partial differential equations of second order},
    SERIES = {Classics in Mathematics},
      NOTE = {Reprint of the 1998 edition},
 PUBLISHER = {Springer-Verlag, Berlin},
      YEAR = {2001},
     PAGES = {xiv+517},
      ISBN = {3-540-41160-7},
   MRCLASS = {35-02 (35Jxx)},
  MRNUMBER = {1814364},
}

@Article{h,
  author    = {Han, Zheng Chao},
  journal   = {Annales de l'I.H.P. Analyse non linéaire},
  title     = {Asymptotic approach to singular solutions for nonlinear elliptic equations involving critical Sobolev exponent},
  year      = {1991},
  number    = {2},
  pages     = {159-174},
  volume    = {8},
  language  = {eng},
  publisher = {Gauthier-Villars},
  url       = {http://eudml.org/doc/78248},
}

@Article{i,
  author   = {Iacopetti, Alessandro},
  journal  = {Annali di Matematica Pura ed Applicata},
  title    = {Asymptotic analysis for radial sign-changing solutions of the Brezis-Nirenberg problem},
  year     = {2015},
  issn     = {1618-1891},
  number   = {6},
  pages    = {1649--1682},
  volume   = {194},
  refid    = {Iacopetti2015},
  url      = {https://doi.org/10.1007/s10231-014-0438-y},
}

@book {k,
    AUTHOR = {Kawohl, Bernhard},
     TITLE = {Rearrangements and convexity of level sets in {PDE}},
    SERIES = {Lecture Notes in Mathematics},
    VOLUME = {1150},
 PUBLISHER = {Springer-Verlag, Berlin},
      YEAR = {1985},
     PAGES = {iv+136},
      ISBN = {3-540-15693-3},
   MRCLASS = {35-02 (35B50 35J60 49A50)},
  MRNUMBER = {810619},
MRREVIEWER = {Michael Wiegner},
       DOI = {10.1007/BFb0075060},
       URL = {https://doi.org/10.1007/BFb0075060},
}

@article {l,
    AUTHOR = {Lin, Chang Shou},
     TITLE = {Uniqueness of least energy solutions to a semilinear elliptic
              equation in {${\bf R}^2$}},
   JOURNAL = {Manuscripta Math.},
  FJOURNAL = {Manuscripta Mathematica},
    VOLUME = {84},
      YEAR = {1994},
    NUMBER = {1},
     PAGES = {13--19},
      ISSN = {0025-2611},
   MRCLASS = {35J60},
  MRNUMBER = {1283323},
       DOI = {10.1007/BF02567439},
       URL = {https://doi.org/10.1007/BF02567439},
}

@article {O,
    AUTHOR = {Ozawa, Shin},
     TITLE = {Singular variation of domains and eigenvalues of the
              {L}aplacian},
   JOURNAL = {Duke Math. J.},
  FJOURNAL = {Duke Mathematical Journal},
    VOLUME = {48},
      YEAR = {1981},
    NUMBER = {4},
     PAGES = {767--778},
      ISSN = {0012-7094},
   MRCLASS = {35P20 (47A55 47F05)},
  MRNUMBER = {782576},
MRREVIEWER = {G. V. Rozenblum},
       DOI = {10.1215/S0012-7094-81-04842-0},
       URL = {https://doi.org/10.1215/S0012-7094-81-04842-0},
}

@Article{p,
  author     = {Pacella, Filomena},
  journal    = {J. Funct. Anal.},
  title      = {Symmetry results for solutions of semilinear elliptic equations with convex nonlinearities},
  year       = {2002},
  issn       = {0022-1236},
  number     = {1},
  pages      = {271--282},
  volume     = {192},
  doi        = {10.1006/jfan.2001.3901},
  fjournal   = {Journal of Functional Analysis},
  mrclass    = {35J60 (35A30 35B05 35J25 58J70)},
  mrnumber   = {1918496},
  mrreviewer = {Stanislaus Maier-Paape},
  url        = {https://doi.org/10.1006/jfan.2001.3901},
}

@article {s,
    AUTHOR = {Shih, Ying},
     TITLE = {A counterexample to the convexity property of the first
              eigenfunction on a convex domain of negative curvature},
   JOURNAL = {Comm. Partial Differential Equations},
  FJOURNAL = {Communications in Partial Differential Equations},
    VOLUME = {14},
      YEAR = {1989},
    NUMBER = {7},
     PAGES = {867--876},
      ISSN = {0360-5302,1532-4133},
   MRCLASS = {58G25 (35P99)},
  MRNUMBER = {1003017},
MRREVIEWER = {P.\ G\"{u}nther},
       DOI = {10.1080/03605308908820634},
       URL = {https://doi.org/10.1080/03605308908820634},
}

@Article{t,
  author     = {Temam, R.},
  journal    = {Comm. Partial Differential Equations},
  title      = {Remarks on a free boundary value problem arising in plasma physics},
  year       = {1977},
  issn       = {0360-5302},
  number     = {6},
  pages      = {563--585},
  volume     = {2},
  doi        = {10.1080/03605307708820039},
  fjournal   = {Communications in Partial Differential Equations},
  mrclass    = {35J65 (82.35)},
  mrnumber   = {0602544},
  mrreviewer = {Pierre-Louis Lions},
  url        = {https://doi.org/10.1080/03605307708820039},
}

@article {w,
    AUTHOR = {Wang, Feng-Yu},
     TITLE = {On estimation of the {D}irichlet spectral gap},
   JOURNAL = {Arch. Math. (Basel)},
  FJOURNAL = {Archiv der Mathematik},
    VOLUME = {75},
      YEAR = {2000},
    NUMBER = {6},
     PAGES = {450--455},
      ISSN = {0003-889X,1420-8938},
   MRCLASS = {35P15 (35J05)},
  MRNUMBER = {1799430},
MRREVIEWER = {Robert\ G.\ Smits},
       DOI = {10.1007/s000130050528},
       URL = {https://doi.org/10.1007/s000130050528},
}
\end{document}